\documentclass[12pt]{article}
\usepackage{amsthm,amsfonts,amssymb,amsmath}
\usepackage{stmaryrd}

\usepackage{cite,hyperref}
\usepackage{epsfig}
\usepackage{url}
\usepackage{xcolor,tikz}
\usetikzlibrary{matrix}
\usepackage{multicol,graphicx}
\usepackage{fullpage}
\usepackage{bbm}

\usepackage{pdflscape}

\usepackage[shortlabels]{enumitem}

\newcommand\scalemath[2]{\scalebox{#1}{\mbox{\ensuremath{\displaystyle #2}}}}

\numberwithin{equation}{section}

\setlist{leftmargin=3\parindent,labelindent=3\parindent}
\setlist[enumerate]{%
 leftmargin=3\parindent,%
 align=left,%
 labelwidth=3\parindent,%
 labelsep=0pt%
}
\setlist[enumerate,1]{% 
 label={\normalfont (\thesection.\arabic{equation})}, ref={\normalfont \thesection.\arabic{equation}},
 resume%
}

\newtheorem{thm}[equation]{Theorem}
\newtheorem{cor}[equation]{Corollary}
\newtheorem{lem}[equation]{Lemma}
\newtheorem{prop}[equation]{Proposition}

\newtheorem{fact}[equation]{Fact}
\newtheorem{conj}[equation]{Conjecture}

\newtheorem{prob}[equation]{Problem}

\theoremstyle{definition}
\newtheorem{defn}[equation]{Definition}

\newtheorem{rem}[equation]{Remark}
\newtheorem{obs}[equation]{Observation}

\newtheorem*{ack}{Acknowledgements}
\newtheorem*{ai}{AI Declaration}

\theoremstyle{remark}

\title{Off-Diagonal Ramsey Multiplicity}
\author{Elena Moss\thanks{School of Computing Science, University of Glasgow, Glasgow, UK. The main research leading to this paper was completed while the first author was an undergraduate student at the University of Victoria. Research supported by an NSERC Undergraduate Student Research Award (USRA). E-mail: \texttt{e.moss.1@research.gla.ac.uk}.} \and Jonathan A. Noel\thanks{Department of Mathematics and Statistics, University of Victoria, Victoria, B.C., Canada. Research supported by NSERC Discovery Grant RGPIN-2021-02460, NSERC Early Career Supplement DGECR-2021-00024 and a Start-Up Grant from the University of Victoria. E-mail: \texttt{noelj@uvic.ca}.}}

\DeclareTextCompositeCommand{\v}{OT1}{l}{l\nobreak\hspace{-.1em}'}
\DeclareTextCompositeCommand{\v}{OT1}{t}{t\nobreak\hspace{-.1em}'\nobreak\hspace{-.15em}}

\DeclareMathOperator{\inj}{inj}
\DeclareMathOperator{\ind}{ind}
\DeclareMathOperator{\aut}{aut}

\pgfdeclarelayer{edgelayer}
\pgfdeclarelayer{nodelayer}
\pgfsetlayers{edgelayer,nodelayer,main}

\tikzstyle{none}=[inner sep=0pt]
\definecolor{hexcolor0xf81e1c}{rgb}{0.973,0.118,0.110}
\definecolor{hexcolor0x3c00ff}{rgb}{0.235,0.000,1.000}

\tikzstyle{root}=[rectangle, fill=white,draw=black, scale=0.70]
\tikzstyle{vertex}=[circle, fill=white,draw=black, scale=0.55]
\tikzstyle{root_small}=[rectangle, fill=white,draw=black, scale=0.350]
\tikzstyle{vertex_small}=[circle, fill=white,draw=black, scale=0.35]
\tikzstyle{dashededge}=[draw=black, densely dotted]
\tikzstyle{edge}=[draw=black]

\newcommand{\ZeroZero}{
\begin{tikzpicture}[scale = 0.75]
	\begin{pgfonlayer}{nodelayer}
		\node [style=root] (2) at (330:1) {};
		\node [style=root] (1) at (90:1) {};
		\node [style=root] (0) at (210:1) {};
		\node [style=vertex] (3) at (0,0) {};
	\end{pgfonlayer}
	\begin{pgfonlayer}{edgelayer}
		\draw [style=dashededge](0.center) to (1.center);
		\draw [style=dashededge](2.center) to (1.center);
		\draw [style=dashededge](0.center) to (2.center);
		\draw [style=dashededge](3.center) to (0.center);
		\draw [style=dashededge](3.center) to (1.center);
		\draw [style=dashededge](3.center) to (2.center);
	\end{pgfonlayer}
\end{tikzpicture}}

\newcommand{\ZeroOne}{
\begin{tikzpicture}[scale = 0.75]
	\begin{pgfonlayer}{nodelayer}
		\node [style=root] (2) at (330:1) {};
		\node [style=root] (1) at (90:1) {};
		\node [style=root] (0) at (210:1) {};
		\node [style=vertex] (3) at (0,0) {};
	\end{pgfonlayer}
	\begin{pgfonlayer}{edgelayer}
		\draw [style=dashededge](0.center) to (1.center);
		\draw [style=dashededge](2.center) to (1.center);
		\draw [style=dashededge](0.center) to (2.center);
		\draw (3.center) to (0.center);
		\draw [style=dashededge](3.center) to (1.center);
		\draw [style=dashededge](3.center) to (2.center);
	\end{pgfonlayer}
\end{tikzpicture}}

\newcommand{\ZeroTwo}{
\begin{tikzpicture}[scale = 0.75]
	\begin{pgfonlayer}{nodelayer}
		\node [style=root] (2) at (330:1) {};
		\node [style=root] (1) at (90:1) {};
		\node [style=root] (0) at (210:1) {};
		\node [style=vertex] (3) at (0,0) {};
	\end{pgfonlayer}
	\begin{pgfonlayer}{edgelayer}
		\draw [style=dashededge](0.center) to (1.center);
		\draw [style=dashededge](2.center) to (1.center);
		\draw [style=dashededge](0.center) to (2.center);
		\draw [style=dashededge](3.center) to (0.center);
		\draw (3.center) to (1.center);
		\draw [style=dashededge](3.center) to (2.center);
	\end{pgfonlayer}
\end{tikzpicture}}

\newcommand{\ZeroThree}{
\begin{tikzpicture}[scale = 0.75]
	\begin{pgfonlayer}{nodelayer}
		\node [style=root] (2) at (330:1) {};
		\node [style=root] (1) at (90:1) {};
		\node [style=root] (0) at (210:1) {};
		\node [style=vertex] (3) at (0,0) {};
	\end{pgfonlayer}
	\begin{pgfonlayer}{edgelayer}
		\draw [style=dashededge](0.center) to (1.center);
		\draw [style=dashededge](2.center) to (1.center);
		\draw [style=dashededge](0.center) to (2.center);
		\draw [style=dashededge](3.center) to (0.center);
		\draw [style=dashededge](3.center) to (1.center);
		\draw (3.center) to (2.center);
	\end{pgfonlayer}
\end{tikzpicture}}

\newcommand{\ZeroFour}{
\begin{tikzpicture}[scale = 0.75]
	\begin{pgfonlayer}{nodelayer}
		\node [style=root] (2) at (330:1) {};
		\node [style=root] (1) at (90:1) {};
		\node [style=root] (0) at (210:1) {};
		\node [style=vertex] (3) at (0,0) {};
	\end{pgfonlayer}
	\begin{pgfonlayer}{edgelayer}
		\draw [style=dashededge](0.center) to (1.center);
		\draw [style=dashededge](2.center) to (1.center);
		\draw [style=dashededge](0.center) to (2.center);
		\draw (3.center) to (0.center);
		\draw (3.center) to (1.center);
		\draw [style=dashededge](3.center) to (2.center);
	\end{pgfonlayer}
\end{tikzpicture}}

\newcommand{\ZeroFive}{
\begin{tikzpicture}[scale = 0.75]
	\begin{pgfonlayer}{nodelayer}
		\node [style=root] (2) at (330:1) {};
		\node [style=root] (1) at (90:1) {};
		\node [style=root] (0) at (210:1) {};
		\node [style=vertex] (3) at (0,0) {};
	\end{pgfonlayer}
	\begin{pgfonlayer}{edgelayer}
		\draw [style=dashededge](0.center) to (1.center);
		\draw [style=dashededge](2.center) to (1.center);
		\draw [style=dashededge](0.center) to (2.center);
		\draw (3.center) to (0.center);
		\draw [style=dashededge](3.center) to (1.center);
		\draw (3.center) to (2.center);
	\end{pgfonlayer}
\end{tikzpicture}}

\newcommand{\ZeroSix}{
\begin{tikzpicture}[scale = 0.75]
	\begin{pgfonlayer}{nodelayer}
		\node [style=root] (2) at (330:1) {};
		\node [style=root] (1) at (90:1) {};
		\node [style=root] (0) at (210:1) {};
		\node [style=vertex] (3) at (0,0) {};
	\end{pgfonlayer}
	\begin{pgfonlayer}{edgelayer}
		\draw [style=dashededge](0.center) to (1.center);
		\draw [style=dashededge](2.center) to (1.center);
		\draw [style=dashededge](0.center) to (2.center);
		\draw [style=dashededge](3.center) to (0.center);
		\draw (3.center) to (1.center);
		\draw (3.center) to (2.center);
	\end{pgfonlayer}
\end{tikzpicture}}

\newcommand{\ZeroSeven}{
\begin{tikzpicture}[scale = 0.75]
	\begin{pgfonlayer}{nodelayer}
		\node [style=root] (2) at (330:1) {};
		\node [style=root] (1) at (90:1) {};
		\node [style=root] (0) at (210:1) {};
		\node [style=vertex] (3) at (0,0) {};
	\end{pgfonlayer}
	\begin{pgfonlayer}{edgelayer}
		\draw [style=dashededge](0.center) to (1.center);
		\draw [style=dashededge](2.center) to (1.center);
		\draw [style=dashededge](0.center) to (2.center);
		\draw (3.center) to (0.center);
		\draw (3.center) to (1.center);
		\draw (3.center) to (2.center);
	\end{pgfonlayer}
\end{tikzpicture}}

\newcommand{\OneZeroZero}{
\begin{tikzpicture}[scale = 0.75]
	\begin{pgfonlayer}{nodelayer}
		\node [style=root] (2) at (330:1) {};
		\node [style=root] (1) at (90:1) {};
		\node [style=root] (0) at (210:1) {};
		\node [style=vertex] (3) at (0,0) {};
	\end{pgfonlayer}
	\begin{pgfonlayer}{edgelayer}
		\draw (0.center) to (1.center);
		\draw [style=dashededge](2.center) to (1.center);
		\draw [style=dashededge](0.center) to (2.center);
		\draw [style=dashededge](3.center) to (0.center);
		\draw [style=dashededge](3.center) to (1.center);
		\draw [style=dashededge](3.center) to (2.center);
	\end{pgfonlayer}
\end{tikzpicture}}

\newcommand{\OneZeroOne}{
\begin{tikzpicture}[scale = 0.75]
	\begin{pgfonlayer}{nodelayer}
		\node [style=root] (2) at (330:1) {};
		\node [style=root] (1) at (90:1) {};
		\node [style=root] (0) at (210:1) {};
		\node [style=vertex] (3) at (0,0) {};
	\end{pgfonlayer}
	\begin{pgfonlayer}{edgelayer}
		\draw (0.center) to (1.center);
		\draw [style=dashededge](2.center) to (1.center);
		\draw [style=dashededge](0.center) to (2.center);
		\draw (3.center) to (0.center);
		\draw [style=dashededge](3.center) to (1.center);
		\draw [style=dashededge](3.center) to (2.center);
	\end{pgfonlayer}
\end{tikzpicture}}

\newcommand{\OneZeroTwo}{
\begin{tikzpicture}[scale = 0.75]
	\begin{pgfonlayer}{nodelayer}
		\node [style=root] (2) at (330:1) {};
		\node [style=root] (1) at (90:1) {};
		\node [style=root] (0) at (210:1) {};
		\node [style=vertex] (3) at (0,0) {};
	\end{pgfonlayer}
	\begin{pgfonlayer}{edgelayer}
		\draw (0.center) to (1.center);
		\draw [style=dashededge](2.center) to (1.center);
		\draw [style=dashededge](0.center) to (2.center);
		\draw [style=dashededge](3.center) to (0.center);
		\draw (3.center) to (1.center);
		\draw [style=dashededge](3.center) to (2.center);
	\end{pgfonlayer}
\end{tikzpicture}}

\newcommand{\OneZeroThree}{
\begin{tikzpicture}[scale = 0.75]
	\begin{pgfonlayer}{nodelayer}
		\node [style=root] (2) at (330:1) {};
		\node [style=root] (1) at (90:1) {};
		\node [style=root] (0) at (210:1) {};
		\node [style=vertex] (3) at (0,0) {};
	\end{pgfonlayer}
	\begin{pgfonlayer}{edgelayer}
		\draw (0.center) to (1.center);
		\draw [style=dashededge](2.center) to (1.center);
		\draw [style=dashededge](0.center) to (2.center);
		\draw [style=dashededge](3.center) to (0.center);
		\draw [style=dashededge](3.center) to (1.center);
		\draw (3.center) to (2.center);
	\end{pgfonlayer}
\end{tikzpicture}}

\newcommand{\OneZeroFour}{
\begin{tikzpicture}[scale = 0.75]
	\begin{pgfonlayer}{nodelayer}
		\node [style=root] (2) at (330:1) {};
		\node [style=root] (1) at (90:1) {};
		\node [style=root] (0) at (210:1) {};
		\node [style=vertex] (3) at (0,0) {};
	\end{pgfonlayer}
	\begin{pgfonlayer}{edgelayer}
		\draw (0.center) to (1.center);
		\draw [style=dashededge](2.center) to (1.center);
		\draw [style=dashededge](0.center) to (2.center);
		\draw (3.center) to (0.center);
		\draw (3.center) to (1.center);
		\draw [style=dashededge](3.center) to (2.center);
	\end{pgfonlayer}
\end{tikzpicture}}

\newcommand{\OneZeroFive}{
\begin{tikzpicture}[scale = 0.75]
	\begin{pgfonlayer}{nodelayer}
		\node [style=root] (2) at (330:1) {};
		\node [style=root] (1) at (90:1) {};
		\node [style=root] (0) at (210:1) {};
		\node [style=vertex] (3) at (0,0) {};
	\end{pgfonlayer}
	\begin{pgfonlayer}{edgelayer}
		\draw (0.center) to (1.center);
		\draw [style=dashededge](2.center) to (1.center);
		\draw [style=dashededge](0.center) to (2.center);
		\draw (3.center) to (0.center);
		\draw [style=dashededge](3.center) to (1.center);
		\draw (3.center) to (2.center);
	\end{pgfonlayer}
\end{tikzpicture}}

\newcommand{\OneZeroSix}{
\begin{tikzpicture}[scale = 0.75]
	\begin{pgfonlayer}{nodelayer}
		\node [style=root] (2) at (330:1) {};
		\node [style=root] (1) at (90:1) {};
		\node [style=root] (0) at (210:1) {};
		\node [style=vertex] (3) at (0,0) {};
	\end{pgfonlayer}
	\begin{pgfonlayer}{edgelayer}
		\draw (0.center) to (1.center);
		\draw [style=dashededge](2.center) to (1.center);
		\draw [style=dashededge](0.center) to (2.center);
		\draw [style=dashededge](3.center) to (0.center);
		\draw (3.center) to (1.center);
		\draw (3.center) to (2.center);
	\end{pgfonlayer}
\end{tikzpicture}}

\newcommand{\OneZeroSeven}{
\begin{tikzpicture}[scale = 0.75]
	\begin{pgfonlayer}{nodelayer}
		\node [style=root] (2) at (330:1) {};
		\node [style=root] (1) at (90:1) {};
		\node [style=root] (0) at (210:1) {};
		\node [style=vertex] (3) at (0,0) {};
	\end{pgfonlayer}
	\begin{pgfonlayer}{edgelayer}
		\draw (0.center) to (1.center);
		\draw [style=dashededge](2.center) to (1.center);
		\draw [style=dashededge](0.center) to (2.center);
		\draw (3.center) to (0.center);
		\draw (3.center) to (1.center);
		\draw (3.center) to (2.center);
	\end{pgfonlayer}
\end{tikzpicture}}

\newcommand{\TwoZeroZero}{
\begin{tikzpicture}[scale = 0.75]
	\begin{pgfonlayer}{nodelayer}
		\node [style=root] (2) at (330:1) {};
		\node [style=root] (1) at (90:1) {};
		\node [style=root] (0) at (210:1) {};
		\node [style=vertex] (3) at (0,0) {};
	\end{pgfonlayer}
	\begin{pgfonlayer}{edgelayer}
		\draw (0.center) to (1.center);
		\draw [style=dashededge](2.center) to (1.center);
		\draw (0.center) to (2.center);
		\draw [style=dashededge](3.center) to (0.center);
		\draw [style=dashededge](3.center) to (1.center);
		\draw [style=dashededge](3.center) to (2.center);
	\end{pgfonlayer}
\end{tikzpicture}}

\newcommand{\TwoZeroOne}{
\begin{tikzpicture}[scale = 0.75]
	\begin{pgfonlayer}{nodelayer}
		\node [style=root] (2) at (330:1) {};
		\node [style=root] (1) at (90:1) {};
		\node [style=root] (0) at (210:1) {};
		\node [style=vertex] (3) at (0,0) {};
	\end{pgfonlayer}
	\begin{pgfonlayer}{edgelayer}
		\draw (0.center) to (1.center);
		\draw [style=dashededge](2.center) to (1.center);
		\draw (0.center) to (2.center);
		\draw (3.center) to (0.center);
		\draw [style=dashededge](3.center) to (1.center);
		\draw [style=dashededge](3.center) to (2.center);
	\end{pgfonlayer}
\end{tikzpicture}}

\newcommand{\TwoZeroTwo}{
\begin{tikzpicture}[scale = 0.75]
	\begin{pgfonlayer}{nodelayer}
		\node [style=root] (2) at (330:1) {};
		\node [style=root] (1) at (90:1) {};
		\node [style=root] (0) at (210:1) {};
		\node [style=vertex] (3) at (0,0) {};
	\end{pgfonlayer}
	\begin{pgfonlayer}{edgelayer}
		\draw (0.center) to (1.center);
		\draw [style=dashededge](2.center) to (1.center);
		\draw (0.center) to (2.center);
		\draw [style=dashededge](3.center) to (0.center);
		\draw (3.center) to (1.center);
		\draw [style=dashededge](3.center) to (2.center);
	\end{pgfonlayer}
\end{tikzpicture}}

\newcommand{\TwoZeroThree}{
\begin{tikzpicture}[scale = 0.75]
	\begin{pgfonlayer}{nodelayer}
		\node [style=root] (2) at (330:1) {};
		\node [style=root] (1) at (90:1) {};
		\node [style=root] (0) at (210:1) {};
		\node [style=vertex] (3) at (0,0) {};
	\end{pgfonlayer}
	\begin{pgfonlayer}{edgelayer}
		\draw (0.center) to (1.center);
		\draw [style=dashededge](2.center) to (1.center);
		\draw (0.center) to (2.center);
		\draw [style=dashededge](3.center) to (0.center);
		\draw [style=dashededge](3.center) to (1.center);
		\draw (3.center) to (2.center);
	\end{pgfonlayer}
\end{tikzpicture}}

\newcommand{\TwoZeroFour}{
\begin{tikzpicture}[scale = 0.75]
	\begin{pgfonlayer}{nodelayer}
		\node [style=root] (2) at (330:1) {};
		\node [style=root] (1) at (90:1) {};
		\node [style=root] (0) at (210:1) {};
		\node [style=vertex] (3) at (0,0) {};
	\end{pgfonlayer}
	\begin{pgfonlayer}{edgelayer}
		\draw (0.center) to (1.center);
		\draw [style=dashededge](2.center) to (1.center);
		\draw (0.center) to (2.center);
		\draw (3.center) to (0.center);
		\draw (3.center) to (1.center);
		\draw [style=dashededge](3.center) to (2.center);
	\end{pgfonlayer}
\end{tikzpicture}}

\newcommand{\TwoZeroFive}{
\begin{tikzpicture}[scale = 0.75]
	\begin{pgfonlayer}{nodelayer}
		\node [style=root] (2) at (330:1) {};
		\node [style=root] (1) at (90:1) {};
		\node [style=root] (0) at (210:1) {};
		\node [style=vertex] (3) at (0,0) {};
	\end{pgfonlayer}
	\begin{pgfonlayer}{edgelayer}
		\draw (0.center) to (1.center);
		\draw [style=dashededge](2.center) to (1.center);
		\draw (0.center) to (2.center);
		\draw (3.center) to (0.center);
		\draw [style=dashededge](3.center) to (1.center);
		\draw (3.center) to (2.center);
	\end{pgfonlayer}
\end{tikzpicture}}

\newcommand{\TwoZeroSix}{
\begin{tikzpicture}[scale = 0.75]
	\begin{pgfonlayer}{nodelayer}
		\node [style=root] (2) at (330:1) {};
		\node [style=root] (1) at (90:1) {};
		\node [style=root] (0) at (210:1) {};
		\node [style=vertex] (3) at (0,0) {};
	\end{pgfonlayer}
	\begin{pgfonlayer}{edgelayer}
		\draw (0.center) to (1.center);
		\draw [style=dashededge](2.center) to (1.center);
		\draw (0.center) to (2.center);
		\draw [style=dashededge](3.center) to (0.center);
		\draw (3.center) to (1.center);
		\draw (3.center) to (2.center);
	\end{pgfonlayer}
\end{tikzpicture}}

\newcommand{\TwoZeroSeven}{
\begin{tikzpicture}[scale = 0.75]
	\begin{pgfonlayer}{nodelayer}
		\node [style=root] (2) at (330:1) {};
		\node [style=root] (1) at (90:1) {};
		\node [style=root] (0) at (210:1) {};
		\node [style=vertex] (3) at (0,0) {};
	\end{pgfonlayer}
	\begin{pgfonlayer}{edgelayer}
		\draw (0.center) to (1.center);
		\draw [style=dashededge](2.center) to (1.center);
		\draw (0.center) to (2.center);
		\draw (3.center) to (0.center);
		\draw (3.center) to (1.center);
		\draw (3.center) to (2.center);
	\end{pgfonlayer}
\end{tikzpicture}}

\newcommand{\ThreeZeroZero}{
\begin{tikzpicture}[scale = 0.75]
	\begin{pgfonlayer}{nodelayer}
		\node [style=root] (2) at (330:1) {};
		\node [style=root] (1) at (90:1) {};
		\node [style=root] (0) at (210:1) {};
		\node [style=vertex] (3) at (0,0) {};
	\end{pgfonlayer}
	\begin{pgfonlayer}{edgelayer}
		\draw (0.center) to (1.center);
		\draw (2.center) to (1.center);
		\draw (0.center) to (2.center);
		\draw [style=dashededge](3.center) to (0.center);
		\draw [style=dashededge](3.center) to (1.center);
		\draw [style=dashededge](3.center) to (2.center);
	\end{pgfonlayer}
\end{tikzpicture}}

\newcommand{\ThreeZeroOne}{
\begin{tikzpicture}[scale = 0.75]
	\begin{pgfonlayer}{nodelayer}
		\node [style=root] (2) at (330:1) {};
		\node [style=root] (1) at (90:1) {};
		\node [style=root] (0) at (210:1) {};
		\node [style=vertex] (3) at (0,0) {};
	\end{pgfonlayer}
	\begin{pgfonlayer}{edgelayer}
		\draw (0.center) to (1.center);
		\draw (2.center) to (1.center);
		\draw (0.center) to (2.center);
		\draw (3.center) to (0.center);
		\draw [style=dashededge](3.center) to (1.center);
		\draw [style=dashededge](3.center) to (2.center);
	\end{pgfonlayer}
\end{tikzpicture}}

\newcommand{\ThreeZeroTwo}{
\begin{tikzpicture}[scale = 0.75]
	\begin{pgfonlayer}{nodelayer}
		\node [style=root] (2) at (330:1) {};
		\node [style=root] (1) at (90:1) {};
		\node [style=root] (0) at (210:1) {};
		\node [style=vertex] (3) at (0,0) {};
	\end{pgfonlayer}
	\begin{pgfonlayer}{edgelayer}
		\draw (0.center) to (1.center);
		\draw (2.center) to (1.center);
		\draw (0.center) to (2.center);
		\draw [style=dashededge](3.center) to (0.center);
		\draw (3.center) to (1.center);
		\draw [style=dashededge](3.center) to (2.center);
	\end{pgfonlayer}
\end{tikzpicture}}

\newcommand{\ThreeZeroThree}{
\begin{tikzpicture}[scale = 0.75]
	\begin{pgfonlayer}{nodelayer}
		\node [style=root] (2) at (330:1) {};
		\node [style=root] (1) at (90:1) {};
		\node [style=root] (0) at (210:1) {};
		\node [style=vertex] (3) at (0,0) {};
	\end{pgfonlayer}
	\begin{pgfonlayer}{edgelayer}
		\draw (0.center) to (1.center);
		\draw (2.center) to (1.center);
		\draw (0.center) to (2.center);
		\draw [style=dashededge](3.center) to (0.center);
		\draw [style=dashededge](3.center) to (1.center);
		\draw (3.center) to (2.center);
	\end{pgfonlayer}
\end{tikzpicture}}

\newcommand{\ThreeZeroFour}{
\begin{tikzpicture}[scale = 0.75]
	\begin{pgfonlayer}{nodelayer}
		\node [style=root] (2) at (330:1) {};
		\node [style=root] (1) at (90:1) {};
		\node [style=root] (0) at (210:1) {};
		\node [style=vertex] (3) at (0,0) {};
	\end{pgfonlayer}
	\begin{pgfonlayer}{edgelayer}
		\draw (0.center) to (1.center);
		\draw (2.center) to (1.center);
		\draw (0.center) to (2.center);
		\draw (3.center) to (0.center);
		\draw (3.center) to (1.center);
		\draw [style=dashededge](3.center) to (2.center);
	\end{pgfonlayer}
\end{tikzpicture}}

\newcommand{\ThreeZeroFive}{
\begin{tikzpicture}[scale = 0.75]
	\begin{pgfonlayer}{nodelayer}
		\node [style=root] (2) at (330:1) {};
		\node [style=root] (1) at (90:1) {};
		\node [style=root] (0) at (210:1) {};
		\node [style=vertex] (3) at (0,0) {};
	\end{pgfonlayer}
	\begin{pgfonlayer}{edgelayer}
		\draw (0.center) to (1.center);
		\draw (2.center) to (1.center);
		\draw (0.center) to (2.center);
		\draw (3.center) to (0.center);
		\draw [style=dashededge](3.center) to (1.center);
		\draw (3.center) to (2.center);
	\end{pgfonlayer}
\end{tikzpicture}}

\newcommand{\ThreeZeroSix}{
\begin{tikzpicture}[scale = 0.75]
	\begin{pgfonlayer}{nodelayer}
		\node [style=root] (2) at (330:1) {};
		\node [style=root] (1) at (90:1) {};
		\node [style=root] (0) at (210:1) {};
		\node [style=vertex] (3) at (0,0) {};
	\end{pgfonlayer}
	\begin{pgfonlayer}{edgelayer}
		\draw (0.center) to (1.center);
		\draw (2.center) to (1.center);
		\draw (0.center) to (2.center);
		\draw [style=dashededge](3.center) to (0.center);
		\draw (3.center) to (1.center);
		\draw (3.center) to (2.center);
	\end{pgfonlayer}
\end{tikzpicture}}

\newcommand{\ThreeZeroSeven}{
\begin{tikzpicture}[scale = 0.75]
	\begin{pgfonlayer}{nodelayer}
		\node [style=root] (2) at (330:1) {};
		\node [style=root] (1) at (90:1) {};
		\node [style=root] (0) at (210:1) {};
		\node [style=vertex] (3) at (0,0) {};
	\end{pgfonlayer}
	\begin{pgfonlayer}{edgelayer}
		\draw (0.center) to (1.center);
		\draw (2.center) to (1.center);
		\draw (0.center) to (2.center);
		\draw (3.center) to (0.center);
		\draw (3.center) to (1.center);
		\draw (3.center) to (2.center);
	\end{pgfonlayer}
\end{tikzpicture}}

\newcommand{\ZeroBlue}{
\begin{tikzpicture}[scale = 0.75]
	\begin{pgfonlayer}{nodelayer}
		\node [style=root] (1) at (270:0.75) {};
		\node [style=vertex] (2) at (150:0.75) {};
		\node [style=vertex] (3) at (30:0.75) {};
	\end{pgfonlayer}
	\begin{pgfonlayer}{edgelayer}
		\draw [style=dashededge](1.center) to (2.center);
		\draw [style=dashededge](1.center) to (3.center);
		\draw [style=dashededge](2.center) to (3.center);
	\end{pgfonlayer}
\end{tikzpicture}}

\newcommand{\OneBlue}{
\begin{tikzpicture}[scale = 0.75]
	\begin{pgfonlayer}{nodelayer}
		\node [style=root] (1) at (270:0.75) {};
		\node [style=vertex] (2) at (150:0.75) {};
		\node [style=vertex] (3) at (30:0.75) {};
	\end{pgfonlayer}
	\begin{pgfonlayer}{edgelayer}
		\draw [style=dashededge](1.center) to (2.center);
		\draw (1.center) to (3.center);
		\draw [style=dashededge](2.center) to (3.center);
	\end{pgfonlayer}
\end{tikzpicture}}

\newcommand{\TwoBlue}{
\begin{tikzpicture}[scale = 0.75]
	\begin{pgfonlayer}{nodelayer}
		\node [style=root] (1) at (270:0.75) {};
		\node [style=vertex] (2) at (150:0.75) {};
		\node [style=vertex] (3) at (30:0.75) {};
	\end{pgfonlayer}
	\begin{pgfonlayer}{edgelayer}
		\draw (1.center) to (2.center);
		\draw (1.center) to (3.center);
		\draw [style=dashededge](2.center) to (3.center);
	\end{pgfonlayer}
\end{tikzpicture}}

\newcommand{\ZeroRed}{
\begin{tikzpicture}[scale = 0.75]
	\begin{pgfonlayer}{nodelayer}
		\node [style=root] (1) at (270:0.75) {};
		\node [style=vertex] (2) at (150:0.75) {};
		\node [style=vertex] (3) at (30:0.75) {};
	\end{pgfonlayer}
	\begin{pgfonlayer}{edgelayer}
		\draw [style=dashededge](1.center) to (2.center);
		\draw [style=dashededge](1.center) to (3.center);
		\draw (2.center) to (3.center);
	\end{pgfonlayer}
\end{tikzpicture}}

\newcommand{\OneRed}{
\begin{tikzpicture}[scale = 0.75]
	\begin{pgfonlayer}{nodelayer}
		\node [style=root] (1) at (270:0.75) {};
		\node [style=vertex] (2) at (150:0.75) {};
		\node [style=vertex] (3) at (30:0.75) {};
	\end{pgfonlayer}
	\begin{pgfonlayer}{edgelayer}
		\draw [style=dashededge](1.center) to (2.center);
		\draw (1.center) to (3.center);
		\draw (2.center) to (3.center);
	\end{pgfonlayer}
\end{tikzpicture}}

\newcommand{\TwoRed}{
\begin{tikzpicture}[scale = 0.75]
	\begin{pgfonlayer}{nodelayer}
		\node [style=root] (1) at (270:0.75) {};
		\node [style=vertex] (2) at (150:0.75) {};
		\node [style=vertex] (3) at (30:0.75) {};
	\end{pgfonlayer}
	\begin{pgfonlayer}{edgelayer}
		\draw (1.center) to (2.center);
		\draw (1.center) to (3.center);
		\draw (2.center) to (3.center);
	\end{pgfonlayer}
\end{tikzpicture}}

\newcommand{\JOne}{
\begin{tikzpicture}[scale = 0.75]
	\begin{pgfonlayer}{nodelayer}
		\node [style=vertex] (1) at (90:1.00) {};
		\node [style=vertex] (2) at (30:0.75) {};
		\node [style=vertex] (3) at (330:0.75) {};
		\node [style=vertex] (4) at (150:0.75) {};
		\node [style=vertex] (5) at (210:0.75) {};
	\end{pgfonlayer}
	\begin{pgfonlayer}{edgelayer}
		\draw (1.center) to (2.center);
		\draw (1.center) to (3.center);
		\draw (2.center) to (3.center);
		\draw (1.center) to (4.center);
		\draw (1.center) to (5.center);
		\draw (4.center) to (5.center);
	\end{pgfonlayer}
\end{tikzpicture}}

\newcommand{\JTwo}{
\begin{tikzpicture}[scale = 0.75]
	\begin{pgfonlayer}{nodelayer}
		\node [style=vertex] (1) at (90:1.00) {};
		\node [style=vertex] (2) at (30:0.75) {};
		\node [style=vertex] (3) at (330:0.75) {};
		\node [style=vertex] (4) at (150:0.75) {};
		\node [style=vertex] (5) at (210:0.75) {};
	\end{pgfonlayer}
	\begin{pgfonlayer}{edgelayer}
		\draw (1.center) to (2.center);
		\draw (1.center) to (3.center);
		\draw (2.center) to (3.center);
		\draw (1.center) to (4.center);
		\draw (1.center) to (5.center);
		\draw (4.center) to (5.center);
		\draw (2.center) to (4.center);
	\end{pgfonlayer}
\end{tikzpicture}}

\newcommand{\JThree}{
\begin{tikzpicture}[scale = 0.75]
	\begin{pgfonlayer}{nodelayer}
		\node [style=vertex] (1) at (90:1.00) {};
		\node [style=vertex] (2) at (30:0.75) {};
		\node [style=vertex] (3) at (330:0.75) {};
		\node [style=vertex] (4) at (150:0.75) {};
		\node [style=vertex] (5) at (210:0.75) {};
	\end{pgfonlayer}
	\begin{pgfonlayer}{edgelayer}
		\draw (1.center) to (2.center);
		\draw (1.center) to (3.center);
		\draw (2.center) to (3.center);
		\draw (1.center) to (4.center);
		\draw (1.center) to (5.center);
		\draw (4.center) to (5.center);
		\draw (2.center) to (4.center);
		\draw (3.center) to (5.center);
	\end{pgfonlayer}
\end{tikzpicture}}

\newcommand{\JFour}{
\begin{tikzpicture}[scale = 0.75]
	\begin{pgfonlayer}{nodelayer}
		\node [style=vertex] (1) at (90:1.00) {};
		\node [style=vertex] (2) at (30:0.75) {};
		\node [style=vertex] (3) at (330:0.75) {};
		\node [style=vertex] (4) at (150:0.75) {};
		\node [style=vertex] (5) at (210:0.75) {};
	\end{pgfonlayer}
	\begin{pgfonlayer}{edgelayer}
		\draw (1.center) to (2.center);
		\draw (1.center) to (3.center);
		\draw (2.center) to (3.center);
		\draw (1.center) to (4.center);
		\draw (1.center) to (5.center);
		\draw (4.center) to (5.center);
		\draw (2.center) to (4.center);
		\draw (2.center) to (5.center);
	\end{pgfonlayer}
\end{tikzpicture}}

\newcommand{\JFive}{
\begin{tikzpicture}[scale = 0.75]
	\begin{pgfonlayer}{nodelayer}
		\node [style=vertex] (1) at (90:1.00) {};
		\node [style=vertex] (2) at (30:0.75) {};
		\node [style=vertex] (3) at (330:0.75) {};
		\node [style=vertex] (4) at (150:0.75) {};
		\node [style=vertex] (5) at (210:0.75) {};
	\end{pgfonlayer}
	\begin{pgfonlayer}{edgelayer}
		\draw (1.center) to (2.center);
		\draw (1.center) to (3.center);
		\draw (2.center) to (3.center);
		\draw (1.center) to (4.center);
		\draw (1.center) to (5.center);
		\draw (4.center) to (5.center);
		\draw (2.center) to (4.center);
		\draw (2.center) to (5.center);
		\draw (3.center) to (5.center);
	\end{pgfonlayer}
\end{tikzpicture}}

\newcommand{\JSix}{
\begin{tikzpicture}[scale = 0.75]
	\begin{pgfonlayer}{nodelayer}
		\node [style=vertex] (1) at (90:1.00) {};
		\node [style=vertex] (2) at (30:0.75) {};
		\node [style=vertex] (3) at (330:0.75) {};
		\node [style=vertex] (4) at (150:0.75) {};
		\node [style=vertex] (5) at (210:0.75) {};
	\end{pgfonlayer}
	\begin{pgfonlayer}{edgelayer}
		\draw (1.center) to (2.center);
		\draw (1.center) to (3.center);
		\draw (2.center) to (3.center);
		\draw (1.center) to (4.center);
		\draw (1.center) to (5.center);
		\draw (4.center) to (5.center);
		\draw (2.center) to (4.center);
		\draw (2.center) to (5.center);
		\draw (3.center) to (4.center);
		\draw (3.center) to (5.center);
	\end{pgfonlayer}
\end{tikzpicture}}

\newcommand{\CTwoOne}{
\begin{tikzpicture}[scale = 0.75]
	\begin{pgfonlayer}{nodelayer}
		\node [style=root] (1) at (225:0.75) {};
		\node [style=root] (2) at (315:0.75) {};
		\node [style=vertex] (3) at (45:0.75) {};
		\node [style=vertex] (4) at (135:0.75) {};
	\end{pgfonlayer}
	\begin{pgfonlayer}{edgelayer}
		\draw (1.center) to (2.center);
		\draw [style=dashededge](1.center) to (3.center);
		\draw [style=dashededge](1.center) to (4.center);
		\draw [style=dashededge](2.center) to (3.center);
		\draw [style=dashededge](2.center) to (4.center);
		\draw [style=dashededge](3.center) to (4.center);
	\end{pgfonlayer}
\end{tikzpicture}}

\newcommand{\CTwoTwo}{
\begin{tikzpicture}[scale = 0.75]
	\begin{pgfonlayer}{nodelayer}
		\node [style=root] (1) at (225:0.75) {};
		\node [style=root] (2) at (315:0.75) {};
		\node [style=vertex] (3) at (45:0.75) {};
		\node [style=vertex] (4) at (135:0.75) {};
	\end{pgfonlayer}
	\begin{pgfonlayer}{edgelayer}
		\draw (1.center) to (2.center);
		\draw [style=dashededge](1.center) to (3.center);
		\draw (1.center) to (4.center);
		\draw [style=dashededge](2.center) to (3.center);
		\draw [style=dashededge](2.center) to (4.center);
		\draw [style=dashededge](3.center) to (4.center);
	\end{pgfonlayer}
\end{tikzpicture}}

\newcommand{\CTwoThree}{
\begin{tikzpicture}[scale = 0.75]
	\begin{pgfonlayer}{nodelayer}
		\node [style=root] (1) at (225:0.75) {};
		\node [style=root] (2) at (315:0.75) {};
		\node [style=vertex] (3) at (45:0.75) {};
		\node [style=vertex] (4) at (135:0.75) {};
	\end{pgfonlayer}
	\begin{pgfonlayer}{edgelayer}
		\draw (1.center) to (2.center);
		\draw (1.center) to (3.center);
		\draw (1.center) to (4.center);
		\draw [style=dashededge](2.center) to (3.center);
		\draw [style=dashededge](2.center) to (4.center);
		\draw [style=dashededge](3.center) to (4.center);
	\end{pgfonlayer}
\end{tikzpicture}}

\newcommand{\CTwoFour}{
\begin{tikzpicture}[scale = 0.75]
	\begin{pgfonlayer}{nodelayer}
		\node [style=root] (1) at (225:0.75) {};
		\node [style=root] (2) at (315:0.75) {};
		\node [style=vertex] (3) at (45:0.75) {};
		\node [style=vertex] (4) at (135:0.75) {};
	\end{pgfonlayer}
	\begin{pgfonlayer}{edgelayer}
		\draw (1.center) to (2.center);
		\draw [style=dashededge](1.center) to (3.center);
		\draw (1.center) to (4.center);
		\draw (2.center) to (3.center);
		\draw [style=dashededge](2.center) to (4.center);
		\draw [style=dashededge](3.center) to (4.center);
	\end{pgfonlayer}
\end{tikzpicture}}

\newcommand{\CTwoFive}{
\begin{tikzpicture}[scale = 0.75]
	\begin{pgfonlayer}{nodelayer}
		\node [style=root] (1) at (225:0.75) {};
		\node [style=root] (2) at (315:0.75) {};
		\node [style=vertex] (3) at (45:0.75) {};
		\node [style=vertex] (4) at (135:0.75) {};
	\end{pgfonlayer}
	\begin{pgfonlayer}{edgelayer}
		\draw (1.center) to (2.center);
		\draw [style=dashededge](1.center) to (3.center);
		\draw [style=dashededge](1.center) to (4.center);
		\draw [style=dashededge](2.center) to (3.center);
		\draw (2.center) to (4.center);
		\draw [style=dashededge](3.center) to (4.center);
	\end{pgfonlayer}
\end{tikzpicture}}

\newcommand{\CTwoSix}{
\begin{tikzpicture}[scale = 0.75]
	\begin{pgfonlayer}{nodelayer}
		\node [style=root] (1) at (225:0.75) {};
		\node [style=root] (2) at (315:0.75) {};
		\node [style=vertex] (3) at (45:0.75) {};
		\node [style=vertex] (4) at (135:0.75) {};
	\end{pgfonlayer}
	\begin{pgfonlayer}{edgelayer}
		\draw (1.center) to (2.center);
		\draw [style=dashededge](1.center) to (3.center);
		\draw [style=dashededge](1.center) to (4.center);
		\draw (2.center) to (3.center);
		\draw (2.center) to (4.center);
		\draw [style=dashededge](3.center) to (4.center);
	\end{pgfonlayer}
\end{tikzpicture}}

\newcommand{\CTwoSeven}{
\begin{tikzpicture}[scale = 0.75]
	\begin{pgfonlayer}{nodelayer}
		\node [style=root] (1) at (225:0.75) {};
		\node [style=root] (2) at (315:0.75) {};
		\node [style=vertex] (3) at (45:0.75) {};
		\node [style=vertex] (4) at (135:0.75) {};
	\end{pgfonlayer}
	\begin{pgfonlayer}{edgelayer}
		\draw (1.center) to (2.center);
		\draw [style=dashededge](1.center) to (3.center);
		\draw (1.center) to (4.center);
		\draw [style=dashededge](2.center) to (3.center);
		\draw (2.center) to (4.center);
		\draw [style=dashededge](3.center) to (4.center);
	\end{pgfonlayer}
\end{tikzpicture}}

\newcommand{\CTwoEight}{
\begin{tikzpicture}[scale = 0.75]
	\begin{pgfonlayer}{nodelayer}
		\node [style=root] (1) at (225:0.75) {};
		\node [style=root] (2) at (315:0.75) {};
		\node [style=vertex] (3) at (45:0.75) {};
		\node [style=vertex] (4) at (135:0.75) {};
	\end{pgfonlayer}
	\begin{pgfonlayer}{edgelayer}
		\draw (1.center) to (2.center);
		\draw (1.center) to (3.center);
		\draw (1.center) to (4.center);
		\draw [style=dashededge](2.center) to (3.center);
		\draw (2.center) to (4.center);
		\draw [style=dashededge](3.center) to (4.center);
	\end{pgfonlayer}
\end{tikzpicture}}

\newcommand{\CTwoNine}{
\begin{tikzpicture}[scale = 0.75]
	\begin{pgfonlayer}{nodelayer}
		\node [style=root] (1) at (225:0.75) {};
		\node [style=root] (2) at (315:0.75) {};
		\node [style=vertex] (3) at (45:0.75) {};
		\node [style=vertex] (4) at (135:0.75) {};
	\end{pgfonlayer}
	\begin{pgfonlayer}{edgelayer}
		\draw (1.center) to (2.center);
		\draw [style=dashededge](1.center) to (3.center);
		\draw (1.center) to (4.center);
		\draw (2.center) to (3.center);
		\draw (2.center) to (4.center);
		\draw [style=dashededge](3.center) to (4.center);
	\end{pgfonlayer}
\end{tikzpicture}}

\newcommand{\CTwoTen}{
\begin{tikzpicture}[scale = 0.75]
	\begin{pgfonlayer}{nodelayer}
		\node [style=root] (1) at (225:0.75) {};
		\node [style=root] (2) at (315:0.75) {};
		\node [style=vertex] (3) at (45:0.75) {};
		\node [style=vertex] (4) at (135:0.75) {};
	\end{pgfonlayer}
	\begin{pgfonlayer}{edgelayer}
		\draw (1.center) to (2.center);
		\draw (1.center) to (3.center);
		\draw (1.center) to (4.center);
		\draw (2.center) to (3.center);
		\draw (2.center) to (4.center);
		\draw [style=dashededge](3.center) to (4.center);
	\end{pgfonlayer}
\end{tikzpicture}}

\newcommand{\CTwoEleven}{
\begin{tikzpicture}[scale = 0.75]
	\begin{pgfonlayer}{nodelayer}
		\node [style=root] (1) at (225:0.75) {};
		\node [style=root] (2) at (315:0.75) {};
		\node [style=vertex] (3) at (45:0.75) {};
		\node [style=vertex] (4) at (135:0.75) {};
	\end{pgfonlayer}
	\begin{pgfonlayer}{edgelayer}
		\draw (1.center) to (2.center);
		\draw [style=dashededge](1.center) to (3.center);
		\draw [style=dashededge](1.center) to (4.center);
		\draw [style=dashededge](2.center) to (3.center);
		\draw [style=dashededge](2.center) to (4.center);
		\draw (3.center) to (4.center);
	\end{pgfonlayer}
\end{tikzpicture}}

\newcommand{\CTwoTwelve}{
\begin{tikzpicture}[scale = 0.75]
	\begin{pgfonlayer}{nodelayer}
		\node [style=root] (1) at (225:0.75) {};
		\node [style=root] (2) at (315:0.75) {};
		\node [style=vertex] (3) at (45:0.75) {};
		\node [style=vertex] (4) at (135:0.75) {};
	\end{pgfonlayer}
	\begin{pgfonlayer}{edgelayer}
		\draw (1.center) to (2.center);
		\draw [style=dashededge](1.center) to (3.center);
            \draw (1.center) to (4.center);
		\draw [style=dashededge](2.center) to (3.center);
		\draw [style=dashededge](2.center) to (4.center);
		\draw (3.center) to (4.center);
	\end{pgfonlayer}
\end{tikzpicture}}

\newcommand{\CTwoThirteen}{
\begin{tikzpicture}[scale = 0.75]
	\begin{pgfonlayer}{nodelayer}
		\node [style=root] (1) at (225:0.75) {};
		\node [style=root] (2) at (315:0.75) {};
		\node [style=vertex] (3) at (45:0.75) {};
		\node [style=vertex] (4) at (135:0.75) {};
	\end{pgfonlayer}
	\begin{pgfonlayer}{edgelayer}
		\draw (1.center) to (2.center);
		\draw (1.center) to (3.center);
            \draw (1.center) to (4.center);
		\draw [style=dashededge](2.center) to (3.center);
		\draw [style=dashededge](2.center) to (4.center);
		\draw (3.center) to (4.center);
	\end{pgfonlayer}
\end{tikzpicture}}

\newcommand{\CTwoFourteen}{
\begin{tikzpicture}[scale = 0.75]
	\begin{pgfonlayer}{nodelayer}
		\node [style=root] (1) at (225:0.75) {};
		\node [style=root] (2) at (315:0.75) {};
		\node [style=vertex] (3) at (45:0.75) {};
		\node [style=vertex] (4) at (135:0.75) {};
	\end{pgfonlayer}
	\begin{pgfonlayer}{edgelayer}
		\draw (1.center) to (2.center);
		\draw [style=dashededge](1.center) to (3.center);
            \draw (1.center) to (4.center);
		\draw (2.center) to (3.center);
		\draw [style=dashededge](2.center) to (4.center);
		\draw (3.center) to (4.center);
	\end{pgfonlayer}
\end{tikzpicture}}

\newcommand{\CTwoFifteen}{
\begin{tikzpicture}[scale = 0.75]
	\begin{pgfonlayer}{nodelayer}
		\node [style=root] (1) at (225:0.75) {};
		\node [style=root] (2) at (315:0.75) {};
		\node [style=vertex] (3) at (45:0.75) {};
		\node [style=vertex] (4) at (135:0.75) {};
	\end{pgfonlayer}
	\begin{pgfonlayer}{edgelayer}
		\draw (1.center) to (2.center);
		\draw [style=dashededge](1.center) to (3.center);
            \draw [style=dashededge](1.center) to (4.center);
		\draw [style=dashededge](2.center) to (3.center);
		\draw (2.center) to (4.center);
		\draw (3.center) to (4.center);
	\end{pgfonlayer}
\end{tikzpicture}}

\newcommand{\CTwoSixteen}{
\begin{tikzpicture}[scale = 0.75]
	\begin{pgfonlayer}{nodelayer}
		\node [style=root] (1) at (225:0.75) {};
		\node [style=root] (2) at (315:0.75) {};
		\node [style=vertex] (3) at (45:0.75) {};
		\node [style=vertex] (4) at (135:0.75) {};
	\end{pgfonlayer}
	\begin{pgfonlayer}{edgelayer}
		\draw (1.center) to (2.center);
		\draw [style=dashededge](1.center) to (3.center);
            \draw [style=dashededge](1.center) to (4.center);
		\draw (2.center) to (3.center);
		\draw (2.center) to (4.center);
		\draw (3.center) to (4.center);
	\end{pgfonlayer}
\end{tikzpicture}}

\newcommand{\CTwoSeventeen}{
\begin{tikzpicture}[scale = 0.75]
	\begin{pgfonlayer}{nodelayer}
		\node [style=root] (1) at (225:0.75) {};
		\node [style=root] (2) at (315:0.75) {};
		\node [style=vertex] (3) at (45:0.75) {};
		\node [style=vertex] (4) at (135:0.75) {};
	\end{pgfonlayer}
	\begin{pgfonlayer}{edgelayer}
		\draw (1.center) to (2.center);
		\draw [style=dashededge](1.center) to (3.center);
            \draw (1.center) to (4.center);
		\draw [style=dashededge](2.center) to (3.center);
		\draw (2.center) to (4.center);
		\draw (3.center) to (4.center);
	\end{pgfonlayer}
\end{tikzpicture}}

\newcommand{\CTwoEighteen}{
\begin{tikzpicture}[scale = 0.75]
	\begin{pgfonlayer}{nodelayer}
		\node [style=root] (1) at (225:0.75) {};
		\node [style=root] (2) at (315:0.75) {};
		\node [style=vertex] (3) at (45:0.75) {};
		\node [style=vertex] (4) at (135:0.75) {};
	\end{pgfonlayer}
	\begin{pgfonlayer}{edgelayer}
		\draw (1.center) to (2.center);
		\draw (1.center) to (3.center);
            \draw (1.center) to (4.center);
		\draw [style=dashededge](2.center) to (3.center);
		\draw (2.center) to (4.center);
		\draw (3.center) to (4.center);
	\end{pgfonlayer}
\end{tikzpicture}}

\newcommand{\CTwoNineteen}{
\begin{tikzpicture}[scale = 0.75]
	\begin{pgfonlayer}{nodelayer}
		\node [style=root] (1) at (225:0.75) {};
		\node [style=root] (2) at (315:0.75) {};
		\node [style=vertex] (3) at (45:0.75) {};
		\node [style=vertex] (4) at (135:0.75) {};
	\end{pgfonlayer}
	\begin{pgfonlayer}{edgelayer}
		\draw (1.center) to (2.center);
		\draw [style=dashededge](1.center) to (3.center);
            \draw (1.center) to (4.center);
		\draw (2.center) to (3.center);
		\draw (2.center) to (4.center);
		\draw (3.center) to (4.center);
	\end{pgfonlayer}
\end{tikzpicture}}

\newcommand{\CTwoTwenty}{
\begin{tikzpicture}[scale = 0.75]
	\begin{pgfonlayer}{nodelayer}
		\node [style=root] (1) at (225:0.75) {};
		\node [style=root] (2) at (315:0.75) {};
		\node [style=vertex] (3) at (45:0.75) {};
		\node [style=vertex] (4) at (135:0.75) {};
	\end{pgfonlayer}
	\begin{pgfonlayer}{edgelayer}
		\draw (1.center) to (2.center);
		\draw (1.center) to (3.center);
            \draw (1.center) to (4.center);
		\draw (2.center) to (3.center);
		\draw (2.center) to (4.center);
		\draw (3.center) to (4.center);
	\end{pgfonlayer}
\end{tikzpicture}}

\newcommand{\HOneOne}{
\begin{tikzpicture}[scale = 0.75]
	\begin{pgfonlayer}{nodelayer}
		\node [style=root] (1) at (150:1.00) {};
		\node [style=root] (2) at (210:0.75) {};
            \node [style=root] (3) at (330:0.75) {};
            \node [style=root] (4) at (30:1.00) {};
		\node [style=vertex] (5) at (90:1.00) {};
	\end{pgfonlayer}
	\begin{pgfonlayer}{edgelayer}
		\draw [style=dashededge](1.center) to (2.center);
		\draw [style=dashededge](1.center) to (3.center);
		\draw [style=dashededge](2.center) to (3.center);
		\draw [style=dashededge](1.center) to (4.center);
		\draw [style=dashededge](1.center) to (5.center);
		\draw [style=dashededge](4.center) to (5.center);
		\draw [style=dashededge](2.center) to (4.center);
            \draw [style=dashededge] (2.center) to (5.center);
		\draw [style=dashededge](3.center) to (5.center);
            \draw [style=dashededge] (3.center) to (4.center);
	\end{pgfonlayer}
\end{tikzpicture}}

\newcommand{\HOneTwo}{
\begin{tikzpicture}[scale = 0.75]
	\begin{pgfonlayer}{nodelayer}
            \node [style=root] (1) at (150:1.00) {};
		\node [style=root] (2) at (210:0.75) {};
            \node [style=root] (3) at (330:0.75) {};
            \node [style=root] (4) at (30:1.00) {};
		\node [style=vertex] (5) at (90:1.00) {};
	\end{pgfonlayer}
	\begin{pgfonlayer}{edgelayer}
		\draw [style=dashededge](1.center) to (2.center);
		\draw [style=dashededge](1.center) to (3.center);
		\draw [style=dashededge](2.center) to (3.center);
		\draw [style=dashededge](1.center) to (4.center);
		\draw (1.center) to (5.center);
		\draw [style=dashededge](4.center) to (5.center);
		\draw [style=dashededge](2.center) to (4.center);
            \draw [style=dashededge] (2.center) to (5.center);
		\draw [style=dashededge](3.center) to (5.center);
            \draw [style=dashededge] (3.center) to (4.center);
	\end{pgfonlayer}
\end{tikzpicture}}

\newcommand{\HOneThree}{
\begin{tikzpicture}[scale = 0.75]
	\begin{pgfonlayer}{nodelayer}
		\node [style=root] (1) at (150:1.00) {};
		\node [style=root] (2) at (210:0.75) {};
            \node [style=root] (3) at (330:0.75) {};
            \node [style=root] (4) at (30:1.00) {};
		\node [style=vertex] (5) at (90:1.00) {};
	\end{pgfonlayer}
	\begin{pgfonlayer}{edgelayer}
		\draw [style=dashededge](1.center) to (2.center);
		\draw [style=dashededge](1.center) to (3.center);
		\draw [style=dashededge](2.center) to (3.center);
		\draw [style=dashededge](1.center) to (4.center);
		\draw [style=dashededge](1.center) to (5.center);
		\draw [style=dashededge](4.center) to (5.center);
		\draw [style=dashededge](2.center) to (4.center);
            \draw  (2.center) to (5.center);
		\draw [style=dashededge](3.center) to (5.center);
            \draw [style=dashededge] (3.center) to (4.center);
	\end{pgfonlayer}
\end{tikzpicture}}

\newcommand{\HOneFour}{
\begin{tikzpicture}[scale = 0.75]
	\begin{pgfonlayer}{nodelayer}
		\node [style=root] (1) at (150:1.00) {};
		\node [style=root] (2) at (210:0.75) {};
            \node [style=root] (3) at (330:0.75) {};
            \node [style=root] (4) at (30:1.00) {};
		\node [style=vertex] (5) at (90:1.00) {};
	\end{pgfonlayer}
	\begin{pgfonlayer}{edgelayer}
		\draw [style=dashededge](1.center) to (2.center);
		\draw [style=dashededge](1.center) to (3.center);
		\draw [style=dashededge](2.center) to (3.center);
		\draw [style=dashededge](1.center) to (4.center);
		\draw (1.center) to (5.center);
		\draw [style=dashededge](4.center) to (5.center);
		\draw [style=dashededge](2.center) to (4.center);
            \draw (2.center) to (5.center);
		\draw [style=dashededge](3.center) to (5.center);
            \draw [style=dashededge] (3.center) to (4.center);
	\end{pgfonlayer}
\end{tikzpicture}}

\newcommand{\HOneFive}{
\begin{tikzpicture}[scale = 0.75]
	\begin{pgfonlayer}{nodelayer}
		\node [style=root] (1) at (150:1.00) {};
		\node [style=root] (2) at (210:0.75) {};
            \node [style=root] (3) at (330:0.75) {};
            \node [style=root] (4) at (30:1.00) {};
		\node [style=vertex] (5) at (90:1.00) {};
	\end{pgfonlayer}
	\begin{pgfonlayer}{edgelayer}
		\draw [style=dashededge](1.center) to (2.center);
		\draw [style=dashededge](1.center) to (3.center);
		\draw [style=dashededge](2.center) to (3.center);
		\draw [style=dashededge](1.center) to (4.center);
		\draw [style=dashededge](1.center) to (5.center);
		\draw [style=dashededge](4.center) to (5.center);
		\draw [style=dashededge](2.center) to (4.center);
            \draw [style=dashededge](2.center) to (5.center);
		\draw (3.center) to (5.center);
            \draw [style=dashededge] (3.center) to (4.center);
	\end{pgfonlayer}
\end{tikzpicture}}

\newcommand{\HOneSix}{
\begin{tikzpicture}[scale = 0.75]
	\begin{pgfonlayer}{nodelayer}
		\node [style=root] (1) at (150:1.00) {};
		\node [style=root] (2) at (210:0.75) {};
            \node [style=root] (3) at (330:0.75) {};
            \node [style=root] (4) at (30:1.00) {};
		\node [style=vertex] (5) at (90:1.00) {};
	\end{pgfonlayer}
	\begin{pgfonlayer}{edgelayer}
		\draw [style=dashededge](1.center) to (2.center);
		\draw [style=dashededge](1.center) to (3.center);
		\draw [style=dashededge](2.center) to (3.center);
		\draw [style=dashededge](1.center) to (4.center);
		\draw (1.center) to (5.center);
		\draw [style=dashededge](4.center) to (5.center);
		\draw [style=dashededge](2.center) to (4.center);
            \draw [style=dashededge](2.center) to (5.center);
		\draw (3.center) to (5.center);
            \draw [style=dashededge] (3.center) to (4.center);
	\end{pgfonlayer}
\end{tikzpicture}}

\newcommand{\HOneSeven}{
\begin{tikzpicture}[scale = 0.75]
	\begin{pgfonlayer}{nodelayer}
		\node [style=root] (1) at (150:1.00) {};
		\node [style=root] (2) at (210:0.75) {};
            \node [style=root] (3) at (330:0.75) {};
            \node [style=root] (4) at (30:1.00) {};
		\node [style=vertex] (5) at (90:1.00) {};
	\end{pgfonlayer}
	\begin{pgfonlayer}{edgelayer}
		\draw [style=dashededge](1.center) to (2.center);
		\draw [style=dashededge](1.center) to (3.center);
		\draw [style=dashededge](2.center) to (3.center);
		\draw [style=dashededge](1.center) to (4.center);
		\draw [style=dashededge](1.center) to (5.center);
		\draw [style=dashededge](4.center) to (5.center);
		\draw [style=dashededge](2.center) to (4.center);
            \draw (2.center) to (5.center);
		\draw (3.center) to (5.center);
            \draw [style=dashededge] (3.center) to (4.center);
	\end{pgfonlayer}
\end{tikzpicture}}

\newcommand{\HOneEight}{
\begin{tikzpicture}[scale = 0.75]
	\begin{pgfonlayer}{nodelayer}
		\node [style=root] (1) at (150:1.00) {};
		\node [style=root] (2) at (210:0.75) {};
            \node [style=root] (3) at (330:0.75) {};
            \node [style=root] (4) at (30:1.00) {};
		\node [style=vertex] (5) at (90:1.00) {};
	\end{pgfonlayer}
	\begin{pgfonlayer}{edgelayer}
		\draw [style=dashededge](1.center) to (2.center);
		\draw [style=dashededge](1.center) to (3.center);
		\draw [style=dashededge](2.center) to (3.center);
		\draw [style=dashededge](1.center) to (4.center);
		\draw (1.center) to (5.center);
		\draw [style=dashededge](4.center) to (5.center);
		\draw [style=dashededge](2.center) to (4.center);
            \draw (2.center) to (5.center);
		\draw (3.center) to (5.center);
            \draw [style=dashededge] (3.center) to (4.center);
	\end{pgfonlayer}
\end{tikzpicture}}

\newcommand{\HOneNine}{
\begin{tikzpicture}[scale = 0.75]
	\begin{pgfonlayer}{nodelayer}
		\node [style=root] (1) at (150:1.00) {};
		\node [style=root] (2) at (210:0.75) {};
            \node [style=root] (3) at (330:0.75) {};
            \node [style=root] (4) at (30:1.00) {};
		\node [style=vertex] (5) at (90:1.00) {};
	\end{pgfonlayer}
	\begin{pgfonlayer}{edgelayer}
		\draw [style=dashededge](1.center) to (2.center);
		\draw [style=dashededge](1.center) to (3.center);
		\draw [style=dashededge](2.center) to (3.center);
		\draw [style=dashededge](1.center) to (4.center);
		\draw [style=dashededge](1.center) to (5.center);
		\draw (4.center) to (5.center);
		\draw [style=dashededge](2.center) to (4.center);
            \draw [style=dashededge](2.center) to (5.center);
		\draw [style=dashededge](3.center) to (5.center);
            \draw [style=dashededge] (3.center) to (4.center);
	\end{pgfonlayer}
\end{tikzpicture}}

\newcommand{\HOneTen}{
\begin{tikzpicture}[scale = 0.75]
	\begin{pgfonlayer}{nodelayer}
		\node [style=root] (1) at (150:1.00) {};
		\node [style=root] (2) at (210:0.75) {};
            \node [style=root] (3) at (330:0.75) {};
            \node [style=root] (4) at (30:1.00) {};
		\node [style=vertex] (5) at (90:1.00) {};
	\end{pgfonlayer}
	\begin{pgfonlayer}{edgelayer}
		\draw [style=dashededge](1.center) to (2.center);
		\draw [style=dashededge](1.center) to (3.center);
		\draw [style=dashededge](2.center) to (3.center);
		\draw [style=dashededge](1.center) to (4.center);
		\draw (1.center) to (5.center);
		\draw (4.center) to (5.center);
		\draw [style=dashededge](2.center) to (4.center);
            \draw [style=dashededge](2.center) to (5.center);
		\draw [style=dashededge](3.center) to (5.center);
            \draw [style=dashededge] (3.center) to (4.center);
	\end{pgfonlayer}
\end{tikzpicture}}

\newcommand{\HOneEleven}{
\begin{tikzpicture}[scale = 0.75]
	\begin{pgfonlayer}{nodelayer}
		\node [style=root] (1) at (150:1.00) {};
		\node [style=root] (2) at (210:0.75) {};
            \node [style=root] (3) at (330:0.75) {};
            \node [style=root] (4) at (30:1.00) {};
		\node [style=vertex] (5) at (90:1.00) {};
	\end{pgfonlayer}
	\begin{pgfonlayer}{edgelayer}
		\draw [style=dashededge](1.center) to (2.center);
		\draw [style=dashededge](1.center) to (3.center);
		\draw [style=dashededge](2.center) to (3.center);
		\draw [style=dashededge](1.center) to (4.center);
		\draw [style=dashededge](1.center) to (5.center);
		\draw (4.center) to (5.center);
		\draw [style=dashededge](2.center) to (4.center);
            \draw (2.center) to (5.center);
		\draw [style=dashededge](3.center) to (5.center);
            \draw [style=dashededge] (3.center) to (4.center);
	\end{pgfonlayer}
\end{tikzpicture}}

\newcommand{\HOneTwelve}{
\begin{tikzpicture}[scale = 0.75]
	\begin{pgfonlayer}{nodelayer}
		\node [style=root] (1) at (150:1.00) {};
		\node [style=root] (2) at (210:0.75) {};
            \node [style=root] (3) at (330:0.75) {};
            \node [style=root] (4) at (30:1.00) {};
		\node [style=vertex] (5) at (90:1.00) {};
	\end{pgfonlayer}
	\begin{pgfonlayer}{edgelayer}
		\draw [style=dashededge](1.center) to (2.center);
		\draw [style=dashededge](1.center) to (3.center);
		\draw [style=dashededge](2.center) to (3.center);
		\draw [style=dashededge](1.center) to (4.center);
		\draw (1.center) to (5.center);
		\draw (4.center) to (5.center);
		\draw [style=dashededge](2.center) to (4.center);
            \draw (2.center) to (5.center);
		\draw [style=dashededge](3.center) to (5.center);
            \draw [style=dashededge] (3.center) to (4.center);
	\end{pgfonlayer}
\end{tikzpicture}}

\newcommand{\HOneThirteen}{
\begin{tikzpicture}[scale = 0.75]
	\begin{pgfonlayer}{nodelayer}
		\node [style=root] (1) at (150:1.00) {};
		\node [style=root] (2) at (210:0.75) {};
            \node [style=root] (3) at (330:0.75) {};
            \node [style=root] (4) at (30:1.00) {};
		\node [style=vertex] (5) at (90:1.00) {};
	\end{pgfonlayer}
	\begin{pgfonlayer}{edgelayer}
		\draw [style=dashededge](1.center) to (2.center);
		\draw [style=dashededge](1.center) to (3.center);
		\draw [style=dashededge](2.center) to (3.center);
		\draw [style=dashededge](1.center) to (4.center);
		\draw [style=dashededge](1.center) to (5.center);
		\draw (4.center) to (5.center);
		\draw [style=dashededge](2.center) to (4.center);
            \draw [style=dashededge](2.center) to (5.center);
		\draw (3.center) to (5.center);
            \draw [style=dashededge] (3.center) to (4.center);
	\end{pgfonlayer}
\end{tikzpicture}}

\newcommand{\HOneFourteen}{
\begin{tikzpicture}[scale = 0.75]
	\begin{pgfonlayer}{nodelayer}
		\node [style=root] (1) at (150:1.00) {};
		\node [style=root] (2) at (210:0.75) {};
            \node [style=root] (3) at (330:0.75) {};
            \node [style=root] (4) at (30:1.00) {};
		\node [style=vertex] (5) at (90:1.00) {};
	\end{pgfonlayer}
	\begin{pgfonlayer}{edgelayer}
		\draw [style=dashededge](1.center) to (2.center);
		\draw [style=dashededge](1.center) to (3.center);
		\draw [style=dashededge](2.center) to (3.center);
		\draw [style=dashededge](1.center) to (4.center);
		\draw (1.center) to (5.center);
		\draw (4.center) to (5.center);
		\draw [style=dashededge](2.center) to (4.center);
            \draw [style=dashededge](2.center) to (5.center);
		\draw (3.center) to (5.center);
            \draw [style=dashededge] (3.center) to (4.center);
	\end{pgfonlayer}
\end{tikzpicture}}

\newcommand{\HOneFifteen}{
\begin{tikzpicture}[scale = 0.75]
	\begin{pgfonlayer}{nodelayer}
		\node [style=root] (1) at (150:1.00) {};
		\node [style=root] (2) at (210:0.75) {};
            \node [style=root] (3) at (330:0.75) {};
            \node [style=root] (4) at (30:1.00) {};
		\node [style=vertex] (5) at (90:1.00) {};
	\end{pgfonlayer}
	\begin{pgfonlayer}{edgelayer}
		\draw [style=dashededge](1.center) to (2.center);
		\draw [style=dashededge](1.center) to (3.center);
		\draw [style=dashededge](2.center) to (3.center);
		\draw [style=dashededge](1.center) to (4.center);
		\draw [style=dashededge](1.center) to (5.center);
		\draw (4.center) to (5.center);
		\draw [style=dashededge](2.center) to (4.center);
            \draw (2.center) to (5.center);
		\draw (3.center) to (5.center);
            \draw [style=dashededge] (3.center) to (4.center);
	\end{pgfonlayer}
\end{tikzpicture}}

\newcommand{\HOneSixteen}{
\begin{tikzpicture}[scale = 0.75]
	\begin{pgfonlayer}{nodelayer}
		\node [style=root] (1) at (150:1.00) {};
		\node [style=root] (2) at (210:0.75) {};
            \node [style=root] (3) at (330:0.75) {};
            \node [style=root] (4) at (30:1.00) {};
		\node [style=vertex] (5) at (90:1.00) {};
	\end{pgfonlayer}
	\begin{pgfonlayer}{edgelayer}
		\draw [style=dashededge](1.center) to (2.center);
		\draw [style=dashededge](1.center) to (3.center);
		\draw [style=dashededge](2.center) to (3.center);
		\draw [style=dashededge](1.center) to (4.center);
		\draw (1.center) to (5.center);
		\draw (4.center) to (5.center);
		\draw [style=dashededge](2.center) to (4.center);
            \draw (2.center) to (5.center);
		\draw (3.center) to (5.center);
            \draw [style=dashededge] (3.center) to (4.center);
	\end{pgfonlayer}
\end{tikzpicture}}

\newcommand{\LOneOne}{
\begin{tikzpicture}[scale = 0.75]
	\begin{pgfonlayer}{nodelayer}
		\node [style=root] (1) at (150:1.00) {};
		\node [style=root] (2) at (210:0.75) {};
            \node [style=root] (3) at (330:0.75) {};
            \node [style=root] (4) at (30:1.00) {};
		\node [style=vertex] (5) at (90:1.00) {};
	\end{pgfonlayer}
	\begin{pgfonlayer}{edgelayer}
		\draw [style=dashededge](1.center) to (2.center);
		\draw [style=dashededge](1.center) to (3.center);
		\draw (2.center) to (3.center);
		\draw (1.center) to (4.center);
		\draw [style=dashededge](1.center) to (5.center);
		\draw [style=dashededge](4.center) to (5.center);
		\draw [style=dashededge](2.center) to (4.center);
            \draw [style=dashededge](2.center) to (5.center);
		\draw [style=dashededge](3.center) to (5.center);
            \draw (3.center) to (4.center);
	\end{pgfonlayer}
\end{tikzpicture}}

\newcommand{\LOneTwo}{
\begin{tikzpicture}[scale = 0.75]
	\begin{pgfonlayer}{nodelayer}
		\node [style=root] (1) at (150:1.00) {};
		\node [style=root] (2) at (210:0.75) {};
            \node [style=root] (3) at (330:0.75) {};
            \node [style=root] (4) at (30:1.00) {};
		\node [style=vertex] (5) at (90:1.00) {};
	\end{pgfonlayer}
	\begin{pgfonlayer}{edgelayer}
		\draw [style=dashededge](1.center) to (2.center);
		\draw [style=dashededge](1.center) to (3.center);
		\draw (2.center) to (3.center);
		\draw (1.center) to (4.center);
		\draw (1.center) to (5.center);
		\draw [style=dashededge](4.center) to (5.center);
		\draw [style=dashededge](2.center) to (4.center);
            \draw [style=dashededge](2.center) to (5.center);
		\draw [style=dashededge](3.center) to (5.center);
            \draw (3.center) to (4.center);
	\end{pgfonlayer}
\end{tikzpicture}}

\newcommand{\LOneThree}{
\begin{tikzpicture}[scale = 0.75]
	\begin{pgfonlayer}{nodelayer}
		\node [style=root] (1) at (150:1.00) {};
		\node [style=root] (2) at (210:0.75) {};
            \node [style=root] (3) at (330:0.75) {};
            \node [style=root] (4) at (30:1.00) {};
		\node [style=vertex] (5) at (90:1.00) {};
	\end{pgfonlayer}
	\begin{pgfonlayer}{edgelayer}
		\draw [style=dashededge](1.center) to (2.center);
		\draw [style=dashededge](1.center) to (3.center);
		\draw (2.center) to (3.center);
		\draw (1.center) to (4.center);
		\draw [style=dashededge](1.center) to (5.center);
		\draw [style=dashededge](4.center) to (5.center);
		\draw [style=dashededge](2.center) to (4.center);
            \draw (2.center) to (5.center);
		\draw [style=dashededge](3.center) to (5.center);
            \draw (3.center) to (4.center);
	\end{pgfonlayer}
\end{tikzpicture}}

\newcommand{\LOneFour}{
\begin{tikzpicture}[scale = 0.75]
	\begin{pgfonlayer}{nodelayer}
		\node [style=root] (1) at (150:1.00) {};
		\node [style=root] (2) at (210:0.75) {};
            \node [style=root] (3) at (330:0.75) {};
            \node [style=root] (4) at (30:1.00) {};
		\node [style=vertex] (5) at (90:1.00) {};
	\end{pgfonlayer}
	\begin{pgfonlayer}{edgelayer}
		\draw [style=dashededge](1.center) to (2.center);
		\draw [style=dashededge](1.center) to (3.center);
		\draw (2.center) to (3.center);
		\draw (1.center) to (4.center);
		\draw (1.center) to (5.center);
		\draw [style=dashededge](4.center) to (5.center);
		\draw [style=dashededge](2.center) to (4.center);
            \draw (2.center) to (5.center);
		\draw [style=dashededge](3.center) to (5.center);
            \draw (3.center) to (4.center);
	\end{pgfonlayer}
\end{tikzpicture}}

\newcommand{\LOneFive}{
\begin{tikzpicture}[scale = 0.75]
	\begin{pgfonlayer}{nodelayer}
		\node [style=root] (1) at (150:1.00) {};
		\node [style=root] (2) at (210:0.75) {};
            \node [style=root] (3) at (330:0.75) {};
            \node [style=root] (4) at (30:1.00) {};
		\node [style=vertex] (5) at (90:1.00) {};
	\end{pgfonlayer}
	\begin{pgfonlayer}{edgelayer}
		\draw [style=dashededge](1.center) to (2.center);
		\draw [style=dashededge](1.center) to (3.center);
		\draw (2.center) to (3.center);
		\draw (1.center) to (4.center);
		\draw [style=dashededge](1.center) to (5.center);
		\draw [style=dashededge](4.center) to (5.center);
		\draw [style=dashededge](2.center) to (4.center);
            \draw [style=dashededge](2.center) to (5.center);
		\draw (3.center) to (5.center);
            \draw (3.center) to (4.center);
	\end{pgfonlayer}
\end{tikzpicture}}

\newcommand{\LOneSix}{
\begin{tikzpicture}[scale = 0.75]
	\begin{pgfonlayer}{nodelayer}
		\node [style=root] (1) at (150:1.00) {};
		\node [style=root] (2) at (210:0.75) {};
            \node [style=root] (3) at (330:0.75) {};
            \node [style=root] (4) at (30:1.00) {};
		\node [style=vertex] (5) at (90:1.00) {};
	\end{pgfonlayer}
	\begin{pgfonlayer}{edgelayer}
		\draw [style=dashededge](1.center) to (2.center);
		\draw [style=dashededge](1.center) to (3.center);
		\draw (2.center) to (3.center);
		\draw (1.center) to (4.center);
		\draw (1.center) to (5.center);
		\draw [style=dashededge](4.center) to (5.center);
		\draw [style=dashededge](2.center) to (4.center);
            \draw [style=dashededge](2.center) to (5.center);
		\draw (3.center) to (5.center);
            \draw (3.center) to (4.center);
	\end{pgfonlayer}
\end{tikzpicture}}

\newcommand{\LOneSeven}{
\begin{tikzpicture}[scale = 0.75]
	\begin{pgfonlayer}{nodelayer}
		\node [style=root] (1) at (150:1.00) {};
		\node [style=root] (2) at (210:0.75) {};
            \node [style=root] (3) at (330:0.75) {};
            \node [style=root] (4) at (30:1.00) {};
		\node [style=vertex] (5) at (90:1.00) {};
	\end{pgfonlayer}
	\begin{pgfonlayer}{edgelayer}
		\draw [style=dashededge](1.center) to (2.center);
		\draw [style=dashededge](1.center) to (3.center);
		\draw (2.center) to (3.center);
		\draw (1.center) to (4.center);
		\draw [style=dashededge](1.center) to (5.center);
		\draw [style=dashededge](4.center) to (5.center);
		\draw [style=dashededge](2.center) to (4.center);
            \draw (2.center) to (5.center);
		\draw (3.center) to (5.center);
            \draw (3.center) to (4.center);
	\end{pgfonlayer}
\end{tikzpicture}}

\newcommand{\LOneEight}{
\begin{tikzpicture}[scale = 0.75]
	\begin{pgfonlayer}{nodelayer}
		\node [style=root] (1) at (150:1.00) {};
		\node [style=root] (2) at (210:0.75) {};
            \node [style=root] (3) at (330:0.75) {};
            \node [style=root] (4) at (30:1.00) {};
		\node [style=vertex] (5) at (90:1.00) {};
	\end{pgfonlayer}
	\begin{pgfonlayer}{edgelayer}
		\draw [style=dashededge](1.center) to (2.center);
		\draw [style=dashededge](1.center) to (3.center);
		\draw (2.center) to (3.center);
		\draw (1.center) to (4.center);
		\draw (1.center) to (5.center);
		\draw [style=dashededge](4.center) to (5.center);
		\draw [style=dashededge](2.center) to (4.center);
            \draw (2.center) to (5.center);
		\draw (3.center) to (5.center);
            \draw (3.center) to (4.center);
	\end{pgfonlayer}
\end{tikzpicture}}

\newcommand{\LOneNine}{
\begin{tikzpicture}[scale = 0.75]
	\begin{pgfonlayer}{nodelayer}
		\node [style=root] (1) at (150:1.00) {};
		\node [style=root] (2) at (210:0.75) {};
            \node [style=root] (3) at (330:0.75) {};
            \node [style=root] (4) at (30:1.00) {};
		\node [style=vertex] (5) at (90:1.00) {};
	\end{pgfonlayer}
	\begin{pgfonlayer}{edgelayer}
		\draw [style=dashededge](1.center) to (2.center);
		\draw [style=dashededge](1.center) to (3.center);
		\draw (2.center) to (3.center);
		\draw (1.center) to (4.center);
		\draw [style=dashededge](1.center) to (5.center);
		\draw (4.center) to (5.center);
		\draw [style=dashededge](2.center) to (4.center);
            \draw [style=dashededge](2.center) to (5.center);
		\draw [style=dashededge](3.center) to (5.center);
            \draw (3.center) to (4.center);
	\end{pgfonlayer}
\end{tikzpicture}}

\newcommand{\LOneTen}{
\begin{tikzpicture}[scale = 0.75]
	\begin{pgfonlayer}{nodelayer}
		\node [style=root] (1) at (150:1.00) {};
		\node [style=root] (2) at (210:0.75) {};
            \node [style=root] (3) at (330:0.75) {};
            \node [style=root] (4) at (30:1.00) {};
		\node [style=vertex] (5) at (90:1.00) {};
	\end{pgfonlayer}
	\begin{pgfonlayer}{edgelayer}
		\draw [style=dashededge](1.center) to (2.center);
		\draw [style=dashededge](1.center) to (3.center);
		\draw (2.center) to (3.center);
		\draw (1.center) to (4.center);
		\draw (1.center) to (5.center);
		\draw (4.center) to (5.center);
		\draw [style=dashededge](2.center) to (4.center);
            \draw [style=dashededge](2.center) to (5.center);
		\draw [style=dashededge](3.center) to (5.center);
            \draw (3.center) to (4.center);
	\end{pgfonlayer}
\end{tikzpicture}}

\newcommand{\LOneEleven}{
\begin{tikzpicture}[scale = 0.75]
	\begin{pgfonlayer}{nodelayer}
		\node [style=root] (1) at (150:1.00) {};
		\node [style=root] (2) at (210:0.75) {};
            \node [style=root] (3) at (330:0.75) {};
            \node [style=root] (4) at (30:1.00) {};
		\node [style=vertex] (5) at (90:1.00) {};
	\end{pgfonlayer}
	\begin{pgfonlayer}{edgelayer}
		\draw [style=dashededge](1.center) to (2.center);
		\draw [style=dashededge](1.center) to (3.center);
		\draw (2.center) to (3.center);
		\draw (1.center) to (4.center);
		\draw [style=dashededge](1.center) to (5.center);
		\draw (4.center) to (5.center);
		\draw [style=dashededge](2.center) to (4.center);
            \draw (2.center) to (5.center);
		\draw [style=dashededge](3.center) to (5.center);
            \draw (3.center) to (4.center);
	\end{pgfonlayer}
\end{tikzpicture}}

\newcommand{\LOneTwelve}{
\begin{tikzpicture}[scale = 0.75]
	\begin{pgfonlayer}{nodelayer}
		\node [style=root] (1) at (150:1.00) {};
		\node [style=root] (2) at (210:0.75) {};
            \node [style=root] (3) at (330:0.75) {};
            \node [style=root] (4) at (30:1.00) {};
		\node [style=vertex] (5) at (90:1.00) {};
	\end{pgfonlayer}
	\begin{pgfonlayer}{edgelayer}
		\draw [style=dashededge](1.center) to (2.center);
		\draw [style=dashededge](1.center) to (3.center);
		\draw (2.center) to (3.center);
		\draw (1.center) to (4.center);
		\draw (1.center) to (5.center);
		\draw (4.center) to (5.center);
		\draw [style=dashededge](2.center) to (4.center);
            \draw (2.center) to (5.center);
		\draw [style=dashededge](3.center) to (5.center);
            \draw (3.center) to (4.center);
	\end{pgfonlayer}
\end{tikzpicture}}

\newcommand{\LOneThirteen}{
\begin{tikzpicture}[scale = 0.75]
	\begin{pgfonlayer}{nodelayer}
		\node [style=root] (1) at (150:1.00) {};
		\node [style=root] (2) at (210:0.75) {};
            \node [style=root] (3) at (330:0.75) {};
            \node [style=root] (4) at (30:1.00) {};
		\node [style=vertex] (5) at (90:1.00) {};
	\end{pgfonlayer}
	\begin{pgfonlayer}{edgelayer}
		\draw [style=dashededge](1.center) to (2.center);
		\draw [style=dashededge](1.center) to (3.center);
		\draw (2.center) to (3.center);
		\draw (1.center) to (4.center);
		\draw [style=dashededge](1.center) to (5.center);
		\draw (4.center) to (5.center);
		\draw [style=dashededge](2.center) to (4.center);
            \draw [style=dashededge](2.center) to (5.center);
		\draw (3.center) to (5.center);
            \draw (3.center) to (4.center);
	\end{pgfonlayer}
\end{tikzpicture}}

\newcommand{\LOneFourteen}{
\begin{tikzpicture}[scale = 0.75]
	\begin{pgfonlayer}{nodelayer}
		\node [style=root] (1) at (150:1.00) {};
		\node [style=root] (2) at (210:0.75) {};
            \node [style=root] (3) at (330:0.75) {};
            \node [style=root] (4) at (30:1.00) {};
		\node [style=vertex] (5) at (90:1.00) {};
	\end{pgfonlayer}
	\begin{pgfonlayer}{edgelayer}
		\draw [style=dashededge](1.center) to (2.center);
		\draw [style=dashededge](1.center) to (3.center);
		\draw (2.center) to (3.center);
		\draw (1.center) to (4.center);
		\draw (1.center) to (5.center);
		\draw (4.center) to (5.center);
		\draw [style=dashededge](2.center) to (4.center);
            \draw [style=dashededge](2.center) to (5.center);
		\draw (3.center) to (5.center);
            \draw (3.center) to (4.center);
	\end{pgfonlayer}
\end{tikzpicture}}

\newcommand{\LOneFifteen}{
\begin{tikzpicture}[scale = 0.75]
	\begin{pgfonlayer}{nodelayer}
		\node [style=root] (1) at (150:1.00) {};
		\node [style=root] (2) at (210:0.75) {};
            \node [style=root] (3) at (330:0.75) {};
            \node [style=root] (4) at (30:1.00) {};
		\node [style=vertex] (5) at (90:1.00) {};
	\end{pgfonlayer}
	\begin{pgfonlayer}{edgelayer}
		\draw [style=dashededge](1.center) to (2.center);
		\draw [style=dashededge](1.center) to (3.center);
		\draw (2.center) to (3.center);
		\draw (1.center) to (4.center);
		\draw [style=dashededge](1.center) to (5.center);
		\draw (4.center) to (5.center);
		\draw [style=dashededge](2.center) to (4.center);
            \draw (2.center) to (5.center);
		\draw (3.center) to (5.center);
            \draw (3.center) to (4.center);
	\end{pgfonlayer}
\end{tikzpicture}}

\newcommand{\LOneSixteen}{
\begin{tikzpicture}[scale = 0.75]
	\begin{pgfonlayer}{nodelayer}
		\node [style=root] (1) at (150:1.00) {};
		\node [style=root] (2) at (210:0.75) {};
            \node [style=root] (3) at (330:0.75) {};
            \node [style=root] (4) at (30:1.00) {};
		\node [style=vertex] (5) at (90:1.00) {};
	\end{pgfonlayer}
	\begin{pgfonlayer}{edgelayer}
		\draw [style=dashededge](1.center) to (2.center);
		\draw [style=dashededge](1.center) to (3.center);
		\draw (2.center) to (3.center);
		\draw (1.center) to (4.center);
		\draw (1.center) to (5.center);
		\draw (4.center) to (5.center);
		\draw [style=dashededge](2.center) to (4.center);
            \draw (2.center) to (5.center);
		\draw (3.center) to (5.center);
            \draw (3.center) to (4.center);
	\end{pgfonlayer}
\end{tikzpicture}}

\newcommand{\MOneOne}{
\begin{tikzpicture}[scale = 0.75]
	\begin{pgfonlayer}{nodelayer}
		\node [style=root] (1) at (150:1.00) {};
		\node [style=root] (2) at (210:0.75) {};
            \node [style=root] (3) at (330:0.75) {};
            \node [style=root] (4) at (30:1.00) {};
		\node [style=vertex] (5) at (90:1.00) {};
	\end{pgfonlayer}
	\begin{pgfonlayer}{edgelayer}
		\draw [style=dashededge](1.center) to (2.center);
		\draw [style=dashededge](1.center) to (3.center);
		\draw [style=dashededge](2.center) to (3.center);
		\draw [style=dashededge](1.center) to (4.center);
		\draw [style=dashededge](1.center) to (5.center);
		\draw [style=dashededge](4.center) to (5.center);
		\draw (2.center) to (4.center);
            \draw [style=dashededge](2.center) to (5.center);
		\draw [style=dashededge](3.center) to (5.center);
            \draw (3.center) to (4.center);
	\end{pgfonlayer}
\end{tikzpicture}}

\newcommand{\MOneTwo}{
\begin{tikzpicture}[scale = 0.75]
	\begin{pgfonlayer}{nodelayer}
		\node [style=root] (1) at (150:1.00) {};
		\node [style=root] (2) at (210:0.75) {};
            \node [style=root] (3) at (330:0.75) {};
            \node [style=root] (4) at (30:1.00) {};
		\node [style=vertex] (5) at (90:1.00) {};
	\end{pgfonlayer}
	\begin{pgfonlayer}{edgelayer}
		\draw [style=dashededge](1.center) to (2.center);
		\draw [style=dashededge](1.center) to (3.center);
		\draw [style=dashededge](2.center) to (3.center);
		\draw [style=dashededge](1.center) to (4.center);
		\draw (1.center) to (5.center);
		\draw [style=dashededge](4.center) to (5.center);
		\draw (2.center) to (4.center);
            \draw [style=dashededge](2.center) to (5.center);
		\draw [style=dashededge](3.center) to (5.center);
            \draw (3.center) to (4.center);
	\end{pgfonlayer}
\end{tikzpicture}}

\newcommand{\MOneThree}{
\begin{tikzpicture}[scale = 0.75]
	\begin{pgfonlayer}{nodelayer}
		\node [style=root] (1) at (150:1.00) {};
		\node [style=root] (2) at (210:0.75) {};
            \node [style=root] (3) at (330:0.75) {};
            \node [style=root] (4) at (30:1.00) {};
		\node [style=vertex] (5) at (90:1.00) {};
	\end{pgfonlayer}
	\begin{pgfonlayer}{edgelayer}
		\draw [style=dashededge](1.center) to (2.center);
		\draw [style=dashededge](1.center) to (3.center);
		\draw [style=dashededge](2.center) to (3.center);
		\draw [style=dashededge](1.center) to (4.center);
		\draw [style=dashededge](1.center) to (5.center);
		\draw [style=dashededge](4.center) to (5.center);
		\draw (2.center) to (4.center);
            \draw (2.center) to (5.center);
		\draw [style=dashededge](3.center) to (5.center);
            \draw (3.center) to (4.center);
	\end{pgfonlayer}
\end{tikzpicture}}

\newcommand{\MOneFour}{
\begin{tikzpicture}[scale = 0.75]
	\begin{pgfonlayer}{nodelayer}
		\node [style=root] (1) at (150:1.00) {};
		\node [style=root] (2) at (210:0.75) {};
            \node [style=root] (3) at (330:0.75) {};
            \node [style=root] (4) at (30:1.00) {};
		\node [style=vertex] (5) at (90:1.00) {};
	\end{pgfonlayer}
	\begin{pgfonlayer}{edgelayer}
		\draw [style=dashededge](1.center) to (2.center);
		\draw [style=dashededge](1.center) to (3.center);
		\draw [style=dashededge](2.center) to (3.center);
		\draw [style=dashededge](1.center) to (4.center);
		\draw (1.center) to (5.center);
		\draw [style=dashededge](4.center) to (5.center);
		\draw (2.center) to (4.center);
            \draw (2.center) to (5.center);
		\draw [style=dashededge](3.center) to (5.center);
            \draw (3.center) to (4.center);
	\end{pgfonlayer}
\end{tikzpicture}}

\newcommand{\MOneFive}{
\begin{tikzpicture}[scale = 0.75]
	\begin{pgfonlayer}{nodelayer}
		\node [style=root] (1) at (150:1.00) {};
		\node [style=root] (2) at (210:0.75) {};
            \node [style=root] (3) at (330:0.75) {};
            \node [style=root] (4) at (30:1.00) {};
		\node [style=vertex] (5) at (90:1.00) {};
	\end{pgfonlayer}
	\begin{pgfonlayer}{edgelayer}
		\draw [style=dashededge](1.center) to (2.center);
		\draw [style=dashededge](1.center) to (3.center);
		\draw [style=dashededge](2.center) to (3.center);
		\draw [style=dashededge](1.center) to (4.center);
		\draw [style=dashededge](1.center) to (5.center);
		\draw [style=dashededge](4.center) to (5.center);
		\draw (2.center) to (4.center);
            \draw [style=dashededge](2.center) to (5.center);
		\draw (3.center) to (5.center);
            \draw (3.center) to (4.center);
	\end{pgfonlayer}
\end{tikzpicture}}

\newcommand{\MOneSix}{
\begin{tikzpicture}[scale = 0.75]
	\begin{pgfonlayer}{nodelayer}
		\node [style=root] (1) at (150:1.00) {};
		\node [style=root] (2) at (210:0.75) {};
            \node [style=root] (3) at (330:0.75) {};
            \node [style=root] (4) at (30:1.00) {};
		\node [style=vertex] (5) at (90:1.00) {};
	\end{pgfonlayer}
	\begin{pgfonlayer}{edgelayer}
		\draw [style=dashededge](1.center) to (2.center);
		\draw [style=dashededge](1.center) to (3.center);
		\draw [style=dashededge](2.center) to (3.center);
		\draw [style=dashededge](1.center) to (4.center);
		\draw (1.center) to (5.center);
		\draw [style=dashededge](4.center) to (5.center);
		\draw (2.center) to (4.center);
            \draw [style=dashededge](2.center) to (5.center);
		\draw (3.center) to (5.center);
            \draw (3.center) to (4.center);
	\end{pgfonlayer}
\end{tikzpicture}}

\newcommand{\MOneSeven}{
\begin{tikzpicture}[scale = 0.75]
	\begin{pgfonlayer}{nodelayer}
		\node [style=root] (1) at (150:1.00) {};
		\node [style=root] (2) at (210:0.75) {};
            \node [style=root] (3) at (330:0.75) {};
            \node [style=root] (4) at (30:1.00) {};
		\node [style=vertex] (5) at (90:1.00) {};
	\end{pgfonlayer}
	\begin{pgfonlayer}{edgelayer}
		\draw [style=dashededge](1.center) to (2.center);
		\draw [style=dashededge](1.center) to (3.center);
		\draw [style=dashededge](2.center) to (3.center);
		\draw [style=dashededge](1.center) to (4.center);
		\draw [style=dashededge](1.center) to (5.center);
		\draw [style=dashededge](4.center) to (5.center);
		\draw (2.center) to (4.center);
            \draw (2.center) to (5.center);
		\draw (3.center) to (5.center);
            \draw (3.center) to (4.center);
	\end{pgfonlayer}
\end{tikzpicture}}

\newcommand{\MOneEight}{
\begin{tikzpicture}[scale = 0.75]
	\begin{pgfonlayer}{nodelayer}
		\node [style=root] (1) at (150:1.00) {};
		\node [style=root] (2) at (210:0.75) {};
            \node [style=root] (3) at (330:0.75) {};
            \node [style=root] (4) at (30:1.00) {};
		\node [style=vertex] (5) at (90:1.00) {};
	\end{pgfonlayer}
	\begin{pgfonlayer}{edgelayer}
		\draw [style=dashededge](1.center) to (2.center);
		\draw [style=dashededge](1.center) to (3.center);
		\draw [style=dashededge](2.center) to (3.center);
		\draw [style=dashededge](1.center) to (4.center);
		\draw (1.center) to (5.center);
		\draw [style=dashededge](4.center) to (5.center);
		\draw (2.center) to (4.center);
            \draw (2.center) to (5.center);
		\draw (3.center) to (5.center);
            \draw (3.center) to (4.center);
	\end{pgfonlayer}
\end{tikzpicture}}

\newcommand{\MOneNine}{
\begin{tikzpicture}[scale = 0.75]
	\begin{pgfonlayer}{nodelayer}
		\node [style=root] (1) at (150:1.00) {};
		\node [style=root] (2) at (210:0.75) {};
            \node [style=root] (3) at (330:0.75) {};
            \node [style=root] (4) at (30:1.00) {};
		\node [style=vertex] (5) at (90:1.00) {};
	\end{pgfonlayer}
	\begin{pgfonlayer}{edgelayer}
		\draw [style=dashededge](1.center) to (2.center);
		\draw [style=dashededge](1.center) to (3.center);
		\draw [style=dashededge](2.center) to (3.center);
		\draw [style=dashededge](1.center) to (4.center);
		\draw [style=dashededge](1.center) to (5.center);
		\draw (4.center) to (5.center);
		\draw (2.center) to (4.center);
            \draw [style=dashededge](2.center) to (5.center);
		\draw [style=dashededge](3.center) to (5.center);
            \draw (3.center) to (4.center);
	\end{pgfonlayer}
\end{tikzpicture}}

\newcommand{\MOneTen}{
\begin{tikzpicture}[scale = 0.75]
	\begin{pgfonlayer}{nodelayer}
		\node [style=root] (1) at (150:1.00) {};
		\node [style=root] (2) at (210:0.75) {};
            \node [style=root] (3) at (330:0.75) {};
            \node [style=root] (4) at (30:1.00) {};
		\node [style=vertex] (5) at (90:1.00) {};
	\end{pgfonlayer}
	\begin{pgfonlayer}{edgelayer}
		\draw [style=dashededge](1.center) to (2.center);
		\draw [style=dashededge](1.center) to (3.center);
		\draw [style=dashededge](2.center) to (3.center);
		\draw [style=dashededge](1.center) to (4.center);
		\draw (1.center) to (5.center);
		\draw (4.center) to (5.center);
		\draw (2.center) to (4.center);
            \draw [style=dashededge](2.center) to (5.center);
		\draw [style=dashededge](3.center) to (5.center);
            \draw (3.center) to (4.center);
	\end{pgfonlayer}
\end{tikzpicture}}

\newcommand{\MOneEleven}{
\begin{tikzpicture}[scale = 0.75]
	\begin{pgfonlayer}{nodelayer}
		\node [style=root] (1) at (150:1.00) {};
		\node [style=root] (2) at (210:0.75) {};
            \node [style=root] (3) at (330:0.75) {};
            \node [style=root] (4) at (30:1.00) {};
		\node [style=vertex] (5) at (90:1.00) {};
	\end{pgfonlayer}
	\begin{pgfonlayer}{edgelayer}
		\draw [style=dashededge](1.center) to (2.center);
		\draw [style=dashededge](1.center) to (3.center);
		\draw [style=dashededge](2.center) to (3.center);
		\draw [style=dashededge](1.center) to (4.center);
		\draw [style=dashededge](1.center) to (5.center);
		\draw (4.center) to (5.center);
		\draw (2.center) to (4.center);
            \draw (2.center) to (5.center);
		\draw [style=dashededge](3.center) to (5.center);
            \draw (3.center) to (4.center);
	\end{pgfonlayer}
\end{tikzpicture}}

\newcommand{\MOneTwelve}{
\begin{tikzpicture}[scale = 0.75]
	\begin{pgfonlayer}{nodelayer}
		\node [style=root] (1) at (150:1.00) {};
		\node [style=root] (2) at (210:0.75) {};
            \node [style=root] (3) at (330:0.75) {};
            \node [style=root] (4) at (30:1.00) {};
		\node [style=vertex] (5) at (90:1.00) {};
	\end{pgfonlayer}
	\begin{pgfonlayer}{edgelayer}
		\draw [style=dashededge](1.center) to (2.center);
		\draw [style=dashededge](1.center) to (3.center);
		\draw [style=dashededge](2.center) to (3.center);
		\draw [style=dashededge](1.center) to (4.center);
		\draw (1.center) to (5.center);
		\draw (4.center) to (5.center);
		\draw (2.center) to (4.center);
            \draw (2.center) to (5.center);
		\draw [style=dashededge](3.center) to (5.center);
            \draw (3.center) to (4.center);
	\end{pgfonlayer}
\end{tikzpicture}}

\newcommand{\MOneThirteen}{
\begin{tikzpicture}[scale = 0.75]
	\begin{pgfonlayer}{nodelayer}
		\node [style=root] (1) at (150:1.00) {};
		\node [style=root] (2) at (210:0.75) {};
            \node [style=root] (3) at (330:0.75) {};
            \node [style=root] (4) at (30:1.00) {};
		\node [style=vertex] (5) at (90:1.00) {};
	\end{pgfonlayer}
	\begin{pgfonlayer}{edgelayer}
		\draw [style=dashededge](1.center) to (2.center);
		\draw [style=dashededge](1.center) to (3.center);
		\draw [style=dashededge](2.center) to (3.center);
		\draw [style=dashededge](1.center) to (4.center);
		\draw [style=dashededge](1.center) to (5.center);
		\draw (4.center) to (5.center);
		\draw (2.center) to (4.center);
            \draw [style=dashededge](2.center) to (5.center);
		\draw (3.center) to (5.center);
            \draw (3.center) to (4.center);
	\end{pgfonlayer}
\end{tikzpicture}}

\newcommand{\MOneFourteen}{
\begin{tikzpicture}[scale = 0.75]
	\begin{pgfonlayer}{nodelayer}
		\node [style=root] (1) at (150:1.00) {};
		\node [style=root] (2) at (210:0.75) {};
            \node [style=root] (3) at (330:0.75) {};
            \node [style=root] (4) at (30:1.00) {};
		\node [style=vertex] (5) at (90:1.00) {};
	\end{pgfonlayer}
	\begin{pgfonlayer}{edgelayer}
		\draw [style=dashededge](1.center) to (2.center);
		\draw [style=dashededge](1.center) to (3.center);
		\draw [style=dashededge](2.center) to (3.center);
		\draw [style=dashededge](1.center) to (4.center);
		\draw (1.center) to (5.center);
		\draw (4.center) to (5.center);
		\draw (2.center) to (4.center);
            \draw [style=dashededge](2.center) to (5.center);
		\draw (3.center) to (5.center);
            \draw (3.center) to (4.center);
	\end{pgfonlayer}
\end{tikzpicture}}

\newcommand{\MOneFifteen}{
\begin{tikzpicture}[scale = 0.75]
	\begin{pgfonlayer}{nodelayer}
		\node [style=root] (1) at (150:1.00) {};
		\node [style=root] (2) at (210:0.75) {};
            \node [style=root] (3) at (330:0.75) {};
            \node [style=root] (4) at (30:1.00) {};
		\node [style=vertex] (5) at (90:1.00) {};
	\end{pgfonlayer}
	\begin{pgfonlayer}{edgelayer}
		\draw [style=dashededge](1.center) to (2.center);
		\draw [style=dashededge](1.center) to (3.center);
		\draw [style=dashededge](2.center) to (3.center);
		\draw [style=dashededge](1.center) to (4.center);
		\draw [style=dashededge](1.center) to (5.center);
		\draw (4.center) to (5.center);
		\draw (2.center) to (4.center);
            \draw (2.center) to (5.center);
		\draw (3.center) to (5.center);
            \draw (3.center) to (4.center);
	\end{pgfonlayer}
\end{tikzpicture}}

\newcommand{\MOneSixteen}{
\begin{tikzpicture}[scale = 0.75]
	\begin{pgfonlayer}{nodelayer}
		\node [style=root] (1) at (150:1.00) {};
		\node [style=root] (2) at (210:0.75) {};
            \node [style=root] (3) at (330:0.75) {};
            \node [style=root] (4) at (30:1.00) {};
		\node [style=vertex] (5) at (90:1.00) {};
	\end{pgfonlayer}
	\begin{pgfonlayer}{edgelayer}
		\draw [style=dashededge](1.center) to (2.center);
		\draw [style=dashededge](1.center) to (3.center);
		\draw [style=dashededge](2.center) to (3.center);
		\draw [style=dashededge](1.center) to (4.center);
		\draw (1.center) to (5.center);
		\draw (4.center) to (5.center);
		\draw (2.center) to (4.center);
            \draw (2.center) to (5.center);
		\draw (3.center) to (5.center);
            \draw (3.center) to (4.center);
	\end{pgfonlayer}
\end{tikzpicture}}
\begin{document}
\setcounter{MaxMatrixCols}{20}

\maketitle

\begin{abstract}
The Ramsey multiplicity problem asks for the minimum asymptotic density of monochromatic labelled copies of a graph $H$ in a red/blue colouring of the edges of $K_n$. We introduce an off-diagonal generalization in which the goal is to minimize a certain weighted sum of the densities of red copies of one graph and blue copies of another. We establish various properties of this new notion, including a useful ``dual formulation,'' and use these results to solve the problem for several pairs of graphs. 
\end{abstract}

\section{Introduction}

The \emph{Ramsey number} of a pair $(H_1,H_2)$ of graphs, denoted by $r(H_1,H_2)$, is the minimum integer $N$ such that every colouring of the edges of the complete graph $K_N$ with red and blue contains a red copy of $H_1$ or a blue copy of $H_2$. The existence of $r(H_1,H_2)$ is implied by a famous theorem of Ramsey from 1929~\cite{Ramsey29}. Determining exact values of Ramsey numbers, even for small graphs, is notoriously difficult; e.g. $r(K_5,K_5)$ is still unknown. For a list of bounds on small Ramsey numbers, see~\cite{Radziszowski94} and, for a survey of asymptotic results, see~\cite{ConlonFoxSudakov15}. 

A natural quantitative extension of this problem is to ask, given $N\gg r(H_1,H_2)$, at least how many red copies of $H_1$ and blue copies of $H_2$ must appear in a red/blue colouring of the edges of $K_N$? To make this question meaningful, we must explain how ``copies'' are counted, and specify the relative weighting on copies of $H_1$ or $H_2$. A \emph{homomorphism} from a graph $H$ to a graph $G$ is a map $f:V(H)\to V(G)$ such that $f(u)f(v)$ is an edge of $G$ whenever $uv$ is an edge of $H$. Let $\hom(H,G)$ be the number of homomorphisms from $H$ to $G$. The \emph{homomorphism density} $t(H,G)$ of $H$ in $G$ is the probability that a random function from $V(H)$ to $V(G)$ is a homomorphism; that is,
\[t(H,G):=\frac{\hom(H,G)}{v(G)^{v(H)}}\]
where $v(F):=|V(F)|$ for every graph $F$. Since there are only $O\left(v(G)^{v(H)-1}\right)$ non-injective functions from $V(H)$ to $V(G)$, the homomorphism density is a good proxy for the number of labelled copies of $H$ in $G$ when $G$ is a large dense graph. The \emph{Ramsey multiplicity constant} (see, e.g.,~\cite{Fox08,Conlon12,GrzesikLeeLidickyVolec22,FoxWigderson23}) of $H$ is defined as follows:
\[c(H) := \liminf_{n\to\infty}\left[t(H,G_n)+ t(H,\overline{G_n})\right]\]
where $G_1,G_2,\dots$ is a sequence containing every finite graph exactly once. By thinking of the edges of $G_n$ and $\overline{G_n}$ as being red and blue, respectively, we can view $c(H)$ as the limit as $n$ tends to infinity of the minimum proportion of monochromatic labelled copies of $H$ over all red/blue colourings of the edges of $K_n$; it is a nice exercise to show that the limit exists. 

Computing $c(H)$ is a difficult problem in general. For example, despite receiving a fair amount of attention~\cite{EvenZoharLinial15,ParczykPokuttaSpiegelSzabo25,GrzesikLeeLidickyVolec22,Niess12,Sperfeld11,Thomason89,Thomason97}, the value of $c(K_4)$ is still unknown. A graph $H$ is said to be \emph{common} if $c(H)=2^{1-e(H)}$, where $e(H):=|E(H)|$. In other words, $H$ is common if its Ramsey multiplicity constant is achieved by a sequence of uniformly random colourings. There are now many families of graphs that are known to be common~\cite{Goodman59,Thomason89,Sidorenko89,Thomason97,JaggerStovicekThomason96,KralVolecWei25,Kral+22,BehagueMorrisonNoel24,GrzesikLeeLidickyVolec22,KoLee23,Hatami+12,LeeNoel26}, but the problem of classifying such graphs is wide open. Fox and Wigderson~\cite{FoxWigderson23} determined $c(H)$ for a wide range of uncommon graphs $H$; prior to this result, there were no uncommon graphs $H$ for which $c(H)$ was known. 

In this paper, we investigate an ``off-diagonal'' extension of Ramsey multiplicity problems to pairs $(H_1,H_2)$ of graphs. Roughly speaking, the goal is to minimize a weighted sum of the homomorphism densities of $H_1$ and $H_2$ in a large graph $G$ and its complement, respectively. The weighting will be chosen in such a way that this minimum density is as large as possible (see Definition~\ref{defn:balancedRM}). This choice is motivated by a desire to establish a ``balance'' between red copies of $H_1$ and blue copies of $H_2$; a nice consequence of this choice is a natural dual formulation of the problem, which we shall discuss in Section~\ref{sec:dual}.

\begin{defn}
\label{defn:lambdaRM}
Given graphs $H_1$ and $H_2$ and $\lambda\in [0,2]$, define the \emph{$\lambda$-Ramsey multiplicity constant} of $(H_1,H_2)$ to be 
\[c_\lambda(H_1,H_2):=\liminf_{n\to\infty}\left[\lambda \cdot t(H_1,G_n)+ (2-\lambda)\cdot t(H_2,\overline{G_n})\right]\]
where $G_1,G_2,\dots$ is a sequence containing every finite graph exactly once.\footnote{An alternative definition of $c_\lambda(H_1,H_2)$ will be provided in Section~\ref{sec:limits} in the language of graph limits.}
\end{defn}

\begin{rem}
\label{rem:swap}
Clearly, $c_\lambda(H_1,H_2)=c_{2-\lambda}(H_2,H_1)$ for any graphs $H_1,H_2$ and $\lambda\in[0,2]$.
\end{rem}

\begin{rem}
\label{rem:c_1(H,H)}
By definition, $c(H)=c_1(H,H)$ for any graph $H$.
\end{rem}

Note that $0\leq t(H,G)\leq1$ for every pair of graphs $H$ and $G$; thus, the following statement holds trivially. 

\begin{obs}
\label{obs:02}
$0\leq c_\lambda(H_1,H_2)\leq2$ for any two graphs $H_1$ and $H_2$ and $\lambda\in[0,2]$. 
\end{obs}

Definition~\ref{defn:lambdaRM}, specialized to the case that $H_1$ and $H_2$ are complete graphs, appeared in a recent paper of Parczyk, Pokutta, Spiegel and Szab\'o~\cite[Equation~(15)]{ParczykPokuttaSpiegelSzabo25}. The results of~\cite{ParczykPokuttaSpiegelSzabo25} include the following:
\[c_1(K_3,K_4)=\frac{689}{3^8}\text{ and }c_1(K_3,K_5)=\frac{24011}{3^{12}}.\]
In both of these cases, the tight examples are colourings based on blow-ups of the $27$-vertex Schl\"afli graph or its complement. Also, in~\cite[Section~4.3]{ParczykPokuttaSpiegelSzabo25}, they mention that the tight colouring for $c_1(K_3,K_4)$ also provides a tight bound on $c_{1-\epsilon}(K_3,K_4)$ for $\epsilon=10^{-4}$.

Outside of a few inherently natural choices for $\lambda$ (e.g. $\lambda=1$), it can be hard to argue that any particular $\lambda$ provides the ``correct'' weighting on copies of $H_1$ and $H_2$. The following definition is an attempt to find a ``natural'' choice of $\lambda$ by taking the supremum of $c_\lambda(H_1,H_2)$ over all possible $\lambda$. This new notion is the central focus of this paper.

\begin{defn}
\label{defn:balancedRM}
Given graphs $H_1$ and $H_2$, define the \emph{balanced Ramsey multiplicity constant} of $(H_1,H_2)$ to be
\[c(H_1,H_2):=\sup_{\lambda\in[0,2]}c_\lambda(H_1,H_2).\]
\end{defn}

The following is a simple consequence of Remark~\ref{rem:swap}. 

\begin{fact}
\label{fact:ch1h2isch2h1}
For any graphs $H_1$ and $H_2$, $c(H_1, H_2) = c(H_2, H_1)$.
\end{fact}

%We should be clear that we are not claiming that choosing $\lambda$ so that $c_\lambda(H_1,H_2)$ is as large as possible is necessarily the ``correct'' choice (or even that there exists one). The recent results of~\cite{ParczykPokuttaSpiegelSzabo25,Behague+22+} demonstrate that there are pairs of graphs for which other choices of $\lambda$ can lead to interesting and illuminating results. Given how little is currently understood about Ramsey multiplicity, we believe that any new insights have the potential to be very valuable. If we want a chance of understanding something as complex as Ramsey multiplicity, we are going to need more ``data points.''

In this paper, we aim to establish some basic properties of the balanced Ramsey multiplicity constant, elucidate connections between it and well-studied problems in extremal graph theory and exactly determine $c(H_1,H_2)$ for several explicit pairs of graphs. The following theorem says that balanced Ramsey multiplicity generalizes usual Ramsey multiplicity, just as one would hope.

\begin{thm}
\label{th:c(H,H)=c(H)}
For any graph $H$, 
\[c(H,H)=c(H).\] 
\end{thm}

A well-known conjecture of Sidorenko~\cite[Conjecture~1]{Sidorenko93}, now known to be equivalent to an earlier conjecture of Erd\H{o}s and Simonovits~\cite[Conjecture~2]{ErdosSimonovits84}, says that, if $H$ is bipartite, then $t(H,G)\geq t(K_2,G)^{e(H)}$ for every graph $G$. A graph $H$ with this property is said to be \emph{Sidorenko}. For background on Sidorenko's Conjecture, see the recent paper~\cite{ConlonLee21} and the references therein. We prove the following theorem, which determines the balanced Ramsey multiplicity constant of every pair of Sidorenko graphs. This will be derived as a corollary of a more general result (Theorem~\ref{th:commonPair}), stated and proven in Section~\ref{sec:applications}. A graph $H$ is said to be \emph{empty} if its edge set is empty.

\begin{thm}
\label{th:Sidorenko}
Let $H_1$ and $H_2$ be non-empty graphs and let $p\in(0,1)$ be such that $p^{e(H_1)}=(1-p)^{e(H_2)}$. If $H_1$ and $H_2$ are Sidorenko, then
\[c(H_1,H_2)=2\cdot p^{e(H_1)}=2\cdot (1-p)^{e(H_2)}.\]
\end{thm}

In other words, Theorem~\ref{th:Sidorenko} says that, if $H_1$ and $H_2$ are Sidorenko, then the optimal construction for $c(H_1,H_2)$ is to colour each edge of $K_n$ red with probability $p$ and blue with probability $1-p$, independently of all other edges, where $p$ is chosen so that $p^{e(H_1)}=(1-p)^{e(H_2)}$. The next theorem shows that random colourings can also be optimal when one of the graphs is not bipartite (and therefore not Sidorenko). For $k\geq3$, let $C_k$ denote the cycle of length $k$. The \emph{banner graph} $B$ is the graph obtained from $C_4$ by adding a pendant edge. See Figure~\ref{fig:diamondMothBanner}.

\begin{thm}
\label{th:C5,B}
$c(C_5,B)=c_{1}(C_5,B)=1/16$.
\end{thm}

In contrast, the next two results concern situations in which the optimal colourings are highly structured. 

\begin{thm}
\label{th:K3,C5}
$c(K_3,C_5)=c_{10/17}(K_3,C_5)=3/34$.
\end{thm}

Let $D$ be the graph obtained from $K_4$ by deleting an edge, which is referred to as the \emph{diamond graph}. Let $M$ be the graph obtained from two disjoint triangles by adding one edge; we call $M$ the \emph{moth graph}.\footnote{The name ``moth graph'' is inspired by the fact that the graph obtained by gluing two triangles on a vertex is often called the \emph{butterfly graph}.} See Figure~\ref{fig:diamondMothBanner}.

\begin{thm}
\label{th:D,M}
$c(D,M)=c_{5/6}(D,M)=1/36$. 
\end{thm}

\begin{figure}[htbp]
\begin{center}
\begin{tikzpicture}[scale=0.6]
 % define vertices with black circles
 \node[circle,fill,inner sep=1.5pt, minimum size=1mm] (A) at (18:1) {};
 \node[circle,fill,inner sep=1.5pt, minimum size=1mm] (B) at (90:1) {};
 \node[circle,fill,inner sep=1.5pt, minimum size=1mm] (C) at (162:1) {};
 \node[circle,fill,inner sep=1.5pt, minimum size=1mm] (D) at (234:1) {};
 \node[circle,fill,inner sep=1.5pt, minimum size=1mm] (E) at (306:1) {};
 % draw edges
 \draw (A) -- (B) -- (C) -- (D)--(E)--(A);
 
 \node (name) at (0,-2) {$C_5$};
\end{tikzpicture}\hspace{3em}
\begin{tikzpicture}[scale=0.6]
 % define vertices with black circles
 \node[circle,fill,inner sep=1.5pt, minimum size=1mm] (A) at (45:1) {};
 \node[circle,fill,inner sep=1.5pt, minimum size=1mm] (B) at (135:1) {};
 \node[circle,fill,inner sep=1.5pt, minimum size=1mm] (C) at (225:1) {};
 \node[circle,fill,inner sep=1.5pt, minimum size=1mm] (D) at (315:1) {};
 \begin{scope}[shift={(225:1)}]
 \node[circle,fill,inner sep=1.5pt, minimum size=1mm] (E) at (270:1.25) {};
 \end{scope}
 % draw edges
 \draw (A) -- (B) -- (C) -- (D)--(A);
 \draw (C)--(E);
 
 \node (name) at (0,-2) {$B$};
\end{tikzpicture}\hspace{3em}
\begin{tikzpicture}[scale=0.6]
 % define vertices with black circles
 \node[circle,fill,inner sep=1.5pt, minimum size=1mm] (A) at (0:1) {};
 \node[circle,fill,inner sep=1.5pt, minimum size=1mm] (B) at (90:1) {};
 \node[circle,fill,inner sep=1.5pt, minimum size=1mm] (C) at (180:1) {};
 \node[circle,fill,inner sep=1.5pt, minimum size=1mm] (D) at (270:1) {};
 % draw edges
 \draw (A) -- (B) -- (C) -- (D)--(A);
 \draw (B)--(D);
 
 \node (D) at (0,-2) {$D$};
\end{tikzpicture}\hspace{3em}
\begin{tikzpicture}[scale=0.6]
 % define vertices with black circles
 \node[circle,fill,inner sep=1.5pt, minimum size=1mm] (A) at (0:1) {};
 \node[circle,fill,inner sep=1.5pt, minimum size=1mm] (B) at (120:1) {};
 \node[circle,fill,inner sep=1.5pt, minimum size=1mm] (C) at (240:1) {};
 \begin{scope}[shift={(3.5,0)}]
 \node[circle,fill,inner sep=1.5pt, minimum size=1mm] (D) at (180:1) {};
 \node[circle,fill,inner sep=1.5pt, minimum size=1mm] (E) at (300:1) {};
 \node[circle,fill,inner sep=1.5pt, minimum size=1mm] (F) at (60:1) {};
 \end{scope}
 % draw edges
 \draw (A) -- (B) -- (C) -- (A);
 \draw (D) -- (E) -- (F) -- (D);
 \draw (A) -- (D);
 
 \node (M) at (1.75,-2) {$M$};
\end{tikzpicture}

\end{center}
 \caption{The cycle $C_5$, the banner graph $B$, the diamond graph $D$ and the moth graph $M$.}
 \label{fig:diamondMothBanner}
\end{figure}
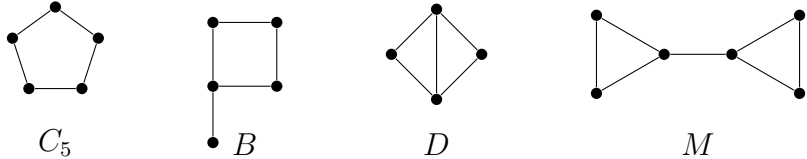

It is interesting to observe differences between the tight examples for Theorems~\ref{th:K3,C5} and~\ref{th:D,M}. To prove the upper bound $c_\lambda(K_3,C_5)\leq 3/34$, we use two different constructions of colourings to cover different ranges of $\lambda$. In contrast, there is a single colouring which proves $c_\lambda(D,M)\leq 1/36$ for all $\lambda\in[0,2]$ simultaneously; see Section~\ref{sec:applications} for details. 

%Given graphs $H_1$ and $H_2$, let $H_1\sqcup H_2$ be the disjoint union of $H_1$ and $H_2$. As it turns out, the balanced Ramsey multiplicity of $C_5$ and $C_4\sqcup K_2$ behaves very similarly to that of $C_5$ and the banner graph. 

%\begin{thm}
%\label{th:C5,C4cupK2}
%$c(C_5,C_4\sqcup K_2)=c_{1}(C_5,C_4\sqcup K_2)=1/16$.
%\end{thm}

In Section~\ref{sec:limits}, we translate the key definitions in this paper into the language of graph limits and establish a few preliminary observations. In Section~\ref{sec:dual}, we obtain an equivalent ``dual formulation'' of the problem of computing $c(H_1,H_2)$ based on finding certificates in the form of two graph limit objects representing red/blue colourings of $E(K_n)$ for large $n$ with certain special properties; see Theorem~\ref{th:dual}. In Section~\ref{sec:applications}, we use this dual formulation to prove Theorems~\ref{th:c(H,H)=c(H)} and~\ref{th:Sidorenko} as well as all of the upper bounds in Theorems~\ref{th:C5,B}--\ref{th:D,M}. In Section~\ref{sec:flags}, we use flag algebras to prove the lower bounds in Theorems~\ref{th:C5,B}--\ref{th:D,M}. Some of the calculations needed to verify these proofs are included in appendices in an ancillary file for arXiv preprint of the paper: \url{https://arxiv.org/src/2306.17388/anc/offDiagonalAppendices.pdf}. We conclude the paper in Section~\ref{sec:concl} by discussing several open problems.

\section{Graph Limits} 
\label{sec:limits}

Many asymptotic notions in extremal graph theory can be elegantly formulated in terms of graph limits. Here, we will introduce only the aspects of graph limit theory that we require in this paper; for a comprehensive treatment of the subject, see~\cite{Lovasz12}. 

A bounded measurable function $U:[0,1]^2\to \mathbb{R}$ such that $U(x,y)=U(y,x)$ for all $(x,y)\in [0,1]^2$ is called a \emph{kernel}. A kernel $W$ such that $0\leq W\leq 1$ is a \emph{graphon}. The set of all graphons is denoted by $\mathcal{W}_0$. One way to think of a graphon is as an analytic extension of the concept of the adjacency matrix of a graph. In particular, if $G$ is a graph with vertices $v_1,\dots,v_n$, then we obtain a graphon $W_G$ associated to $G$ by dividing $[0,1]$ into $n$ intervals $I_1,\dots,I_n$ of measure $1/n$ and setting $W_G=1$ on the set $\bigcup_{v_iv_j\in E(G)}\left(I_i\times I_j\right)$ and $W_G=0$ elsewhere. Likewise, a kernel generalizes the concept of an edge-weighted graph. Because of these analogies, we often refer to an element $x\in[0,1]$ as a \emph{vertex} of a kernel $U$. 

The \emph{homomorphism density} of a graph $H$ in a kernel $U$ is defined by
\[t(H,U):=\int_{[0,1]^{V(H)}}\prod_{uv\in E(H)}U(x_u,x_v)\prod_{v\in V(H)}dx_v.\]
It is easily observed that $t(H,G)=t(H,W_G)$ for every graph $G$, where $W_G$ is a graphon associated to $G$ defined in the previous paragraph. Therefore, the notion of homomorphism density for kernels generalizes homomorphism density for graphs. Next, we introduce a standard notion of convergence for dense graphs, which is sometimes referred to as ``left-convergence'' in the graph limits literature.

\begin{defn}
\label{defn:converges}
Say that a sequence $(W_n)_{n=1}^\infty$ of graphons \emph{converges} to a graphon $W$ if $\lim_{n\to\infty} t(H,W_n)=t(H,W)$ for every graph $H$. A sequence $(G_n)_{n=1}^\infty$ of graphs is said to \emph{converge} to a graphon $W$ if $(W_{G_n})_{n=1}^\infty$ converges to $W$. 
\end{defn}

The following ``compactness'' result is vitally important to the study of graph limits and extremal graph theory. For instance, it is a close relative of the powerful Szemer\'edi Regularity Lemma; see~\cite{LovaszSzegedy07}.

\begin{thm}[Lov\'asz and Szegedy~\cite{LovaszSzegedy06}]
\label{th:LS}
For any sequence $(W_n)_{n=1}^\infty$ of graphons, there is a subsequence $(W_{n_k})_{k=1}^\infty$ and a graphon $W$ such that $(W_{n_k})_{k=1}^\infty$ converges to $W$. 
\end{thm}

The next result tells us that graphons corresponding to finite graphs are ``dense'' in $\mathcal{W}_0$.

\begin{thm}[See~{\cite[Chapter~10]{Lovasz12}}]
\label{th:LSconverse}
For every graphon $W$ there is a sequence $(G_n)_{n=1}^\infty$ of finite graphs such that $v(G_n)\to\infty$ and $(G_n)_{n=1}^\infty$ converges to $W$.
\end{thm}

We state another easy consequence of standard results in graph limit theory.

\begin{lem}[See, e.g.,~{\cite[Lemma~10.23]{Lovasz12}}]
\label{lem:complement}
If $(G_n)_{n=1}^\infty$ is a sequence of graphs which converges to a graphon $W$ such that $v(G_n)\to\infty$, then $(\overline{G_n})_{n=1}^\infty$ converges to $1-W$.
\end{lem}

Our next goal is to rephrase the definitions from the introduction in terms of graph limits and to use this perspective to derive some basic properties. We start with the following alternative definition of $c_\lambda(H_1,H_2)$. 

\begin{lem}
\label{lem:altDef}
For any graphs $H_1$ and $H_2$ and any $\lambda\in [0,2]$, 
\[c_\lambda(H_1,H_2)=\min_{W\in\mathcal{W}_0}\left[\lambda\cdot t(H_1,W) + (2-\lambda)\cdot t(H_2,1-W)\right].\]
\end{lem}

\begin{proof}
By definition of $c_\lambda(H_1,H_2)$, there exists a sequence $(G_n)_{n=1}^\infty$ of graphs such that $v(G_n)\to\infty$ and 
\[\lim_{n\to\infty} [\lambda\cdot t(H_1,G_n) + (2-\lambda)\cdot t(H_2,\overline{G_n})]=c_\lambda(H_1,H_2).\]
By Theorem~\ref{th:LS}, there is a subsequence $(G_{n_k})_{k=1}^\infty$ and a graphon $W$ such that $(G_{n_k})_{k=1}^\infty$ converges to $W$. By Lemma~\ref{lem:complement}, $\overline{G_{n_k}}$ converges to $1-W$. Therefore, 
\[c_\lambda(H_1,H_2)=\lim_{n\to\infty} [\lambda\cdot t(H_1,G_n) + (2-\lambda)\cdot t(H_2,\overline{G_n})]\]
\[= \lim_{k\to\infty} [\lambda\cdot t(H_1,G_{n_k}) + (2-\lambda)\cdot t(H_2,\overline{G_{n_k}})] = \lambda \cdot t(H_1,W)+(2-\lambda)\cdot t(H_2,1-W).\]
So, there exists a graphon $W$ such that
\[\lambda \cdot t(H_1,W)+(2-\lambda)\cdot t(H_2,1-W)=c_\lambda(H_1,H_2).\]
All that remains is to show that there cannot exist a graphon $W'$ such that
\[\lambda \cdot t(H_1,W')+(2-\lambda)\cdot t(H_2,1-W')<c_\lambda(H_1,H_2).\]
To see this, we use Theorem~\ref{th:LSconverse}. That is, if such a $W'$ existed, then we could take $G_1',G_2',\dots$ to be a sequence of graphs of increasing orders which converges to $W'$. By Lemma~\ref{lem:complement}, the sequence of complements of these graphs converges to $1-W'$. However, then we would have
\[\lim_{n\to\infty} [\lambda\cdot t(H_1,G_n') + (2-\lambda)\cdot t(H_2,\overline{G_n'})] = \lambda \cdot t(H_1,W')+(2-\lambda)\cdot t(H_2,1-W')<c_\lambda(H_1,H_2)\]
which contradicts the definition of $c_\lambda(H_1,H_2)$ and completes the proof.
\end{proof}

The previous lemma suggests the following definition.

\begin{defn}
Let $H_1$ and $H_2$ be graphs and $\lambda\in [0,2]$. We say that a graphon $W$ is \emph{$\lambda$-optimal} for $(H_1,H_2)$ if 
\[\lambda\cdot t(H_1,W) + (2-\lambda)\cdot t(H_2,1-W) = c_\lambda(H_1,H_2).\]
\end{defn}

Next, we show that $c_\lambda(H_1,H_2)$ is continuous when viewed as a function of $\lambda$.

\begin{lem}
\label{lem:cont}
For any two graphs $H_1$ and $H_2$, $c_\lambda(H_1,H_2)$ is a continuous function of $\lambda$. 
\end{lem}

\begin{proof}
Suppose, to the contrary, that there exists $\lambda,\lambda_1,\lambda_2,\ldots\in [0,2]$ and $\varepsilon>0$ such that $(\lambda_n)_{n=1}^\infty$ converges to $\lambda$ and $|c_{\lambda_n}(H_1,H_2)-c_\lambda(H_1,H_2)|\geq \varepsilon$ for all $n\geq1$. By Lemma~\ref{lem:altDef}, for each $n\geq1$ we can let $W_n$ be a $\lambda_n$-optimal graphon. By Theorem~\ref{th:LS}, there is a subsequence $(W_{n_k})_{k=1}^\infty$ and a graphon $W$ such that $(W_{n_k})_{k=1}^\infty$ converges to $W$. By Observation~\ref{obs:02} and compactness of $[0,2]$, we can additionally assume that $\lim_{k\to\infty}c_{\lambda_{n_k}}(H_1,H_2)$ exists. Since $W_{n_k}$ is $\lambda_{n_k}$-optimal, we have that
\begin{equation}\label{eq:clambda1}
\begin{gathered}
\lim_{k\to\infty} c_{\lambda_{n_k}}(H_1,H_2) = \lim_{k\to\infty} \lambda_{n_k}t(H_1,W_{n_k})+(2-\lambda_{n_k})t(H_2,1-W_{n_k})\\= \lambda\cdot t(H_1,W) + (2-\lambda)t(H_2,1-W)\geq c_\lambda(H_1,H_2)\end{gathered}\end{equation}
where the last inequality holds due to Lemma~\ref{lem:altDef}.

Now, by Lemma~\ref{lem:altDef}, we can let $W'$ be a $\lambda$-optimal graphon. Applying Lemma~\ref{lem:altDef} again, we see that 
\[c_{\lambda_{n_k}}(H_1,H_2)\leq \lambda_{n_k}\cdot t(H_1,W')+(2-\lambda_{n_k})\cdot t(H_2,1-W')\]
for all $k\geq1$. So, 
\begin{equation}
\label{eq:clambda2}
\begin{gathered}
\lim_{k\to\infty}c_{\lambda_{n_k}}(H_1,H_2)\leq \lim_{k\to\infty}\left(\lambda_{n_k}\cdot t(H_1,W')+(2-\lambda_{n_k})\cdot t(H_2,1-W')\right)\\
=\lambda\cdot t(H_1,W')+(2-\lambda)\cdot t(H_2,1-W') = c_\lambda(H_1,H_2).
\end{gathered}
\end{equation}
Putting \eqref{eq:clambda1} and \eqref{eq:clambda2} together, we get that 
\[\lim_{k\to\infty}c_{\lambda_{n_k}}(H_1,H_2) = c_\lambda(H_1,H_2).\]
However, this contradicts the assumption that $|c_{\lambda_n}(H_1,H_2)-c_\lambda(H_1,H_2)|\geq \varepsilon$ for all $n\geq1$ and completes the proof.
\end{proof}

Thus, by Lemma~\ref{lem:cont} and the Extreme Value Theorem, we get that the supremum in Definition~\ref{defn:balancedRM} can be replaced by a maximum.

\begin{cor}
\label{cor:max} 
For any graphs $H_1$ and $H_2$, 
\[c(H_1,H_2)=\max_{\lambda\in[0,2]}c_\lambda(H_1,H_2).\]
\end{cor}

Next, we derive a lower bound on $c_\lambda(H_1,H_2)$ via a standard double-counting argument.

\begin{lem}
\label{lem:trivialLower}
For any non-empty graphs $H_1$ and $H_2$ and $\lambda\in[0,2]$,
\[c_\lambda(H_1,H_2)\geq \frac{\min\{\lambda,2-\lambda\}\cdot(r(H_1,H_2)-\max\{v(H_1),v(H_2)\})!}{r(H_1,H_2)!}.\]
\end{lem}

\begin{proof}
Let $G_1,G_2,\dots$ be a sequence of graphs such that $v(G_n)\to\infty$. By definition of $r(H_1,H_2)$, for every set $S\subseteq V(G_n)$ of cardinality $r(H_1,H_2)$, either the subgraph of $G_n$ induced by $S$, denoted $G_n[S]$, contains a labelled copy of $H_1$ or $\overline{G_n}[S]$ contains a labelled copy of $H_2$. Letting $N_{n,1}$ be the number of such sets $S$ containing a copy of $H_1$ in $G_n[S]$ and $N_{n,2}$ be the number containing a copy of $H_2$ in $\overline{G_n}[S]$, we have that
\begin{equation}\label{eq:overcount}N_{n,1} + N_{n,2}\geq \binom{v(G_n)}{r(H_1,H_2)}.\end{equation}
The quantity $N_{n,1}$ counts every labelled copy of $H_1$ in $G_n$ at most $\binom{v(G_n)-v(H_1)}{r(H_1,H_2)-v(H_1)}$ times. Thus, the total number of copies of $H_1$ in $G_n$ is at least
\[\frac{N_{n,1}}{\binom{v(G_n)-v(H_1)}{r(H_1,H_2)-v(H_1)}}\geq\frac{N_{n,1}(r(H_1,H_2)-v(H_1))!}{v(G_n)^{r(H_1,H_2)-v(H_1)}}\]
and, likewise, the total number of copies of $H_2$ in $\overline{G_n}$ is at least
\[\frac{N_{n,2}}{\binom{v(G_n)-v(H_2)}{r(H_1,H_2)-v(H_2)}}\geq \frac{N_{n,2}(r(H_1,H_2)-v(H_2))!}{v(G_n)^{r(H_1,H_2)-v(H_2)}}.\]
Thus, letting $h=\max\{v(H_1),v(H_2)\}$ and $\ell:=\min\{\lambda,2-\lambda\}$, we get
\[\lambda\cdot t(H_1,G_n) + (2-\lambda)\cdot t(H_2,\overline{G_n})\]
\[\geq \frac{\lambda\cdot N_{n,1}\cdot(r(H_1,H_2)-v(H_1))! + (2-\lambda)\cdot N_{n,2}\cdot(r(H_1,H_2)-v(H_2))!}{v(G_n)^{r(H_1,H_2)}}\]
\[\geq \frac{\ell\cdot \left(N_{n,1}+N_{n,2}\right)\cdot (r(H_1,H_2)-h)! }{v(G_n)^{r(H_1,H_2)}}\]
By \eqref{eq:overcount}, we can bound this quantity below by 
\[\frac{\ell\cdot \binom{v(G_n)}{r(H_1,H_2)}(r(H_1,H_2)-h)!}{v(G_n)^{r(H_1,H_2)}}.\]
The limit of this expression as $v(G_n)$ tends to infinity is $\frac{\ell\cdot (r(H_1,H_2)-h)!}{r(H_1,H_2)!}$. Thus, 
\[c_\lambda(H_1,H_2)\geq\frac{\ell\cdot(r(H_1,H_2)-h)!}{r(H_1,H_2)!}.\]
\end{proof}

The following proposition is useful for ruling out the ``extreme points'' $\lambda=0$ and $\lambda=2$ when computing $c(H_1,H_2)$.

\begin{prop}
\label{prop:nonempty}
Let $H_1$ and $H_2$ be graphs. Then $c(H_1,H_2)=c_0(H_1,H_2)$ if and only if $H_2$ is empty. 
\end{prop}

\begin{proof}
If $H_2$ is empty, then $t(H_2,1-W)=1$ for every graphon $W$. Thus, 
\[c(H_1,H_2)\geq c_0(H_1,H_2)=2\cdot \min_{W\in\mathcal{W}_0}t(H_2,1-W) = 2.\]
Also, $c(H_1,H_2)\leq 2$ by Observation~\ref{obs:02}. Thus, $c(H_1,H_2)=c_0(H_1,H_2)=2$ if $H_2$ is empty. 

If $H_1$ is empty and $H_2$ is not empty, then $c_2(H_1,H_2)=2$, while, as above, $c_0(H_1,H_2)=0$. Hence $c(H_1,H_2)\neq c_0(H_1,H_2)$.

On the other hand, if both $H_1$ and $H_2$ are non-empty, then, by taking $W:=1$, we see that
\[c_0(H_1,H_2)\leq 2\cdot t(H_2,1-W) = 0.\]
However, by Lemma~\ref{lem:trivialLower}, we have $c(H_1,H_2)\geq c_1(H_1,H_2)>0$. Thus, $c(H_1,H_2)\neq c_0(H_1,H_2)$ and the result follows. 
\end{proof}

\section{A Dual Perspective}
\label{sec:dual}

By Corollary~\ref{cor:max}, $c(H_1,H_2)$ is the maximum over all $\lambda$ of $c_\lambda(H_1,H_2)$. Our goal in this section is to show how computing $c(H_1,H_2)$ can be equivalently phrased as a minimization problem over pairs of graphons representing edge colourings of cliques. The key definition and theorem of this section are as follows.

\begin{defn}
\label{defn:certificate}
Let $H_1$ and $H_2$ be graphs and let $\alpha\in\mathbb{R}$. Say that a pair $(W_1,W_2)$ is an \emph{$\alpha$-certificate} for $(H_1,H_2)$ if there exists $\lambda\in[0,2]$ such that the following four bounds hold:
\begin{enumerate}
\stepcounter{equation}
 \item\label{eq:ralpha} $\lambda\cdot t(H_1,W_1) + (2-\lambda)\cdot t(H_2,1-W_1) \leq \alpha$,
\stepcounter{equation}
 \item\label{eq:balpha} $\lambda\cdot t(H_1,W_2) + (2-\lambda)\cdot t(H_2,1-W_2) \leq \alpha$,
\stepcounter{equation}
 \item\label{eq:rbigger}$t(H_1,W_1)\geq t(H_2,1-W_1)$ and
\stepcounter{equation}
 \item\label{eq:bbigger} $t(H_1,W_2)\leq t(H_2,1-W_2)$.
\end{enumerate}
\end{defn}

\begin{thm}
\label{th:dual}
Let $H_1$ and $H_2$ be non-empty graphs. Then $c(H_1,H_2)$ is equal to the minimum over all $\alpha$ for which there exists an $\alpha$-certificate for $(H_1,H_2)$. 
\end{thm}

\begin{proof}
First, we show that, if there exists an $\alpha$-certificate $(W_1,W_2)$ for $(H_1,H_2)$, then $c(H_1,H_2)\leq \alpha$. Let $\lambda\in [0,2]$ so that \eqref{eq:ralpha} and \eqref{eq:balpha} both hold and let $\gamma\in[0,2]$. First, if $\gamma\in [0,\lambda]$, then
\[c_\gamma(H_1,H_2) \leq \gamma\cdot t(H_1,W_1) + (2-\gamma)\cdot t(H_2,1-W_1)\]
\[=\lambda\cdot t(H_1,W_1) + (2-\lambda)\cdot t(H_2,1-W_1) + (\gamma-\lambda)\cdot \left(t(H_1,W_1) -t(H_2,1-W_1)\right)\]
which is at most $\alpha$ by \eqref{eq:ralpha}, \eqref{eq:rbigger} and the fact that $\gamma\leq \lambda$. The case $\gamma\in[\lambda,2]$ is analogous, but with $W_2$ in the place of $W_1$.

Now, we let $\alpha:=c(H_1,H_2)$ and show that there exists an $\alpha$-certificate for $(H_1,H_2)$. By Corollary~\ref{cor:max}, we can choose $\lambda\in[0,2]$ so that $c_\lambda(H_1,H_2)=c(H_1,H_2)$. By Proposition~\ref{prop:nonempty}, Remark~\ref{rem:swap}, Fact~\ref{fact:ch1h2isch2h1} and that $H_1$ and $H_2$ are non-empty, we can assume that $\lambda\in (0,2)$. Now, let $\lambda_n$ be a sequence in $(0,2)$ which converges to $\lambda$ from the right. For each $n$, by Lemma~\ref{lem:altDef}, we can let $W_{2,n}$ be a graphon such that
\[\lambda_n\cdot t(H_1,W_{2,n}) + (2-\lambda_n)\cdot t(H_2,1-W_{2,n})=c_{\lambda_n}(H_1,H_2).\] 
Note that
\[c_{\lambda}(H_1, H_2) \leq \lambda\cdot t(H_1,W_{2,n}) + (2-\lambda)\cdot t(H_2,1-W_{2,n})\] 
\[=\lambda_n\cdot t(H_1,W_{2,n}) + (2-\lambda_n)\cdot t(H_2,1-W_{2,n}) + (\lambda-\lambda_n)(t(H_1,W_{2,n}) - t(H_2,1-W_{2,n}))\]
\[=c_{\lambda_n}(H_1,H_2) + (\lambda-\lambda_n)(t(H_1,W_{2,n}) - t(H_2,1-W_{2,n}))\]
\[\leq c_\lambda(H_1,H_2)+ (\lambda-\lambda_n)(t(H_1,W_{2,n}) - t(H_2,1-W_{2,n})),\] where in the last line we used that $c_{\lambda_n}(H_1, H_2) \leq c(H_1, H_2) = c_{\lambda}(H_1, H_2)$.
Therefore, by definition of $c_\lambda(H_1,H_2)$ and the fact that $\lambda_n>\lambda$, we must have that $t(H_2,1-W_{2,n})\geq t(H_1,W_{2,n})$. By Theorem~\ref{th:LS}, we can take a subsequence $(W_{2,n_k})_{k=1}^\infty$ of $(W_{2,n})_{n=1}^\infty$ that converges to a graphon $W_{2}$. By convergence of $(W_{2,n_k})_{k=1}^\infty$ and by Lemma~\ref{lem:cont}, we have
\[\lambda\cdot t(H_1,W_2)+(2-\lambda)\cdot t(H_2,1-W_2) = \lim_{k\to\infty}c_{\lambda_{n_k}}(H_1,H_2) = c_\lambda(H_1,H_2).\]
Since $t(H_2,1-W_{2,n})\geq t(H_1,W_{2,n})$ for all $n$, we must have $t(H_2,1-W_{2})\geq t(H_1,W_{2})$, and so \eqref{eq:bbigger} holds. The construction of $W_1$ is similar, except that we take $(\lambda_n)_{n=1}^\infty$ to be a sequence converging to $\lambda$ from the left. 
\end{proof}

In light of Theorem~\ref{th:dual}, we make the following definition.

\begin{defn}
If $(W_1,W_2)$ is a $c(H_1,H_2)$-certificate for a pair $(H_1,H_2)$ of graphs, then we say that $(W_1,W_2)$ is an \emph{optimal certificate} for $(H_1,H_2)$.
\end{defn}

\section{Applying Duality}
\label{sec:applications}

Our next goal is to provide several applications of Theorem~\ref{th:dual}. We start with a proof of Theorem~\ref{th:c(H,H)=c(H)}.

\begin{proof}[Proof of Theorem~\ref{th:c(H,H)=c(H)}]
If $H$ is empty, then $t(H,W)=1$ for every graphon $W$, so $c(H)=2$ and $c(H,H)=2$. Hence we may assume that $H$ is non-empty.

By definition of $c(H,H)$ and Remark~\ref{rem:c_1(H,H)}, we have
\[c(H,H)\geq c_1(H,H)=c(H).\]
So, all that remains is to prove $c(H,H)\leq c(H)$. By Lemma~\ref{lem:altDef} and Remark~\ref{rem:c_1(H,H)}, we can let $W$ be a graphon such that
\[t(H,W) + t(H,1-W) = c_1(H,H)=c(H).\]
Without loss of generality, $t(H,W)\geq t(H,1-W)$; otherwise, simply replace $W$ with $1-W$. Thus, the pair $(W,1-W)$ is a $c(H)$-certificate for $(H,H)$ and so we are done by Theorem~\ref{th:dual}. 
\end{proof}

We now prove a general upper bound on $c(H_1,H_2)$ coming from random colourings in which each edge is red with density $p$ and blue with density $1-p$ and characterize the case of equality. The characterization involves the following off-diagonal generalization of the notion of common graphs introduced recently in~\cite{BehagueMorrisonNoel25}.

\begin{defn}
For $p\in (0,1)$, a pair $(H_1,H_2)$ of non-empty graphs is said to be \emph{$(p,1-p)$-common} if the following inequality holds for every graphon $W$:
\[\frac{t(H_1,W)}{e(H_1) p^{e(H_1)-1}} + \frac{t(H_2,1-W)}{e(H_2) (1-p)^{e(H_2)-1}}\geq \frac{p}{e(H_1)}+\frac{1-p}{e(H_2)}.\]
\end{defn}

\begin{thm}
\label{th:commonPair}
Let $H_1$ and $H_2$ be non-empty graphs and let $p\in(0,1)$ such that $p^{e(H_1)}=(1-p)^{e(H_2)}$. Then
\[c(H_1,H_2)\leq 2\cdot p^{e(H_1)}=2\cdot(1-p)^{e(H_2)}\]
with equality if and only if $(H_1,H_2)$ is $(p,1-p)$-common. 
\end{thm}

\begin{proof}
Let $H_1$ and $H_2$ be arbitrary non-empty graphs. For the upper bound, we show that the pair $(W_1,W_2)$ where $W_1=W_2=p$ is a $2p^{e(H_1)}$-certificate for $(H_1,H_2)$ and apply Theorem~\ref{th:dual}. Since $p^{e(H_1)}=(1-p)^{e(H_2)}$, we have
\[t(H_1,W_1)=t(H_1,W_2)=p^{e(H_1)}=(1-p)^{e(H_2)}=t(H_2,1-W_1)=t(H_2,1-W_2)\]
and so \eqref{eq:rbigger} and \eqref{eq:bbigger} both hold. Using the fact that $p^{e(H_1)}=(1-p)^{e(H_2)}$ again, for any $\lambda\in[0,2]$, we have
\[\lambda\cdot t(H_1,W_1) + (2-\lambda)\cdot t(H_2,1-W_1)\]
\[=\lambda\cdot t(H_1,W_2) + (2-\lambda)\cdot t(H_2,1-W_2)= \lambda\cdot p^{e(H_1)}+ (2-\lambda)(1-p)^{e(H_2)} = 2\cdot p^{e(H_1)}\]
and so \eqref{eq:ralpha} and \eqref{eq:balpha} hold for $\alpha=2p^{e(H_1)}$. Thus, $(W_1,W_2)$ is a $2p^{e(H_1)}$-certificate for $(H_1, H_2)$ and we are done with the upper bound. 

What remains is to characterize the case of equality. Define
\begin{equation}\label{eq:commonlambda}\lambda_0 :=2\left( \frac{e(H_2)(1-p)^{e(H_2)-1}}{e(H_1)p^{e(H_1)-1} + e(H_2)(1-p)^{e(H_2)-1}}\right)\end{equation}
and
\begin{equation}\label{eq:commonsigma}\sigma_0 := \frac{1}{2}\left(\frac{e(H_1)p^{e(H_1)-1} + e(H_2)(1-p)^{e(H_2)-1}}{\left(e(H_1)p^{e(H_1)-1}\right)\cdot\left(e(H_2)(1-p)^{e(H_2)-1}\right)}\right).\end{equation}
Note that
\begin{equation}\label{eq:lambdasigma1}\sigma_0\cdot\lambda_0 = \frac{1}{e(H_1)p^{e(H_1)-1}}\end{equation}
and
\begin{equation}\label{eq:lambdasigma2}\sigma_0\cdot(2-\lambda_0) = \frac{1}{e(H_2)(1-p)^{e(H_2)-1}}.\end{equation}
Therefore, $(H_1,H_2)$ is $(p,1-p)$-common if and only if 
\[\sigma_0\cdot\lambda_0\cdot t(H_1,W) + \sigma_0\cdot(2-\lambda_0)\cdot t(H_2,1-W)\geq \frac{p}{e(H_1)}+\frac{1-p}{e(H_2)}\]
or, equivalently, 
\[\lambda_0\cdot t(H_1,W) + (2-\lambda_0)\cdot t(H_2,1-W)\geq \frac{1}{\sigma_0}\left(\frac{p}{e(H_1)}+\frac{1-p}{e(H_2)}\right)\]
for every graphon $W$. Now, observe that 
\[\frac{1}{\sigma_0}\left(\frac{p}{e(H_1)}+\frac{1-p}{e(H_2)}\right) = \frac{2e(H_1)p^{e(H_1)-1}e(H_2)(1-p)^{e(H_2)-1}}{e(H_1)p^{e(H_1)-1} + e(H_2)(1-p)^{e(H_2)-1}}\left(\frac{p}{e(H_1)}+\frac{1-p}{e(H_2)}\right)\]
\[ = \frac{2e(H_1)p^{e(H_1)-1}e(H_2)(1-p)^{e(H_2)-1}}{e(H_1)p^{e(H_1)-1} + e(H_2)(1-p)^{e(H_2)-1}}\left(\frac{e(H_2)p+e(H_1)(1-p)}{e(H_1)e(H_2)}\right)\]
\[ = \frac{2e(H_2)p^{e(H_1)}(1-p)^{e(H_2)-1}+2e(H_1)p^{e(H_1)-1}(1-p)^{e(H_2)}}{e(H_1)p^{e(H_1)-1} + e(H_2)(1-p)^{e(H_2)-1}}\]
which, since $(1-p)^{e(H_2)}=p^{e(H_1)}$, is equal to $2\cdot p^{e(H_1)}$. So, $(H_1,H_2)$ is $(p,1-p)$-common if and only if 
\[\lambda_0\cdot t(H_1,W) + (2-\lambda_0)\cdot t(H_2,1-W)\geq2\cdot p^{e(H_1)}\]
for every graphon $W$. In particular, if $(H_1,H_2)$ is $(p,1-p)$-common, then
\[c(H_1,H_2)\geq c_{\lambda_0}(H_1,H_2)\geq 2\cdot p^{e(H_1)}\]
which proves one direction of the characterization. 

Now, suppose that $(H_1,H_2)$ is not $(p,1-p)$-common. Our goal is to find an $\alpha$-certificate for $(H_1,H_2)$ with $\alpha<2\cdot p^{e(H_1)}$. Using the characterization above, let $W$ be a graphon such that
\[
\lambda_0\cdot t(H_1,W)+(2-\lambda_0)\cdot t(H_2,1-W)<2\cdot p^{e(H_1)}.
\]
For $\gamma\in[0,2]$, define
\[
g_\gamma(x):=\gamma\cdot (p+x)^{e(H_1)}+(2-\gamma)\cdot(1-p-x)^{e(H_2)}.
\]
By the definition of $\lambda_0$, we have $g'_\gamma(0)>0$ when $\gamma>\lambda_0$ and $g'_\gamma(0)<0$ when $\gamma<\lambda_0$.

First suppose that $t(H_1,W)\geq t(H_2,1-W)$. Choose $\gamma>\lambda_0$ sufficiently close to $\lambda_0$ so that
\[
\gamma\cdot t(H_1,W)+(2-\gamma)\cdot t(H_2,1-W)<2\cdot p^{e(H_1)}.
\]
Set $W_1:=W$ and
\[
\alpha_1:=\gamma\cdot t(H_1,W_1)+(2-\gamma)\cdot t(H_2,1-W_1).
\]
Since $g'_\gamma(0)>0$, we may choose $\varepsilon>0$ sufficiently small so that $p-\varepsilon>0$ and $
g_\gamma(-\varepsilon)<g_\gamma(0)=2p^{e(H_1)}$. Set $W_2:=p-\varepsilon$ and $\alpha_2:=g_\gamma(-\varepsilon)$. Then
\[
t(H_1,W_2)=(p-\varepsilon)^{e(H_1)}<(1-p+\varepsilon)^{e(H_2)}=t(H_2,1-W_2).
\]
Thus $(W_1,W_2)$ is an $\alpha$-certificate with $\alpha:=\max\{\alpha_1,\alpha_2\}<2p^{e(H_1)}$. 

Now suppose instead that $t(H_1,W)<t(H_2,1-W)$. Choose $\gamma<\lambda_0$ sufficiently close to $\lambda_0$ so that
\[
\gamma\cdot t(H_1,W)+(2-\gamma)\cdot t(H_2,1-W)<2\cdot p^{e(H_1)}.
\]
Set $W_2:=W$ and
\[
\alpha_2:=\gamma\cdot t(H_1,W_2)+(2-\gamma)\cdot t(H_2,1-W_2).
\]
Since $g'_\gamma(0)<0$, we may choose $\varepsilon>0$ sufficiently small so that $p+\varepsilon<1$ and $g_\gamma(\varepsilon)<g_\gamma(0)=2p^{e(H_1)}$. Set $W_1:=p+\varepsilon$ and $\alpha_1:=g_\gamma(\varepsilon)$. Then
\[
t(H_1,W_1)=(p+\varepsilon)^{e(H_1)}>(1-p-\varepsilon)^{e(H_2)}=t(H_2,1-W_1).
\]
Thus $(W_1,W_2)$ is an $\alpha$-certificate with $\alpha:=\max\{\alpha_1,\alpha_2\}<2p^{e(H_1)}$. In both cases, Theorem~\ref{th:dual} gives $c(H_1,H_2)<2p^{e(H_1)}$, as desired.
\end{proof}

Theorem~\ref{th:Sidorenko} is now a consequence of Theorem~\ref{th:commonPair} and the following statement, which recently appeared in~\cite{BehagueMorrisonNoel25}. 

\begin{thm}
\label{th:SidCommon}
If $H_1$ and $H_2$ are non-empty and Sidorenko, then $(H_1,H_2)$ is $(p,1-p)$-common for all $p\in (0,1)$. 
\end{thm}

\begin{proof}
Let $H_1$ and $H_2$ be Sidorenko, let $p\in (0,1)$ and let $W$ be a graphon. Define $x:=t(K_2,W)-p$. Then
\[\frac{t(H_1,W)}{e(H_1)p^{e(H_1)-1}}+\frac{t(H_2,1-W)}{e(H_2)(1-p)^{e(H_2)-1}}\geq\frac{(p+x)^{e(H_1)}}{e(H_1)p^{e(H_1)-1}}+\frac{(1-p-x)^{e(H_2)}}{e(H_2)(1-p)^{e(H_2)-1}}\]
\[=\frac{p^{e(H_1)}(1+xp^{-1})^{e(H_1)}}{e(H_1)p^{e(H_1)-1}}+\frac{(1-p)^{e(H_2)}(1-x(1-p)^{-1})^{e(H_2)}}{e(H_2)(1-p)^{e(H_2)-1}}\]
\[=\frac{p(1+xp^{-1})^{e(H_1)}}{e(H_1)}+\frac{(1-p)(1-x(1-p)^{-1})^{e(H_2)}}{e(H_2)}.\]
Using Bernoulli's Inequality (i.e. the fact that $(1+y)^r\geq 1+ry$ for all $r\geq 1$ and $y\geq-1$) on each term of the above expression and cancelling yields $\frac{p}{e(H_1)}+\frac{1-p}{e(H_2)}$, as desired.
\end{proof}

\begin{proof}[Proof of Theorem~\ref{th:Sidorenko}]
Let $H_1$ and $H_2$ be non-empty Sidorenko graphs and let $p\in (0,1)$ such that $p^{e(H_1)}=(1-p)^{e(H_2)}$. By Theorem~\ref{th:SidCommon}, $(H_1,H_2)$ is $(p,1-p)$-common. Thus, the result follows by Theorem~\ref{th:commonPair}.
\end{proof}

We now turn our attention to the upper bounds in Theorems~\ref{th:C5,B}--\ref{th:D,M}. The upper bound in Theorem~\ref{th:C5,B} follows easily from Theorem~\ref{th:commonPair}. 

\begin{prop}
\label{prop:C5,Bupper}
$c(C_5,B)\leq 1/16$.
\end{prop}

\begin{proof}
The graphs $C_5$ and $B$ both have five edges; thus, $p^{e(C_5)}=(1-p)^{e(B)}$ if and only if $p=1/2$. So, by Theorem~\ref{th:commonPair}, $c(C_5,B)\leq 2(1/2)^5=1/16$.
\end{proof}

Next, we prove the upper bounds in Theorems~\ref{th:K3,C5} and~\ref{th:D,M}. In each case, the proof will simply boil down to exhibiting an $\alpha$-certificate $(W_1,W_2)$ for an appropriate choice of $\alpha$, verifying that it satisfies the properties outlined in Definition~\ref{defn:certificate} and concluding with an application of Theorem~\ref{th:dual}. After giving the examples, we will include some further discussion to provide insight into how we found them.

\begin{prop}
\label{prop:K3,C5upper}
$c(K_3,C_5)\leq 3/34$.
\end{prop}

\begin{proof}
Define $\alpha=3/34$ and $\lambda=10/17$. Our goal is to find an $\alpha$-certificate $(W_1,W_2)$ for $(K_3,C_5)$. Let $W_2=W_{K_2}$. We have $t(K_3,W_2)=0$ and $t(C_5,1-W_2)=(1/2)^4$. Therefore, \eqref{eq:bbigger} holds. Also,
\[\lambda\cdot t(K_3,W_2) + (2-\lambda)\cdot t(C_5,1-W_2)=(10/17)\cdot 0 + (24/17)\cdot(1/16) = 3/34\]
and so \eqref{eq:balpha} holds. 

Now, let $W_1=W_{\overline{C_6}}$. A standard fact from algebraic graph theory is that, if $G$ is an $n$-vertex graph and $\lambda_1,\dots,\lambda_n$ are the eigenvalues of its adjacency matrix, then
\[t(C_k,G)=\sum_{i=1}^n\left(\frac{\lambda_i}{n}\right)^k\]
for any $k\geq3$; see, e.g.,~\cite[Example~5.11]{Lovasz12}. The eigenvalues of the adjacency matrix of $\overline{C_6}$ are $3,-2,-2,1,0,0$. Therefore, 
\[t(K_3,W_1)=(3/6)^3 + (-2/6)^3 + (-2/6)^3+(1/6)^3 = 1/18.\]
Similarly, 
\[t(C_5,1-W_1)=(3/6)^5 + (2/6)^5 + (2/6)^5+(-1/6)^5 = 17/432.\]
This is because the matrix obtained from subtracting the adjacency matrix of $\overline{C_6}$ from the $6\times 6$ all-ones matrix has eigenvalues $3,2,2,-1,0,0$. Note that $1/18>17/432$ and so \eqref{eq:rbigger} holds. Also, 
\[\lambda\cdot (1/18) + (2-\lambda)\cdot (17/432) = (10/17)\cdot (1/18)+(24/17)\cdot(17/432)=3/34.\]
So, \eqref{eq:ralpha} holds as well. The proposition now follows from Theorem~\ref{th:dual}.
\end{proof}

It is worth noting that the two tight colourings described in the proof of Proposition~\ref{prop:K3,C5upper} seem to be far from the only such colourings. For example, if $H$ is the $8$-vertex graph obtained from $\overline{C_6}$ by duplicating two adjacent vertices that are not in a common triangle, then $(10/17)t(K_3,W_H) + (24/17)t(C_5,1-W_H)=3/34$ as well; see Figure~\ref{fig:anotherExample}.

\begin{figure}[htbp]
\begin{center}
\begin{tikzpicture}[scale=0.6]
 % define vertices with black circles
 \node[circle,fill,inner sep=1.5pt, minimum size=1mm] (A) at (0:1.4) {};
 \node[circle,fill,inner sep=1.5pt, minimum size=1mm] (B) at (60:1.4) {};
 \node[circle,fill,inner sep=1.5pt, minimum size=1mm] (C) at (120:1.4) {};
 \node[circle,fill,inner sep=1.5pt, minimum size=1mm] (D) at (180:1.4) {};
 \node[circle,fill,inner sep=1.5pt, minimum size=1mm] (E) at (240:1.4) {};
 \node[circle,fill,inner sep=1.5pt, minimum size=1mm] (F) at (300:1.4) {};
 % draw edges
 \draw (F) -- (A) -- (B)--(F);
 \draw (C)--(D)--(E)--(C);
 \draw(A)--(D);
 \draw(B)--(C);
 \draw(F)--(E);
 
 \node (name) at (0,-2.2) {$\overline{C_6}$};
\end{tikzpicture}\hspace{3em}
\begin{tikzpicture}[scale=0.6]
 % define vertices with black circles
 \node[circle,fill,inner sep=1.5pt, minimum size=1mm] (A) at (345:1.4) {};
 \node[circle,fill,inner sep=1.5pt, minimum size=1mm] (A2) at (15:1.4) {};
 \node[circle,fill,inner sep=1.5pt, minimum size=1mm] (B) at (60:1.4) {};
 \node[circle,fill,inner sep=1.5pt, minimum size=1mm] (C) at (120:1.4) {};
 \node[circle,fill,inner sep=1.5pt, minimum size=1mm] (D) at (165:1.4) {};
 \node[circle,fill,inner sep=1.5pt, minimum size=1mm] (D2) at (195:1.4) {};
 \node[circle,fill,inner sep=1.5pt, minimum size=1mm] (E) at (240:1.4) {};
 \node[circle,fill,inner sep=1.5pt, minimum size=1mm] (F) at (300:1.4) {};
 % draw edges
 \draw (F) -- (A) -- (B)--(F);
 \draw(F)--(A2)--(B);
 \draw (C)--(D)--(E)--(C);
 \draw(C)--(D2)--(E);
 \draw(A)--(D);
 \draw(A)--(D2);
 \draw(A2)--(D);
 \draw(A2)--(D2);
 \draw(B)--(C);
 \draw(F)--(E);
 
 \node (name) at (0,-2.2) {$H$};
\end{tikzpicture}
\end{center}
 \caption{The graph $\overline{C_6}$ used in the proof of Proposition~\ref{prop:K3,C5upper} and another tight example $H$.}
 \label{fig:anotherExample}
\end{figure}
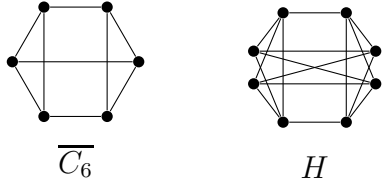

Next, we consider $D$ and $M$. In our analysis, the following result, often referred to as Goodman's Formula~\cite{Goodman59}, will be valuable. For $k\geq1$, let $P_k$ denote the path with $k$ vertices.

\begin{thm}[Goodman's Formula~\cite{Goodman59}]
\label{th:Goodman}
For any graphon $W$, $t(K_3,W)+t(K_3,1-W)$ is equal to
\[t(K_2,W)^3+t(K_2,1-W)^3+\frac{3}{2}\left(t(P_3,W)+t(P_3,1-W)-t(K_2,W)^2-t(K_2,1-W)^2\right).\]
\end{thm}

\begin{prop}
\label{prop:D,Mupper}
$c(D,M)\leq 1/36$.
\end{prop}

\begin{proof}
Define $\alpha=1/36$. Let $K$ be the tensor product of $K_3$ and $K_4$; that is, $K$ has $12$ vertices labelled $u_{i,j}$ for $1\leq i\leq 3$ and $1\leq j\leq 4$ where $u_{i,j}$ is adjacent to $u_{i',j'}$ if and only if $i\neq i'$ and $j\neq j'$; see Figure~\ref{fig:tensor}. We show that $(W_K,W_K)$ is an $\alpha$-certificate for $(D,M)$. 

First, we compute $t(D,K)=t(D,W_K)$. Consider the edge $u_{1,1}u_{2,2}$. If $u_{i,j}$ is a common neighbour of $u_{1,1}$ and $u_{2,2}$, then $i$ must be $3$ and $j$ must be $3$ or $4$. Thus, there are precisely two such common neighbours. Let $x$ and $y$ be the two vertices of degree $3$ in $D$. The number of homomorphisms from $D$ to $K$ such that $f(x)=u_{1,1}$ and $f(y)=u_{2,2}$ is precisely the number of ways to choose an ordered pair of common neighbours of these two vertices with replacement; thus, it is $2^2=4$. Since the automorphism group of $K$ acts transitively on pairs of adjacent vertices, we get that
\[t(D,K) = \frac{\hom(D,K)}{12^4} = \frac{2\cdot |E(K)|\cdot 4}{12^4} = \frac{6\cdot 12\cdot 4}{12^4}=\frac{1}{72}.\]

Now, let us compute $t(M,1-W_K)$. Since $K$ is $6$-regular, we get that 
\[t(K_2,W_K)=t(K_2,1-W_K)=1/2\]
and
\[t(P_3,W_K)=t(P_3,1-W_K)=1/4.\]
Therefore, by Goodman's Formula, $t(K_3,1-W_K) = 1/4 - t(K_3,W_K)$. Using the observations from the previous paragraph, we get that $t(K_3,W_K)=\frac{6\cdot 12\cdot 2}{12^3}=\frac{1}{12}$. Therefore, $t(K_3,1-W_K)=1/6$. Finally, since $K$ is vertex-transitive and $M$ consists of two triangles and a $K_2$ glued together on single vertices, we have
\[t(M,1-W_K) = t(K_3,1-W_K)\cdot t(K_2,1-W_K)\cdot t(K_3,1-W_K)=1/72.\]
Thus, if $W_1=W_2=W_K$, then $(W_1,W_2)$ satisfies all of the conditions of Definition~\ref{defn:certificate} (for any $\lambda\in [0,2]$) and so we are done.
\end{proof}

\begin{figure}[htbp]
\begin{center}
\begin{tikzpicture}[scale=0.6]
 % define vertices with black circles
 \node[circle,fill,inner sep=1.5pt, minimum size=1mm] (A1) at (337.5:2.5) {};
 \node[circle,fill,inner sep=1.5pt, minimum size=1mm] (A2) at (0:2.5) {};
 \node[circle,fill,inner sep=1.5pt, minimum size=1mm] (A3) at (22.5:2.5) {};
 \node[circle,fill,inner sep=1.5pt, minimum size=1mm] (B1) at (67.5:2.5) {};
 \node[circle,fill,inner sep=1.5pt, minimum size=1mm] (B2) at (90:2.5) {};
 \node[circle,fill,inner sep=1.5pt, minimum size=1mm] (B3) at (112.5:2.5) {};
 \node[circle,fill,inner sep=1.5pt, minimum size=1mm] (C1) at (157.5:2.5) {};
 \node[circle,fill,inner sep=1.5pt, minimum size=1mm] (C2) at (180:2.5) {};
 \node[circle,fill,inner sep=1.5pt, minimum size=1mm] (C3) at (202.5:2.5) {};
 \node[circle,fill,inner sep=1.5pt, minimum size=1mm] (D1) at (247.5:2.5) {};
 \node[circle,fill,inner sep=1.5pt, minimum size=1mm] (D2) at (270:2.5) {};
 \node[circle,fill,inner sep=1.5pt, minimum size=1mm] (D3) at (292.5:2.5) {};
 % draw edges
 \draw (A1) -- (B2);
 \draw (A1) -- (B3);
 \draw (A1) -- (C2);
 \draw (A1) -- (C3);
 \draw (A1) -- (D2);
 \draw (A1) -- (D3);
 \draw (A2) -- (B1);
 \draw (A2) -- (B3);
 \draw (A2) -- (C1);
 \draw (A2) -- (C3);
 \draw (A2) -- (D1);
 \draw (A2) -- (D3);
 \draw (A3) -- (B2);
 \draw (A3) -- (B1);
 \draw (A3) -- (C2);
 \draw (A3) -- (C1);
 \draw (A3) -- (D2);
 \draw (A3) -- (D1);
 
 \draw (B1) -- (C2);
 \draw (B1) -- (C3);
 \draw (B1) -- (D2);
 \draw (B1) -- (D3);
 \draw (B2) -- (C1);
 \draw (B2) -- (C3);
 \draw (B2) -- (D1);
 \draw (B2) -- (D3);
 \draw (B3) -- (C2);
 \draw (B3) -- (C1);
 \draw (B3) -- (D2);
 \draw (B3) -- (D1);
 
 \draw (C1) -- (D2);
 \draw (C1) -- (D3);
 \draw (C2) -- (D1);
 \draw (C2) -- (D3);
 \draw (C3) -- (D2);
 \draw (C3) -- (D1);
\end{tikzpicture}
\end{center}
 \caption{The graph $K$ used in the proof of Proposition~\ref{prop:D,Mupper}.}
 \label{fig:tensor}
\end{figure}
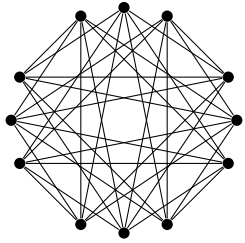

We now explain how we found the constructions in the previous two propositions. In both cases, the first step was to apply the flag algebra method to obtain a lower bound on $c(H_1,H_2)$. An introduction to this approach will be given in the next section but, for now, let us just say that this involved solving, by computer, a certain semidefinite program aimed at maximizing a variable, say $t$, which provides a lower bound on $c(H_1,H_2)$. At first glance, one may wonder about how to choose the right value of $\lambda$ to enter into this optimization problem. However, as it turns out, one can formulate the semidefinite program so that $\lambda$ is one of the variables, and the computer aims to maximize, over all choices of $\lambda$, a variable $t$ satisfying $t\leq c_\lambda(H_1,H_2)$. That is, the output of the optimization problem not only gives us a lower bound on the balanced Ramsey multiplicity constant, but it also points us toward what the SDP solver ``believes'' to be the best possible $\lambda$. 

After obtaining a candidate for $\lambda$ and a lower bound of the form $c_\lambda(H_1,H_2)\geq \alpha$ from flag algebras, where $\alpha$ is the optimal value of $t$, we started to search for an $\alpha$-certificate to get the matching upper bound. For the pair $(K_3,C_5)$, the certificate used in Proposition~\ref{prop:K3,C5upper} is very simple; it is just a pair of graphons corresponding to two graphs on $2$ and $6$ vertices, respectively. The first graphon was easy to guess and the second was found quickly via an exhaustive computer search through all graphs on small numbers of vertices. 

For the pair $(D,M)$, the search for the certificate went somewhat differently. By analyzing the dual of the semidefinite program in the flag algebra calculation, we were able to determine various constraints that a graphon in a $1/36$-certificate is likely to satisfy; e.g., we found values of the homomorphism densities of $K_2$, $K_3$, etc in a theoretical tight example. Moreover, the dual of the SDP suggested that the tight construction should be rather symmetric; e.g. each vertex should have the same degree, be contained in the same number of triangles, etc. By analyzing the dual further and doing a bit of numerological guesswork, we felt that there was a good chance that the optimal certificate involved a graphon corresponding to a graph on 12 vertices. However, all of these observations needed to be taken with a grain of salt due to the fact that an optimal certificate may have involved a pair of graphons $W_1$ and $W_2$ which may have different subgraph densities. 

Given these insights, we ran a heuristic search on the set of $6$-regular graphs on $12$ vertices in an attempt to find such a certificate. The heuristic search started with a random such graph $G$ generated via the configuration model and applied random ``switchings;'' i.e. replacing edges $u_1v_1$ and $u_2v_2$ with $u_1v_2$ and $u_2v_1$ for four vertices $u_1,u_2,v_1,v_2$ such that the edges $u_1v_1$ and $u_2v_2$ are currently present in $G$ and $u_1v_2$ and $u_2v_1$ are not. If the switching decreased the objective function $(5/6)t(D,W_G)+(7/6)t(M,1-W_G)$, then we kept it; if not, then we undid it. The search concluded with the construction in Proposition~\ref{prop:D,Mupper} rather quickly. It seems that the fact that the optimal certificate consists of only a single graphon in this case was very lucky, and it made it much easier to use information from the dual of the SDP to find it. 

\section{Flag Algebra Lower Bounds}
\label{sec:flags}

In this section, we apply the flag algebra method of Razborov~\cite{Razborov07} to complete the proofs of Theorems~\ref{th:C5,B}--\ref{th:D,M}. Let us begin by stating the key definitions and lemmas that we will need. All of the lemmas stated here are very standard; proofs of these statements, or statements very similar to them, can be found in, e.g.,~\cite{Lovasz12,Razborov07}. See~\cite[Section~5]{BehagueMorrisonNoel25} for a gentle introduction to the method which contains proofs of similar statements in a more restricted setting.

\begin{defn}
The \emph{induced homomorphism density} of a graph $J$ in a graphon $W$ is
\[t_{\ind}(J,W):=\int_{[0,1]^{V(J)}}\prod_{uv\in E(J)}W(x_u,x_v)\prod_{uv\in E(\overline{J})}(1-W(x_u,x_v))\prod_{v\in V(J)}dx_v.\]
\end{defn}

\begin{defn}
For a graph $J$ and a graphon $W$, define $d(J,W):=\frac{v(J)!}{\aut(J)}\cdot t_{\ind}(J,W)$, where $\aut(J)$ is the number of automorphisms of $J$.
\end{defn}

\begin{lem}
\label{lem:sumToOne}
For any positive integer $\ell$ and graphon $W$,
\[\sum_{J:v(J)=\ell}d(J,W)=1\]
where the sum is over all graphs $J$ on $\ell$ vertices up to isomorphism. 
\end{lem}

\begin{defn}
Given graphs $H$ and $J$ with $v(H)\leq v(J)$, the \emph{injective homomorphism density} of $H$ in $J$, denoted $t_{\inj}(H,J)$, is the probability that a random injective function from $V(H)$ to $V(J)$ is a homomorphism. 
\end{defn}

\begin{lem}
\label{lem:tind}
For any graph $H$, integer $\ell\geq v(H)$ and graphon $W$, 
\[t(H,W)=\sum_{J: v(J)=\ell}t_{\inj}(H,J)\cdot d(J,W)\]
where the sum is over all graphs $J$ on $\ell$ vertices up to isomorphism. 
\end{lem}

\begin{defn}
Let $F$ be a graph on vertex set $[k]:=\{1,\dots,k\}$ and $0\leq r\leq k$. Given a graphon $W$, define $t_{\ind,r}(F,W):[0,1]^{r}\to [0,1]$ by
\[t_{\ind,r}(F,W)(x_1,\dots,x_r)=\int_{[0,1]^{k-r}}\prod_{ij\in E(F)}W(x_i,x_j)\prod_{ij\in E(\overline{F})}(1-W(x_i,x_j))\prod_{i=r+1}^kdx_i.\]
\end{defn}

\begin{defn}
Given graphs $F_1,\dots,F_t$ on vertex set $[k]$, an integer $0\leq r\leq k$ and a graph $F$ on vertex set $[r]$, we say that $F_1,\dots,F_t$ are \emph{$F$-compatible} if the subgraph of $F_i$ induced by $[r]$ is equal to $F$ for all $1\leq i\leq t$. Less precisely, we say that $F_1,\dots,F_t$ are \emph{$r$-compatible} if they are $F$-compatible for some graph $F$ on vertex set $[r]$.
\end{defn}

\begin{defn}
Let $F$ be a graph on $[r]$ and let $F_1$ and $F_2$ be $F$-compatible graphs on $[k]$. Define $t_{\ind,r}(F_1\cdot F_2,W)(x_{1},\dots,x_r)$ to be equal to
\[\begin{cases}\displaystyle\frac{t_{\ind,r}(F_1,W)(x_{1},\dots,x_r)\cdot t_{\ind,r}(F_2,W)(x_{1},\dots,x_r)}{t_{\ind,r}(F,W)(x_1,\dots,x_r)} &\text{if } t_{\ind,r}(F,W)(x_1,\dots,x_r)\neq 0,\\ \displaystyle0 &\text{otherwise}.\end{cases}\]
\end{defn}

\begin{lem}
\label{lem:integrateOut}
Let $F$ be a graph on $[r]$, let $F_1$ and $F_2$ be $F$-compatible graphs on $[k]$ and let $\ell\geq2k-r$. Then there exist constants $a_r(F_1,F_2;J)$ for each graph $J$ such that $v(J)=\ell$ such that
\[\int_{[0,1]^r}t_{\ind,r}(F_1\cdot F_2,W)(x_1,\dots,x_r)dx_1\cdots dx_r=\sum_{J:v(J)=\ell}a_r(F_1,F_2;J)\cdot d(J,W)\]
for every graphon $W$, where the sum on the right side is over all graphs $J$ on $\ell$ vertices up to isomorphism. 
\end{lem}

Let us illustrate the way in which one can compute the coefficients $a_r(F_1,F_2;J)$ in the previous lemma with an example. Let $F_1=F_2$ be the complete graph on $[3]$ and let $r=1$. Then, after a slight change of variables, we can write $\int_0^1t_{\ind,1}(F_1\cdot F_2,W)(x_1)dx_1$ as
\[\int_{[0,1]^5}\prod_{\substack{i,j\in\{1,2,3\}\\i<j}}W(x_i,x_j)\prod_{\substack{i,j\in\{1,4,5\}\\i<j}}W(x_i,x_j)dx_1\cdots dx_5\]
\[=\int_{[0,1]^5}\prod_{\substack{i,j\in\{1,2,3\}\\i<j}}W(x_i,x_j)\prod_{\substack{i,j\in\{1,4,5\}\\i<j}}W(x_i,x_j)\prod_{\substack{i\in\{2,3\}\\ j\in \{4,5\}}}(W(x_i,x_j)+(1-W(x_i,x_j)))dx_1\cdots dx_5.\]
If we expand the third product inside the integral and collect terms, we get
\begin{equation}\label{eq:tindJ}t_{\ind}(J_1,W)+4t_{\ind}(J_2,W)+2t_{\ind}(J_3,W)+4t_{\ind}(J_4,W)+4t_{\ind}(J_5,W)+t_{\ind}(J_6,W)\end{equation}
where $J_1,\dots,J_6$ are the following graphs:
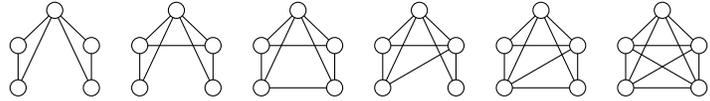
\[\JOne\quad \JTwo \quad \JThree \quad \JFour \quad \JFive \quad \JSix.\]
Then, by definition of $d(J,W)$, \eqref{eq:tindJ} can be rewritten as
\[\frac{8}{120}d(J_1,W) + 4\cdot\frac{2}{120}d(J_2,W)+2\cdot\frac{8}{120}d(J_3,W)+ 4\cdot\frac{4}{120}d(J_4,W)+4\cdot\frac{12}{120}d(J_5,W)+\frac{120}{120}d(J_6,W).\]
Therefore, we have, e.g., $a_1(F_1,F_2;K_5)=a_1(F_1,F_2;J_6)=1$ and $a_1(F_1,F_2;\overline{K_5})=0$. This illustrates the way in which the coefficients in Lemma~\ref{lem:integrateOut} are computed for $\ell=2k-r$. The following lemma allows us to transfer the statement of Lemma~\ref{lem:integrateOut} from $\ell=2k-r$ to larger values of $\ell$. 

\begin{defn}
Given graphs $F$ and $J$ with $v(F)\leq v(J)$, let $d(F,J)$ be the probability that a random subset of $V(J)$ of cardinality $v(F)$ induces a subgraph of $J$ isomorphic to $F$. 
\end{defn}

\begin{lem}
For any graph $F$, integer $\ell\geq v(F)$ and graphon $W$,
\[d(F,W)=\sum_{J: v(J)=\ell}d(F,J)d(J,W)\]
where the sum is over all graphs $J$ on $\ell$ vertices up to isomorphism.
\end{lem}

Recall that a $t\times t$ matrix $A$ is \emph{positive semidefinite (PSD)} if it is real-valued and symmetric and all of its eigenvalues are non-negative. Equivalently, $A$ is PSD if 
\[\sum_{i=1}^t\sum_{j=1}^tA(i,j)x_ix_j\geq0\] 
for any $x_1,\dots,x_t\in\mathbb{R}$, where $A(i,j)$ denotes the entry on the $i$th row and $j$th column of $A$. Using the lemmas presented in this section, together with the second definition of PSD matrices, we derive the following lemma which is key to the main proofs in this section.

\begin{lem}
\label{lem:flag}
Let $H_1$ and $H_2$ be graphs, let $\lambda\in [0,2]$ and let $\ell\geq\max\{v(H_1),v(H_2)\}$. Let $m,r_1,\dots,r_m,k_1,\dots,k_m,t_1,\dots,t_m$ be positive integers and let $F^1,\dots,F^m$ be graphs. For $1\leq q\leq m$, suppose that $2k_q-r_q\leq\ell$ and $V(F^q)=[r_q]$, let $F^q_1,\dots,F^q_{t_q}$ be $F^q$-compatible graphs on vertex set $[k_q]$ and let $A_q$ be a $t_q\times t_q$ PSD matrix. Then $c_\lambda(H_1,H_2)$ is at least
\[\min_{J: v(J)=\ell}\left\{\lambda\cdot t_{\inj}(H_1,J)+(2-\lambda)\cdot t_{\inj}(H_2,\overline{J})-\sum_{q=1}^m\sum_{i=1}^{t_q}\sum_{j=1}^{t_q}A_q(i,j)a_{r_q}(F^q_i, F^q_j;J)\right\}\]
where the minimum is over all graphs $J$ on $\ell$ vertices up to isomorphism.
\end{lem}

\begin{proof}
By Lemma~\ref{lem:altDef}, we can let $W$ be a graphon such that $\lambda\cdot t(H_1,W)+(2-\lambda)\cdot t(H_2,1-W)=c_\lambda(H_1,H_2)$. By Lemma~\ref{lem:tind},
\begin{equation}
\label{eq:tinjJ}
\begin{gathered}
c_\lambda(H_1,H_2)=\lambda\cdot t(H_1,W)+(2-\lambda)\cdot t(H_2,1-W)\\=\sum_{J:v(J)=\ell}\left(\lambda \cdot t_{\inj}(H_1,J)+(2-\lambda)\cdot t_{\inj}(H_2,\overline{J})\right)\cdot d(J,W).
\end{gathered}\end{equation}
Now, by Lemma~\ref{lem:integrateOut} and Fubini's Theorem, we can write
\[\sum_{J:v(J)=\ell}\sum_{q=1}^m\sum_{i=1}^{t_q}\sum_{j=1}^{t_q}A_q(i,j)a_{r_q}(F^q_i, F^q_j;J)d(J,W)\]
\[=\sum_{q=1}^m\int_{[0,1]^{r_q}}\sum_{i=1}^{t_q}\sum_{j=1}^{t_q}A_q(i,j)t_{\ind,r_q}(F^q_i\cdot F^q_j,W)(x_1,\dots,x_{r_q})dx_1\cdots dx_{r_q}.\]
Now, for fixed $1\leq q\leq m$, by definition of $t_{\ind,r_q}(F^q_i\cdot F^q_j,W)(x_1,\dots,x_{r_q})$, for each choice of $(x_1,\dots,x_{r_q})\in [0,1]^{r_q}$, we have that $\sum_{i=1}^{t_q}\sum_{j=1}^{t_q}A_q(i,j)t_{\ind,r_q}(F^q_i\cdot F^q_j,W)(x_1,\dots,x_{r_q})$ is equal to zero if $t_{\ind,r_q}(F^q,W)(x_1,\dots,x_{r_q})=0$ and, otherwise, is equal to
\[\sum_{i=1}^{t_q}\sum_{j=1}^{t_q}\frac{A_q(i,j)t_{\ind,r_q}(F_i^q,W)(x_1,\dots,x_{r_q})\cdot t_{\ind,r_q}(F_j^q,W)(x_1,\dots,x_{r_q})}{t_{\ind,r_q}(F^q,W)(x_1,\dots,x_{r_q})}\]
which is non-negative because $A_q$ is PSD. Thus, 
\[
\sum_{J:v(J)=\ell}\sum_{q=1}^m\sum_{i=1}^{t_q}\sum_{j=1}^{t_q}
A_q(i,j)a_{r_q}(F^q_i, F^q_j;J)d(J,W)\geq0.
\]
Combining this with \eqref{eq:tinjJ}, we see that $c_\lambda(H_1,H_2)$ is at least
\[\sum_{J:v(J)=\ell}\left(\lambda\cdot t_{\inj}(H_1,J)+(2-\lambda)\cdot t_{\inj}(H_2,\overline{J})-\sum_{q=1}^m\sum_{i=1}^{t_q}\sum_{j=1}^{t_q}A_q(i,j)a_{r_q}(F^q_i, F^q_j;J)\right)d(J,W) \]
which, since the quantities $d(J,W)$ are non-negative and sum to one by Lemma~\ref{lem:sumToOne}, implies the conclusion of the lemma. 
\end{proof}

The lower bounds in Theorems~\ref{th:C5,B}--\ref{th:D,M} are all proven using Lemma~\ref{lem:flag}. We present these proofs in order of increasing complexity. The easiest one is Theorem~\ref{th:K3,C5}.

\begin{proof}[Proof of Theorem~\ref{th:K3,C5}]
The upper bound was proven in Proposition~\ref{prop:K3,C5upper}. We apply Lemma~\ref{lem:flag} to show that $c_{10/17}(K_3,C_5)\geq 3/34$. In this application of Lemma~\ref{lem:flag}, we set $\ell=5$, $m=r_1=1$, $k_1=3$ and $t_1=6$. Let $F^1$ be the unique $1$-vertex graph and let $F^1_1,\dots,F^1_6$ be as follows:
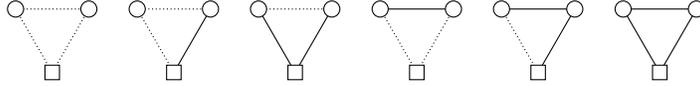
\[\ZeroBlue \quad \OneBlue \quad \TwoBlue \quad \ZeroRed \quad \OneRed \quad \TwoRed\]
In each of the above diagrams, the square vertex is labelled $1$ and the round vertices are labelled $2$ and $3$ from left to right. Define
\[A_1:=\frac{15}{34}\begin{bmatrix}
 3 & 3 & 0 & -2 & -3 & -1\\
 3 & 8 & 5 & -6 & -8 & -2\\
 0 & 5 & 5 & -4 & -5 & -1\\
-2 & -6 & -4 & 5 & 6 & 1\\
-3 & -8 & -5 & 6 & 8 & 2\\
-1 & -2 & -1 & 1 & 2 & 1
\end{bmatrix}.\]
The matrix $A_1$ is PSD; the non-zero eigenvalues of $(34/15)\cdot A_1$ are approximately 25.36, 3.76 and 0.88. There are 34 graphs of order 5, up to isomorphism. For every such graph $J$, it holds that
\begin{equation}(10/17)\cdot t_{\inj}(K_3,J) + (24/17)\cdot t_{\inj}(C_5,\overline{J})-\sum_{i=1}^{6}\sum_{j=1}^{6}A_1(i,j)a_1(F^1_{i},F^1_{j}; J)\geq 3/34.\end{equation}
The calculations needed to verify this for all graphs $J$ on five vertices are tedious; the full list has been included in an appendix within an ancillary file of the arXiv version of the paper: \url{https://arxiv.org/src/2306.17388/anc/offDiagonalAppendices.pdf}. We include two sample calculations to illustrate how they work in principle. 

First, consider $J=K_5$. In this case, $a_1(F^1_{i},F^1_{j}; K_5)$ is equal to $0$ unless $i=j=6$, in which case it is equal to $1$. Also, clearly, $t_{\inj}(K_3,K_5)=1$ and $t_{\inj}(C_5,\overline{K_5})=0$. Therefore, 
\[(10/17)\cdot t_{\inj}(K_3,K_5) + (24/17)\cdot t_{\inj}(C_5,\overline{K_5})-\sum_{i=1}^{6}\sum_{j=1}^{6}A_1(i,j)a_1(F^1_{i},F^1_{j}; K_5)\]
\[= (10/17) - 1\cdot A_1(6,6) = 5/34.\]
For a slightly more involved example, let $J=K_{1,4}$. Then one can check that 
\[a_1(F^1_{3},F^1_{3}; K_{1,4})=1/5,\]
\[a_1(F^1_{1},F^1_{5}; K_{1,4})=a_1(F^1_{5},F^1_{1}; K_{1,4})=1/5\]
and $a_1(F^1_{i},F^1_{j}; K_{1,4})=0$ for all other $i$ and $j$. Also, $t_{\inj}(K_3,K_{1,4})=t_{\inj}(C_5,\overline{K_{1,4}}) = 0$. Therefore, 
\[(10/17)\cdot t_{\inj}(K_3,K_{1,4}) + (24/17)\cdot t_{\inj}(C_5,\overline{K_{1,4}})-\sum_{i=1}^{6}\sum_{j=1}^{6}A_1(i,j)\cdot a_1(F^1_{i},F^1_{j}; K_{1,4})\]
\[= -(1/5)\cdot A_1(3,3)-(1/5)\cdot A_1(1,5)-(1/5)\cdot A_1(5,1) = 3/34.\]
\end{proof}

Next, we show that $c_1(C_5,B)\geq1/16$ which, by Theorem~\ref{th:commonPair} and the fact that $e(C_5)=e(B)$, is equivalent to the statement that $(C_5,B)$ is $(1/2,1/2)$-common. 

\begin{proof}[Proof of Theorem~\ref{th:C5,B}]
The upper bound follows from Proposition~\ref{prop:C5,Bupper}. To prove the lower bound, we apply Lemma~\ref{lem:flag} with $\ell=5$, $m=4$, $r_1=r_2=r_3=r_4=3$ and $k_1=k_2=k_3=k_4=4$ and $t_1=t_2=t_3=t_4=8$. We define graphs $F^q_{i}$ for $1\leq q\leq 4$ and $1\leq i\leq 8$ as in the pictures below, where the eight graphs corresponding to $q=1$ are listed first, followed by the eight graphs for $q=2$, and so on. In these depictions, the square vertices are labelled $1,2,3$ from left to right and the round vertex is labelled $4$. 
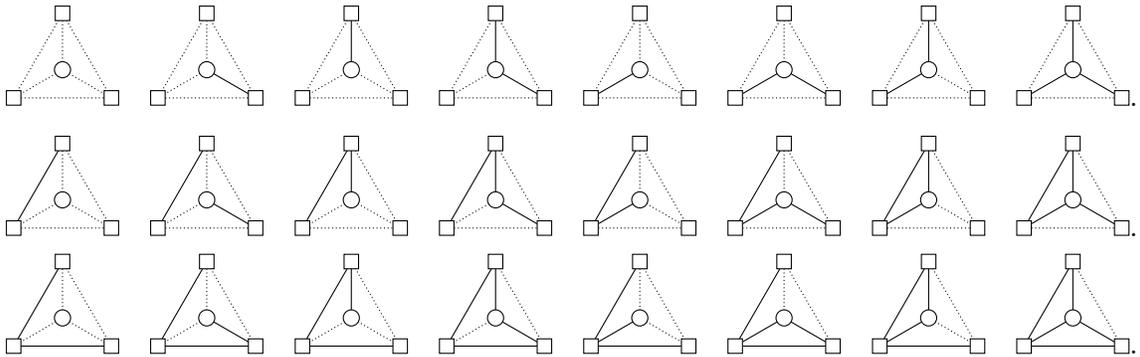
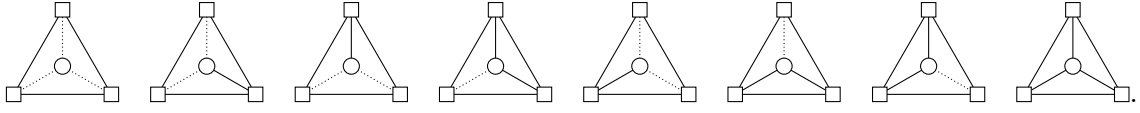
\[\ZeroZero \quad \ZeroThree \quad \ZeroTwo \quad \ZeroSix \quad \ZeroOne \quad \ZeroFive \quad \ZeroFour \quad \ZeroSeven.\]
\[\OneZeroZero \quad \OneZeroThree \quad \OneZeroTwo \quad \OneZeroSix \quad \OneZeroOne \quad \OneZeroFive \quad \OneZeroFour \quad \OneZeroSeven.\]
\[\TwoZeroZero \quad \TwoZeroThree \quad \TwoZeroTwo \quad \TwoZeroSix \quad \TwoZeroOne \quad \TwoZeroFive \quad \TwoZeroFour \quad \TwoZeroSeven.\]
\[\ThreeZeroZero \quad \ThreeZeroThree \quad \ThreeZeroTwo \quad \ThreeZeroSix \quad \ThreeZeroOne \quad \ThreeZeroFive \quad \ThreeZeroFour \quad \ThreeZeroSeven.\]
Consider the following four $8\times 8$ matrices:
\[A_1:=\frac{1}{192}\begin{bmatrix}
180 &60 &60 &-60 &60 &-60 &-60 &-180\\
60 &64 &-2 &-5 &-2 &-5 &-50 &-60\\
60 &-2 &64 &-5 &-2 &-50 &-5 &-60\\
-60 &-5 &-5 &100 &-50 &-20 &-20 &60\\
60 &-2 &-2 &-50 &64 &-5 &-5 &-60\\
-60 &-5 &-50 &-20 &-5 &100 &-20 &60\\
-60 &-50 &-5 &-20 &-5 &-20 &100 &60\\
-180 &-60 &-60 &60 &-60 &60 &60 &180 \end{bmatrix}\]
\[A_2:=\frac{1}{192}\begin{bmatrix}
300 &150 &75 &-75 &75 &-75 &-150 &-300\\
150 &192 &-21 &0 &-21 &0 &-150 &-150\\
75 &-21 &96 &-15 &0 &-60 &0 &-75\\
-75 &0 &-15 &120 &-60 &30 &-75 &75\\
75 &-21 &0 &-60 &96 &-15 &0 &-75\\
-75 &0 &-60 &30 &-15 &120 &-75 &75\\
-150 &-150 &0 &-75 &0 &-75 &300 &150\\
-300 &-150 &-75 &75 &-75 &75 &150 &300 \end{bmatrix}\]
\[A_3:=\frac{1}{192}\begin{bmatrix}
300 &30 &30 &-240 &240 &-30 &-30 &-300\\
30 &66 &-6 &30 &-30 &-30 &-30 &-30\\
30 &-6 &66 &30 &-30 &-30 &-30 &-30\\
-240 &30 &30 &300 &-300 &-30 &-30 &240\\
240 &-30 &-30 &-300 &300 &30 &30 &-240\\
-30 &-30 &-30 &-30 &30 &90 &-30 &30\\
-30 &-30 &-30 &-30 &30 &-30 &90 &30\\
-300 &-30 &-30 &240 &-240 &30 &30 &300
\end{bmatrix}\]
\[A_4:=\frac{1}{192}\begin{bmatrix}
180 &60 &60 &-60 &60 &-60 &-60 &-180\\
60 &60 &0 &0 &0 &0 &-60 &-60\\
60 &0 &60 &0 &0 &-60 &0 &-60\\
-60 &0 &0 &60 &-60 &0 &0 &60\\
60 &0 &0 &-60 &60 &0 &0 &-60\\
-60 &0 &-60 &0 &0 &60 &0 &60\\
-60 &-60 &0 &0 &0 &0 &60 &60\\
-180 &-60 &-60 &60 &-60 &60 &60 &180
\end{bmatrix}\]
The matrices $A_1,A_2,A_3$ and $A_4$ are PSD. For each graph $J$ of order five, we have
\begin{equation}t_{\inj}(C_5,J) + t_{\inj}(B,\overline{J})-\sum_{q=1}^4\sum_{i=1}^{8}\sum_{j=1}^{8}A_q(i,j)a_{r_q}(F^q_i,F^q_j;J)=1/16.\end{equation}
The full list of calculations required to verify this for all graphs $J$ on five vertices has been included in an ancillary file with the arXiv preprint of the paper: \url{https://arxiv.org/src/2306.17388/anc/offDiagonalAppendices.pdf}. 
\end{proof}

Finally, we prove Theorem~\ref{th:D,M}. This proof follows the same principles of the other two in this section, but the details are much more complicated.

\begin{proof}[Proof of Theorem~\ref{th:D,M}]
The upper bound was proven in Proposition~\ref{prop:D,Mupper}. We prove the lower bound $c_{5/6}(D,M)\geq 1/36$ by applying Lemma~\ref{lem:flag} with $\ell=6$, $m=5$, $r_1=3$, $r_2=2$, $r_3=r_4=r_5=4$, $k_1=k_2=4$, $k_3=k_4=k_5=5$, $t_1=8$, $t_2=20$ and $t_3=t_4=t_5=16$. Let $F^1_1,\dots,F^1_8$ be as in the proof of Theorem~\ref{th:C5,B}. Define $F^2_1,\dots,F^2_{20}$ to be the following graphs where the square vertices are labelled $1$ and $2$ from left to right and the round vertices are labelled $3$ and $4$ from left to right:
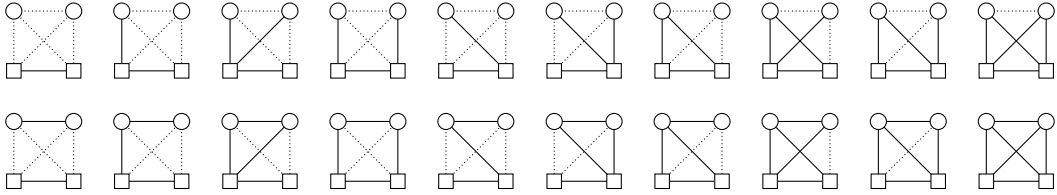
\[\CTwoOne\quad\CTwoTwo\quad\CTwoThree\quad\CTwoFour\quad\CTwoFive\quad\CTwoSix\quad\CTwoSeven\quad\CTwoEight\quad\CTwoNine\quad\CTwoTen\]
\[\CTwoEleven\quad\CTwoTwelve\quad\CTwoThirteen\quad\CTwoFourteen\quad\CTwoFifteen\quad\CTwoSixteen\quad\CTwoSeventeen\quad\CTwoEighteen\quad\CTwoNineteen\quad\CTwoTwenty\]
Next, we define $F^q_i$ for $3\leq q\leq 5$ and $1\leq i\leq 16$ as follows, where the graphs $F^3_1,\dots,F^3_{16}$ are listed first, followed by $F^4_1,\dots,F^4_{16}$ and then $F^5_1,\dots,F^5_{16}$. Square vertices are labelled $1,2,3$ and $4$ from left to right and round vertices are labelled $5$.
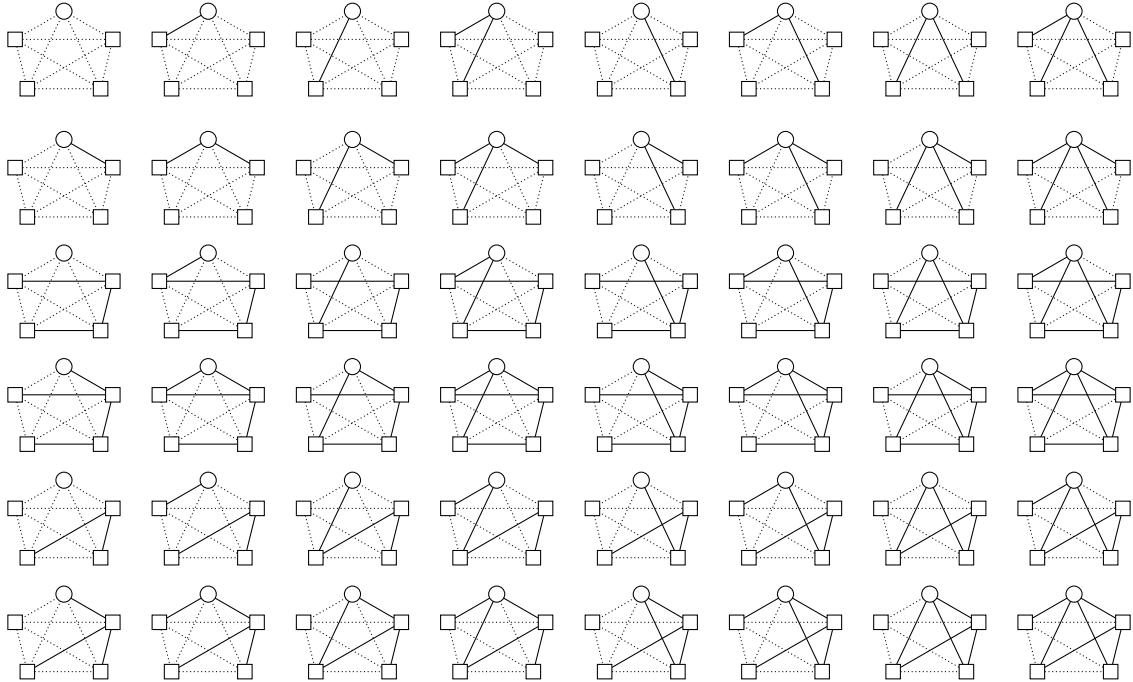
\[\HOneOne\quad\HOneTwo\quad\HOneThree\quad\HOneFour\quad\HOneFive\quad\HOneSix\quad\HOneSeven\quad\HOneEight\]
\[\HOneNine\quad\HOneTen\quad\HOneEleven\quad\HOneTwelve\quad\HOneThirteen\quad\HOneFourteen\quad\HOneFifteen\quad\HOneSixteen\]
\[\LOneOne\quad\LOneTwo\quad\LOneThree\quad\LOneFour\quad\LOneFive\quad\LOneSix\quad\LOneSeven\quad\LOneEight\]
\[\LOneNine\quad\LOneTen\quad\LOneEleven\quad\LOneTwelve\quad\LOneThirteen\quad\LOneFourteen\quad\LOneFifteen\quad\LOneSixteen\]
\[\MOneOne\quad\MOneTwo\quad\MOneThree\quad\MOneFour\quad\MOneFive\quad\MOneSix\quad\MOneSeven\quad\MOneEight\]
\[\MOneNine\quad\MOneTen\quad\MOneEleven\quad\MOneTwelve\quad\MOneThirteen\quad\MOneFourteen\quad\MOneFifteen\quad\MOneSixteen\]
Let $A_1,\dots,A_5$ be $\frac{1}{619173642240}$ times the following five matrices, respectively. 
\[\scalemath{0.67}{\begin{bmatrix}
287949980052 & 95985403560 & 95982308280 & -95982268212 & 95982268212 & -95982308280 & -95985403560 & -287949980052 \\
95985403560 & 41365839780 & 27309810060 & -27309753720 & 27309753720 & -27309810060 & -41365839780 & -95985403560 \\
95982308280 & 27309810060 & 41363696376 & -27308801844 & 27308801844 & -41363696376 & -27309810060 & -95982308280 \\
-95982268212 & -27309753720 & -27308801844 & 41363712648 & -41363712648 & 27308801844 & 27309753720 & 95982268212 \\
95982268212 & 27309753720 & 27308801844 & -41363712648 & 41363712648 & -27308801844 & -27309753720 & -95982268212 \\
-95982308280 & -27309810060 & -41363696376 & 27308801844 & -27308801844 & 41363696376 & 27309810060 & 95982308280 \\
-95985403560 & -41365839780 & -27309810060 & 27309753720 & -27309753720 & 27309810060 & 41365839780 & 95985403560 \\
-287949980052 & -95985403560 & -95982308280 & 95982268212 & -95982268212 & 95982308280 & 95985403560 & 287949980052
\end{bmatrix}}\]

\[\scalemath{0.255}{\begin{bmatrix}
1131265129284 & 184507193658 & 99130531245 & -642119961108 & 184457538846 & 99241198095 & -877057679628 & 229364308632 & 229367515446 & 138201906510 & 405034352178 & -280447484154 & 9410203173 & -119657863368 & -280587239814 & 9476709726 & -978369617385 & -85355623230 & -85312177500 & -569848526100 \\
184507193658 & 517178073672 & -38045911104 & -193027147728 & -187544056548 & 172246863402 & -147784455936 & -46005730122 & -124561640748 & -49217217192 & 130760723514 & 468520023192 & -535011098574 & 1282469256 & -52069329372 & 286461450888 & -122425371774 & -544223295768 & -9032118246 & -217953986292 \\
99130531245 & -38045911104 & 509547150576 & -8012447856 & 172207608768 & -304245560808 & -77918769864 & 106745178789 & 40554565704 & 33461424432 & 90874755126 & -449583420696 & 26340942330 & -96780707496 & 478465306482 & -255987934338 & -48391702560 & -141396072336 & -341050715676 & -152078958144 \\
-642119961108 & -193027147728 & -8012447856 & 730911902940 & -193031372484 & -8033329440 & 504494186436 & 161668494864 & 161657949060 & 168683550408 & -238284890208 & -178981768908 & -81153983268 & -48618896976 & -178937409420 & -81131719932 & 485571419928 & 16929608148 & 16954460316 & 347291255268 \\
184457538846 & -187544056548 & 172207608768 & -193031372484 & 517163384592 & -38054242404 & -147746201148 & -124612095600 & -46061765070 & -49271049468 & 130733496720 & -52027081596 & 286469790696 & 1312357032 & 468542054904 & -534965701500 & -122374809936 & -8992185840 & -544149821184 & -217912121172 \\
99241198095 & 172246863402 & -304245560808 & -8033329440 & -38054242404 & 509647225212 & -78004099872 & 40640002863 & 106818602472 & 33495756876 & 90928468728 & 478441313688 & -256066489947 & -96848811972 & -449689633956 & 26363039976 & -48474761247 & -341164199484 & -141422678190 & -152159019696 \\
-877057679628 & -147784455936 & -77918769864 & 504494186436 & -147746201148 & -78004099872 & 681868832868 & -172456693596 & -172459140876 & -104851822896 & -316509850836 & 206090698680 & -3881460060 & 91218294720 & 206198410104 & -3932796564 & 757507278456 & 75013487640 & 74979914436 & 446376891816 \\
229364308632 & -46005730122 & 106745178789 & 161668494864 & -124612095600 & 40640002863 & -172456693596 & 514298685204 & 345303775008 & 409120996824 & 84685817844 & -561075155586 & -244292787216 & -169972253604 & -230540400180 & 106131742668 & -320874582276 & -253705886640 & 82979850276 & -113882137380 \\
229367515446 & -124561640748 & 40554565704 & 161657949060 & -46061765070 & 106818602472 & -172459140876 & 345303775008 & 514280398968 & 409106932308 & 84682175112 & -230460082992 & 106054979292 & -169967178144 & -561159130698 & -244191159432 & -320874268968 & 82932372936 & -253613845206 & -113874893496 \\
138201906510 & -49217217192 & 33461424432 & 168683550408 & -49271049468 & 33495756876 & -104851822896 & 409120996824 & 409106932308 & 694356497316 & 91261680372 & -209490897780 & -171865442628 & -57965871240 & -209527242804 & -171778573692 & -346400403876 & -319480211178 & -319412774952 & -144951014160 \\
405034352178 & 130760723514 & 90874755126 & -238284890208 & 130733496720 & 90928468728 & -316509850836 & 84685817844 & 84682175112 & 91261680372 & 205100989704 & 17728507584 & -105524997264 & -64510476246 & 17673401664 & -105476885208 & -345771885420 & -259708898700 & -259670841012 & -315521227908 \\
-280447484154 & 468520023192 & -449583420696 & -178981768908 & -52027081596 & 478441313688 & 206090698680 & -561075155586 & -230460082992 & -209490897780 & 17728507584 & 1768902943788 & -136835949348 & 154997093160 & -422079156648 & -205803441792 & 308711686428 & -213123466872 & -607375126188 & -88637001516 \\
9410203173 & -535011098574 & 26340942330 & -81153983268 & 286469790696 & -256066489947 & -3881460060 & -244292787216 & 106054979292 & -171865442628 & -105524997264 & -136835949348 & 1064739385404 & 54667527588 & -205754248548 & -515477856312 & 19967171976 & 1120038735348 & -191114839908 & 198617708124 \\
-119657863368 & 1282469256 & -96780707496 & -48618896976 & 1312357032 & -96848811972 & 91218294720 & -169972253604 & -169967178144 & -57965871240 & -64510476246 & 154997093160 & 54667527588 & 148548463164 & 155059867512 & 54609269148 & 82709822196 & 98417695680 & 98369130672 & 97306676568 \\
-280587239814 & -52069329372 & 478465306482 & -178937409420 & 468542054904 & -449689633956 & 206198410104 & -230540400180 & -561159130698 & -209527242804 & 17673401664 & -422079156648 & -205754248548 & 155059867512 & 1769070442068 & -136867009608 & 308821998060 & -607310478144 & -213109469910 & -88559088264 \\
9476709726 & 286461450888 & -255987934338 & -81131719932 & -534965701500 & 26363039976 & -3932796564 & 106131742668 & -244191159432 & -171778573692 & -105476885208 & -205803441792 & -515477856312 & 54609269148 & -136867009608 & 1064615955012 & 19900919484 & -191183574888 & 1119859135020 & 198540260028 \\
-978369617385 & -122425371774 & -48391702560 & 485571419928 & -122374809936 & -48474761247 & 757507278456 & -320874582276 & -320874268968 & -346400403876 & -345771885420 & 308711686428 & 19967171976 & 82709822196 & 308821998060 & 19900919484 & 1003963314744 & 53447774904 & 53404693524 & 484204319316 \\
-85355623230 & -544223295768 & -141396072336 & 16929608148 & -8992185840 & -341164199484 & 75013487640 & -253705886640 & 82932372936 & -319480211178 & -259708898700 & -213123466872 & 1120038735348 & 98417695680 & -607310478144 & -191183574888 & 53447774904 & 1657504562004 & 444949397640 & 473910953040 \\
-85312177500 & -9032118246 & -341050715676 & 16954460316 & -544149821184 & -141422678190 & 74979914436 & 82979850276 & -253613845206 & -319412774952 & -259670841012 & -607375126188 & -191114839908 & 98369130672 & -213109469910 & 1119859135020 & 53404693524 & 444949397640 & 1657285425900 & 473848176420 \\
-569848526100 & -217953986292 & -152078958144 & 347291255268 & -217912121172 & -152159019696 & 446376891816 & -113882137380 & -113874893496 & -144951014160 & -315521227908 & -88637001516 & 198617708124 & 97306676568 & -88559088264 & 198540260028 & 484204319316 & 473910953040 & 473848176420 & 497046744252
\end{bmatrix}}\]

\[\scalemath{0.345}{\begin{bmatrix}
417220001388 & 208609957332 & 208609908264 & -135792 & 208610379024 & 334968 & 285900 & -208609758156 & 208609758156 & -285900 & -334968 & -208610379024 & 135792 & -208609908264 & -208609957332 & -417220001388 \\
208609957332 & 105053290920 & 104055488244 & 498821832 & 104055723540 & 499057128 & -498745548 & -104055411960 & 104055411960 & 498745548 & -499057128 & -104055723540 & -498821832 & -104055488244 & -105053290920 & -208609957332 \\
208609908264 & 104055488244 & 105053243040 & 498823020 & 104055697368 & -498722652 & 499032144 & -104055387876 & 104055387876 & -499032144 & 498722652 & -104055697368 & -498823020 & -105053243040 & -104055488244 & -208609908264 \\
-135792 & 498821832 & 498823020 & 997780644 & -498958116 & -492 & 696 & 498958320 & -498958320 & -696 & 492 & 498958116 & -997780644 & -498823020 & -498821832 & 135792 \\
208610379024 & 104055723540 & 104055697368 & -498958116 & 105053714172 & 499058688 & 499032516 & -104055622968 & 104055622968 & -499032516 & -499058688 & -105053714172 & 498958116 & -104055697368 & -104055723540 & -208610379024 \\
334968 & 499057128 & -498722652 & -492 & 499058688 & 997780848 & 1068 & 498723228 & -498723228 & -1068 & -997780848 & -499058688 & 492 & 498722652 & -499057128 & -334968 \\
285900 & -498745548 & 499032144 & 696 & 499032516 & 1068 & 997778760 & 498747312 & -498747312 & -997778760 & -1068 & -499032516 & -696 & -499032144 & 498745548 & -285900 \\
-208609758156 & -104055411960 & -104055387876 & 498958320 & -104055622968 & 498723228 & 498747312 & 105053093508 & -105053093508 & -498747312 & -498723228 & 104055622968 & -498958320 & 104055387876 & 104055411960 & 208609758156 \\
208609758156 & 104055411960 & 104055387876 & -498958320 & 104055622968 & -498723228 & -498747312 & -105053093508 & 105053093508 & 498747312 & 498723228 & -104055622968 & 498958320 & -104055387876 & -104055411960 & -208609758156 \\
-285900 & 498745548 & -499032144 & -696 & -499032516 & -1068 & -997778760 & -498747312 & 498747312 & 997778760 & 1068 & 499032516 & 696 & 499032144 & -498745548 & 285900 \\
-334968 & -499057128 & 498722652 & 492 & -499058688 & -997780848 & -1068 & -498723228 & 498723228 & 1068 & 997780848 & 499058688 & -492 & -498722652 & 499057128 & 334968 \\
-208610379024 & -104055723540 & -104055697368 & 498958116 & -105053714172 & -499058688 & -499032516 & 104055622968 & -104055622968 & 499032516 & 499058688 & 105053714172 & -498958116 & 104055697368 & 104055723540 & 208610379024 \\
135792 & -498821832 & -498823020 & -997780644 & 498958116 & 492 & -696 & -498958320 & 498958320 & 696 & -492 & -498958116 & 997780644 & 498823020 & 498821832 & -135792 \\
-208609908264 & -104055488244 & -105053243040 & -498823020 & -104055697368 & 498722652 & -499032144 & 104055387876 & -104055387876 & 499032144 & -498722652 & 104055697368 & 498823020 & 105053243040 & 104055488244 & 208609908264 \\
-208609957332 & -105053290920 & -104055488244 & -498821832 & -104055723540 & -499057128 & 498745548 & 104055411960 & -104055411960 & -498745548 & 499057128 & 104055723540 & 498821832 & 104055488244 & 105053290920 & 208609957332 \\
-417220001388 & -208609957332 & -208609908264 & 135792 & -208610379024 & -334968 & -285900 & 208609758156 & -208609758156 & 285900 & 334968 & 208610379024 & -135792 & 208609908264 & 208609957332 & 417220001388
\end{bmatrix}}\]

\[\scalemath{0.305}{\begin{bmatrix}
2774557681848 & 1350104056350 & 1347143775276 & -77309850222 & 1455606240186 & 31152614688 & -539210101032 & -1512789674616 & 1453634987136 & -540165862188 & 26221080564 & -1514819630808 & 195677928324 & -1263927084048 & -1266023285172 & -2889305072352 \\
1350104056350 & 3041948223276 & 200280072666 & 1892124239592 & 109878657486 & 1801722824412 & -751193597052 & -158738863536 & -518023820004 & 226107539712 & -1667847803688 & -402458834628 & -922781378910 & -220522167720 & -2058253894728 & -1617854097798 \\
1347143775276 & 200280072666 & 3038362906344 & 1891499203734 & -517068352590 & -1663932055200 & 229477736376 & -397645752078 & 106116469188 & -752406189480 & 1797335600256 & -159579697536 & -923321870718 & -2055325892940 & -220435480152 & -1613950865388 \\
-77309850222 & 1892124239592 & 1891499203734 & 3860933293548 & -1862795935290 & 106638154524 & 17494240356 & 956405059002 & -1865542337952 & 13867212420 & 103266716004 & 952781098644 & -2041781177952 & -1011920976612 & -1012666089708 & -342499890834 \\
1455606240186 & 109878657486 & -517068352590 & -1862795935290 & 2158051475952 & 812323893252 & 144620362548 & -722650875984 & 1170426695436 & -884554498872 & -802247897340 & -1909964208996 & 1166932848882 & 94226760000 & -450470714544 & -1475758232010 \\
31152614688 & 1801722824412 & -1663932055200 & 106638154524 & 812323893252 & 2582894102976 & -67363133472 & 631399935096 & -801232111704 & -118281096972 & -2496316781592 & -797603412816 & 48473541648 & 1137631676328 & -1242701324100 & -204307257456 \\
-539210101032 & -751193597052 & 229477736376 & 17494240356 & 144620362548 & -67363133472 & 1481961192156 & 1340469291600 & -884718676980 & -753146813628 & -116030839572 & -264668752908 & -353196464616 & -12636718560 & 417610645776 & 905998047120 \\
-1512789674616 & -158738863536 & -397645752078 & 956405059002 & -722650875984 & 631399935096 & 1340469291600 & 2138829558516 & -1909443198330 & -265547604996 & -794299275792 & 453538155348 & -916678698552 & 540037841760 & 301872054960 & 1838588356008 \\
1453634987136 & -518023820004 & 106116469188 & -1865542337952 & 1170426695436 & -801232111704 & -884718676980 & -1909443198330 & 2158253725932 & 146593024668 & 810735207984 & -721224577764 & 1167529692678 & -448001784000 & 94984738452 & -1472641179696 \\
-540165862188 & 226107539712 & -752406189480 & 13867212420 & -884554498872 & -118281096972 & -753146813628 & -265547604996 & 146593024668 & 1483209554436 & -65647302624 & 1341610437060 & -351968570280 & 418216318956 & -10477990008 & 908022661380 \\
26221080564 & -1667847803688 & 1797335600256 & 103266716004 & -802247897340 & -2496316781592 & -116030839572 & -794299275792 & 810735207984 & -65647302624 & 2581849727676 & 634015355508 & 48529893636 & -1239400592892 & 1140572543472 & -197286972732 \\
-1514819630808 & -402458834628 & -159579697536 & 952781098644 & -1909964208996 & -797603412816 & -264668752908 & 453538155348 & -721224577764 & 1341610437060 & 634015355508 & 2140629493680 & -915759960912 & 302390549568 & 543254001516 & 1841995745424 \\
195677928324 & -922781378910 & -923321870718 & -2041781177952 & 1166932848882 & 48473541648 & -353196464616 & -916678698552 & 1167529692678 & -351968570280 & 48529893636 & -915759960912 & 1539757295856 & 523267609284 & 523026009324 & -389684798892 \\
-1263927084048 & -220522167720 & -2055325892940 & -1011920976612 & 94226760000 & 1137631676328 & -12636718560 & 540037841760 & -448001784000 & 418216318956 & -1239400592892 & 302390549568 & 523267609284 & 1680142768740 & 227521974960 & 1467464917176 \\
-1266023285172 & -2058253894728 & -220435480152 & -1012666089708 & -450470714544 & -1242701324100 & 417610645776 & 301872054960 & 94984738452 & -10477990008 & 1140572543472 & 543254001516 & 523026009324 & 227521974960 & 1682439809724 & 1470280846428 \\
-2889305072352 & -1617854097798 & -1613950865388 & -342499890834 & -1475758232010 & -204307257456 & 905998047120 & 1838588356008 & -1472641179696 & 908022661380 & -197286972732 & 1841995745424 & -389684798892 & 1467464917176 & 1470280846428 & 3692493532728
\end{bmatrix}}\]

\[\scalemath{0.305}{\begin{bmatrix}
1522897677684 & -140544997830 & 1562240672190 & -101202003324 & 1559711933154 & -103730742360 & 1599054927660 & -64387747854 & 121761443106 & -1541681232408 & -548138809584 & -1589386013604 & -549137385720 & -1589940286992 & -994142382192 & -1637645068188 \\
-140544997830 & 939362217432 & 229502997432 & 1309410212694 & 231074579748 & 1310981795010 & 601122575010 & 1681029790272 & -1336739112072 & -256831896810 & -1131160690044 & -402509518680 & -1130749230816 & -402358736700 & -1092411838548 & -548036358570 \\
1562240672190 & 229502997432 & 2886440959368 & 1553703284610 & 928400906790 & -404336767968 & 2252601193968 & 919863519210 & -685590101262 & -2018327776020 & -834096075012 & -1894519222884 & -1433289161268 & -2154596061222 & -1308014133264 & -2030787508086 \\
-101202003324 & 1309410212694 & 1553703284610 & 2964315500628 & -400236446616 & 1010375769402 & 1254668841318 & 2665281057336 & -2144090656440 & -733478440422 & -1417117955472 & -707642727960 & -2014901006364 & -967014510930 & -1406283589620 & -941178798468 \\
1559711933154 & 231074579748 & 928400906790 & -400236446616 & 2882680869588 & 1554043516182 & 2251369843224 & 922732489818 & -688039544322 & -2016676897728 & -1433229945444 & -2152331014596 & -835916015124 & -1893443707278 & -1307861888400 & -2029097824146 \\
-103730742360 & 1310981795010 & -404336767968 & 1010375769402 & 1554043516182 & 2968756053552 & 1253437490574 & 2668150027944 & -2146540099500 & -731827562130 & -2016251825904 & -965454519672 & -1417527860220 & -705862156986 & -1406131344756 & -939489114528 \\
1599054927660 & 601122575010 & 2252601193968 & 1254668841318 & 2251369843224 & 1253437490574 & 2904916109532 & 1906983756882 & -1495391088690 & -2493323441340 & -1719187210872 & -2457464223876 & -1720067790672 & -2458099481508 & -1621733639472 & -2422240264044 \\
-64387747854 & 1681029790272 & 919863519210 & 2665281057336 & 922732489818 & 2668150027944 & 1906983756882 & 3652401295008 & -2953891643868 & -1208474105742 & -2302209091332 & -1270587728952 & -2301679635768 & -1270517931216 & -1720003095828 & -1332631554426 \\
121761443106 & -1336739112072 & -685590101262 & -2144090656440 & -688039544322 & -2146540099500 & -1495391088690 & -2953891643868 & 2427492260196 & 968991705018 & 1800875286828 & 950038935624 & 1800434816964 & 949990093632 & 1271589408180 & 931037324238 \\
-1541681232408 & -256831896810 & -2018327776020 & -733478440422 & -2016676897728 & -731827562130 & -2493323441340 & -1208474105742 & 968991705018 & 2253841040616 & 1217853406368 & 2136915430548 & 1218822971868 & 2137571643924 & 1173319951824 & 2020646033856 \\
-548138809584 & -1131160690044 & -834096075012 & -1417117955472 & -1433229945444 & -2016251825904 & -1719187210872 & -2302209091332 & 1800875286828 & 1217853406368 & 2189701794024 & 1669799126304 & 1402395946596 & 1098860698656 & 1751181244200 & 1550806418592 \\
-1589386013604 & -402509518680 & -1894519222884 & -707642727960 & -2152331014596 & -965454519672 & -2457464223876 & -1270587728952 & 950038935624 & 2136915430548 & 1669799126304 & 2432050701252 & 1099649380716 & 1935348329556 & 1532069128656 & 2230483600260 \\
-549137385720 & -1130749230816 & -1433289161268 & -2014901006364 & -835916015124 & -1417527860220 & -1720067790672 & -2301679635768 & 1800434816964 & 1218822971868 & 1402395946596 & 1099649380716 & 2189652959520 & 1670800614480 & 1751368886820 & 1551627023328 \\
-1589940286992 & -402358736700 & -2154596061222 & -967014510930 & -1893443707278 & -705862156986 & -2458099481508 & -1270517931216 & 949990093632 & 2137571643924 & 1098860698656 & 1935348329556 & 1670800614480 & 2433219276528 & 1532220580656 & 2230995962160 \\
-994142382192 & -1092411838548 & -1308014133264 & -1406283589620 & -1307861888400 & -1406131344756 & -1621733639472 & -1720003095828 & 1271589408180 & 1173319951824 & 1751181244200 & 1532069128656 & 1751368886820 & 1532220580656 & 2647565911944 & 1890969757488 \\
-1637645068188 & -548036358570 & -2030787508086 & -941178798468 & -2029097824146 & -939489114528 & -2422240264044 & -1332631554426 & 931037324238 & 2020646033856 & 1550806418592 & 2230483600260 & 1551627023328 & 2230995962160 & 1890969757488 & 2440833528564
\end{bmatrix}}\]
The matrices $A_1,A_2,A_3,A_4$ and $A_5$ are PSD and, for each graph $J$ of order six (of which there are 156, up to isomorphism), we have
\begin{equation}\label{eq:DMineq}(5/6)t_{\inj}(D,J) + (7/6)t_{\inj}(M,\overline{J})-\sum_{q=1}^5\sum_{i=1}^{t_q}\sum_{j=1}^{t_q}A_q(i,j)a_{r_q}(F^q_i,F^q_j;J)\geq 1/36.\end{equation}
The full list of calculations required to verify this for all graphs $J$ on six vertices has been included in an appendix in an ancillary file for the arXiv version of the paper: \url{https://arxiv.org/src/2306.17388/anc/offDiagonalAppendices.pdf}. 
\end{proof}

\section{Concluding Remarks}
\label{sec:concl}

We conclude the paper by proposing some open problems related to this work. It seems to us that it would be particularly interesting to determine or estimate $c(K_3,K_4)$, as this may provide a natural ``bridge'' between the classical result of Goodman~\cite{Goodman59} that $c(K_3)=1/4$ and the longstanding open problem of determining $c(K_4)$. 

\begin{prob}
Determine $c(K_3,K_4)$. 
\end{prob}

A flag algebra calculation on $7$-vertex graphs yields a bound of $c(K_3,K_4)> 0.052634$ and suggests that, if this is tight, then the optimal value of $\lambda$ is approximately $0.52103$. However, we caution that these values are from an approximate solution to a semidefinite program that was found by computer, and so they may contain rounding errors. It would also be interesting to compute the balanced Ramsey multiplicity of other pairs of small graphs; in particular, the following is open.

\begin{prob}
Determine $c(K_3,C_4)$. 
\end{prob}

A flag algebra calculation in the universe of $7$-vertex graphs suggests that $c(K_3,C_4)> 0.075159$ with the optimal $\lambda$ being approximately $0.37367$; once again, these values have been copied directly from the output of a computer program and so they should not necessarily be trusted. 

One of the great triumphs in Ramsey theory during the 20th century was the determination of the asymptotics of $r(K_3,K_k)$ as $k\to\infty$. The upper bound $r(K_3,K_k)=O(k^2/\log(k))$ was first proven by Ajtai, Koml\'os and Szemer\'edi~\cite{AjtaiKomlosSzemeredi80} and subsequently improved by Shearer~\cite{Shearer83} and the lower bound was proven by Kim~\cite{Kim95}. A new proof of Kim's~\cite{Kim95} result was obtained by Bohman~\cite{Bohman09} by analyzing the triangle-free process; the best known lower bound on $r(K_3,K_k)$ was obtained in the seminal works of Fiz Pontiveros, Griffiths and Morris~\cite{FizPontiverosGriffithsMorris20} and Bohman and Keevash~\cite{BohmanKeevash21} via a far more detailed analysis of this process. We propose a Ramsey multiplicity analogue of this problem.

\begin{prob}
Prove sharp bounds on the asymptotics of $c(K_3,K_k)$ as $k\to\infty$. 
\end{prob}

Mattheus and Verstra\"ete~\cite{MattheusVerstraete24} determined the asymptotics of $r(K_4,K_k)$ up to a factor of $O(\log^2(k))$. Apart from this breakthrough and the aforementioned results on $r(K_3,K_k)$, the growth rate of $r(K_s,K_k)$ for fixed $s$ and $k\to\infty$ is not well understood. It would also be interesting to analyze the asymptotics of the analogous balanced Ramsey multiplicity constant $c(K_s,K_k)$ for fixed $s\geq4$ as $k$ tends to infinity. 

We are also interested in the range of possible values for $c(H_1,H_2)$. While there are probably many pairs of graphs for which $c(H_1,H_2)$ is irrational, it is hard to imagine that it could be transcendental. 

\begin{conj}
There do not exist graphs $H_1$ and $H_2$ such that $c(H_1,H_2)$ is transcendental.
\end{conj}

Classical (i.e. diagonal) Ramsey multiplicity has also been studied in the multicolour setting~\cite{FoxWigderson23,Cummings+13,JaggerStovicekThomason96,KralVolecWei25,Kral+22}. It may be interesting to investigate a generalization of the balanced Ramsey multiplicity constant to $r$-tuples of graphs $(H_1,\dots,H_r)$ in $r$-edge colourings of $K_n$ for large $n$ when $r\geq3$.

%We also find it interesting to determine when $c(H_1,H_2)$ is attained by Tur\'an-type colourings, as is studied in~\cite{FoxWigderson23} in the symmetric case; a generalization of the main result of~\cite{FoxWigderson23} is obtained in a companion paper of the second author with J. Hyde and J.-B. Lee (currently in preparation). 

\begin{ack}
The authors would like to thank Joseph Hyde and Jae-baek Lee for several enlightening conversations on topics related to the subject of this paper. In particular, we thank Joseph for helpful comments on early drafts of this paper which improved the exposition and for an idea which helped us to discover the construction in Proposition~\ref{prop:D,Mupper}. The second author thanks Jan Volec for sharing some tips on how to ``round'' matrices to convert computer outputs into rigorous flag algebra proofs.
\end{ack}

\begin{ai}
In the first arXiv version of this paper, the matrices in the proof of Theorem~\ref{th:D,M} did not satisfy the required inequality \eqref{eq:DMineq} for all graphs $J$ on six vertices due to an error in our rounding procedure. We delayed submitting the paper until we could correct the proof. The corrected matrices in this version of the paper were found with assistance from ChatGPT 5.5 Pro. We provided ChatGPT with prompts that included the full floating point approximations of the matrices output by an SDP solver, a list of constraints that the ``rounded'' matrices must satisfy, and a description of the procedure that one typically follows to round the matrices when making a flag algebra proof rigorous. It then produced the matrices in the current draft. We verified ourselves that they are PSD and that they satisfy inequality \eqref{eq:DMineq} for every graph $J$ on six vertices, thereby certifying the proof. ChatGPT contributed only to these technical details of the proof and did not contribute substantive mathematical ideas. ChatGPT was also used for proofreading. The authors take full responsibility for the mathematical content and final wording of the paper.
\end{ai}

\bibliographystyle{plain}
\bibliography{rm}

\end{document}